\documentclass[12pt]{amsart}
\usepackage{amsmath,amssymb,amsthm}

\newtheorem{thm}{Theorem}[section]
  
  \newtheorem{lem}[thm]{Lemma}
  \newtheorem{prop}[thm]{Proposition}
  \newtheorem{cor}[thm]{Corollary}

\newcommand{\Hom}{\mathop{\mathrm{Hom}}\nolimits}

\newcommand{\SLF}{\mathop{\mathrm{SLF}}\nolimits}
\newcommand{\End}{\mathop{\mathrm{End}}\nolimits}
\newcommand{\trace}{\mathop{\mathrm{tr}}\nolimits}
\newcommand{\gpq}{\boldsymbol{\mathfrak{g}_{p_1, p_2}}}
\newcommand{\ch}{\operatorname{Ch}}
\newcommand{\ad}{\operatorname{Ad}}

\numberwithin{equation}{subsection}
\begin{document}
\title[A matrix realization of $\boldsymbol{\mathfrak{g}_{p,q}}$]{A matrix realization of the quantum group 
$\boldsymbol{\mathfrak{g}_{p,q}}$}
\author[Y. ARIKE]{Yusuke ARIKE}
\address{Department of Pure and Applied Mathematics, Graduate School of Information Science and Technology, Osaka University, Toyonaka, Osaka 560-0043, JAPAN}
\email{y-arike@cr.math.sci.osaka-u.ac.jp}
\begin{abstract}
In this paper we will find a matrix realizations of  the quantum group $\boldsymbol{\mathfrak{g}_{p, q}}$
defined in \cite{FGST3}.
For this purpose, we 
construct all primitive idempotents and a basis of $\gpq$.
We determine the action of elements of the basis on the indecomposable projective modules,
which give rise to a matrix realization of $\gpq$.
By using this result, we obtain a basis of the space of symmetric linear functions on 
$\boldsymbol{\mathfrak{g}_{p, q}}$
and express the symmetric linear functions obtained by
the left integral, the balancing element and the center of $\boldsymbol{\mathfrak{g}_{p, q}}$
in term of this basis. 
\end{abstract}
\maketitle

\section{Introduction}

In \cite{FGST1} and \cite{FGST2},
it is pointed out that
the triplet vertex operator algebra
$\mathcal{W}(p)$ is closely related with the restricted quantum group $\overline{U}_{q}(sl_2)$ at
the primitive $2p$-th root of the unity.
It is also found in \cite{FGST1} and 
\cite{FGST2}  find that the space of functions obtained by
applying modular transformations to characters of simple modules of $\mathcal{W}(p)$ 
is isomorphic to
the center of  $\overline{U}_{q}(sl_2)$.
It is conjectured in \cite{FGST2}  that 
the category of modules of $\mathcal{W}(p)$ and the category of finite-dimensional modules of 
$\overline{U}_{q}(sl_2)$ are equivalent  as tensor categories. 
It is proved the conjecture for $p=2$.
The equivalence of categories as abelian categories is proved 
in \cite{NT2}.
We also mention that the tensor category consisting of finite-dimensional $\overline{U}_q(sl_2)$-modules
is not braided if $p \not= 2$ (\cite{KS}).

Feigin, Gainutdinov, Semikhatov and Tipunin also 
construct the vertex operator algebra $\mathcal{W}(p_1, p_2)$ for coprime positive
integers $p_1$ and $p_2$.
They find $2p_1p_2 + \frac{1}{2} (p_1-1)(p_2-1)$ non-isomorphic simple modules and
 show that the space of functions obtained by applying modular 
transformations on 
characters of simple modules
is $\frac{1}{2} (3 p_1-1) (3 p_2-1)$-dimensional. 
They study in \cite{FGST3} the
unimodular  finite-dimensional Hopf algebra $\boldsymbol{\mathfrak{g}_{p_1, p_2}}$
which is a quotient algebra of the  tensor product of two restricted 
quantum groups $\overline{U}_{q_1}(sl_2)$ at $q_1=\exp(\sqrt{-1} p_2 \pi/p_1 )$
and $\overline{U}_{q_2}(sl_2)$ at $q_1=\exp(\sqrt{-1} p_1 \pi/p_2 )$.
In \cite{FGST4},
they show that $\boldsymbol{\mathfrak{g}_{p_1, p_2}}$ has $2 p_1 p_2$ non-isomorphic simple modules, which shows that 
the category of modules of $\mathcal{W}(p_1, p_2)$ is not equivalent to the category of finite-dimensional modules 
of $\boldsymbol{\mathfrak{g}_{p_1, p_2}}$.
They prove that the space of functions
obtained by applying modular transformations to characters of simple modules of 
$\mathcal{W}(p_1, p_2)$
and the center of 
$\boldsymbol{\mathfrak{g}_{p_1, p_2}}$, on which the group $SL_2(\mathbb{Z})$ 
naturally acts, are equivalent as representations of $SL_2(\mathbb{Z})$.

In this paper, we determine a matrix realization of $\gpq$.
More precisely, it is shown in \cite{FGST4} that 
the quantum group $\gpq$ has a decomposition into two-sided ideals
\begin{align}
\gpq = Q(p_1, p_2) \oplus Q(0, p_2) \oplus \bigoplus_{r_1=1}^{p_1-1}Q(r_1, p_2)
\oplus \bigoplus_{r_2 = 1}^{p_2-1} Q(p_1, r_2) \oplus \bigoplus_{(r_1, r_2) \in 
 I} Q(r_1, r_2), \label{eq-intro-1}
\end{align}
where $I = \{(r_1, r_2) | 1 \le r_i \le p_i-1, \ p_1r_2 + p_2r_1 \le p_1p_2\}$,
we will get a  matrix realization of a two-sided ideal
which appears in \eqref{eq-intro-1}.
In order to construct a basis, we first
determine primitive idempotents and  derive direct sum decompositions of 
two-sided ideals
into their indecomposable left ideals
as it is done in \cite{Ar} and \cite{S}.
Then, by calculating actions of two-sided ideals on projective modules, we 
can have  matrix realizations of two-sided ideals.

\begingroup
\renewcommand*\thethm{\ref{thm:5.1.1.1}}
\begin{thm}
The algebras $Q(p_1, p_2)$ and $Q(0, p_2)$ are isomorphic to the matrix algebras
$M_{p_1 p_2}(\mathbb{C})$.
\end{thm}
\endgroup

\begingroup
\renewcommand*\thethm{\ref{thm:5.2.1}}
\begin{thm}
For $1 \le r_1 \le p_1-1$, the algebra $Q(r_1, p_2)$ is isomorphic to the subalgebra of 
$M_{2p_1p_2}(\mathbb{C}) \oplus M_{2p_1p_2}(\mathbb{C})$ with the shape:
\begin{align}
\Bigl(  
\begin{bmatrix}
B^{+, \uparrow}_{r_1, p_2} & 0 & 0 & 0 \\
B^{+, \leftarrow}_{r_1, p_2} & B^{-, \uparrow}_{p_1-r_1, p_2} & 0 & 0 \\
B^{+, \rightarrow}_{r_1, p_2} & 0 & B^{-, \uparrow}_{p_1-r_1, p_2} & 0 \\
B^{+, \downarrow}_{r_1, p_2} & B^{-, \rightarrow}_{p_1-r_1, p_2} & B^{-, \leftarrow}_{p_1-r_1, p_2} 
& B^{+, \uparrow}_{r_1, p_2}
\end{bmatrix},
\begin{bmatrix}
B_{p_1-r_1, p_2}^{-, \uparrow} & 0 & 0 & 0 \\
B_{p_1-r_1, p_2}^{-, \leftarrow} & B_{r_1, p_2}^{+, \uparrow} & 0 & 0 \\
B_{p_1-r_1, p_2}^{-, \rightarrow} & 0 & B_{r_1, p_2}^{+, \uparrow} & 0 \\
B_{p_1-r_1, p_2}^{-, \downarrow} & B_{r_1, p_2}^{+, \rightarrow} & B_{r_1, p_2}^{-, \leftarrow} 
& B_{p_1-r_1, p_2}^{-, \uparrow}
\end{bmatrix}        \Bigr). \notag
\end{align}
\end{thm}
\endgroup

\begingroup
\renewcommand*\thethm{\ref{thm:5.2.2}}
\begin{thm}
For $1 \le r_1 \le p_1-1$, the algebra $Q(r_1, p_2)$ is isomorphic to the subalgebra of 
$M_{2p_1p_2}(\mathbb{C}) \oplus M_{2p_1p_2}(\mathbb{C})$:
\begin{align}
\Bigl(  
\begin{bmatrix}
B^{+, \uparrow}_{p_1, r_2} & 0 & 0 & 0 \\
B^{+, \leftarrow}_{p_1, r_2} & B^{-, \uparrow}_{p_1, p_2-r_2} & 0 & 0 \\
B^{+, \rightarrow}_{p_1, r_2} & 0 & B^{-, \uparrow}_{p_1, p_2-r_2} & 0 \\
B^{+, \downarrow}_{p_1, r_2} & B^{-, \rightarrow}_{p_1, p_2 - r_2} & 
B^{-, \leftarrow}_{p_1, p_2 - r_2} & B^{+, \uparrow}_{p_1, r_2}
\end{bmatrix},
\begin{bmatrix}
B^{-, \uparrow}_{p_1, p_2-r_2} & 0 & 0 & 0 \\
B^{-, \leftarrow}_{p_1, p_2-r_2} & B^{+, \uparrow}_{p_1, r_2} & 0 & 0 \\
B^{-, \rightarrow}_{p_1, p_2-r_2} & 0 & B^{+, \uparrow}_{p_1, r_2} & 0 \\
B^{-}_{p_1, p_2-r_2} & B^{+, \rightarrow}_{p_1, r_2} & B^{+, \leftarrow}_{p_1, r_2} & 
B^{-, \uparrow}_{p_1, p_2-r_2}
\end{bmatrix}        \Bigr). \notag
\end{align}
\end{thm}
\endgroup

\begingroup
\renewcommand*\thethm{\ref{thm:5.3.1}}
\begin{thm}
The algebra of $Q(r_1, r_2)$ with $(r_1, r_2) \in I$ is isomorphic to the subalgebra of 
$M_{4 p_1 p_2}(\mathbb{C}) \oplus M_{4 p_1 p_2}(\mathbb{C}) \oplus
M_{4 p_1 p_2}(\mathbb{C}) \oplus M_{4 p_1 p_2}(\mathbb{C})$ given by
\begin{align}
&\Bigl( \begin{bmatrix}
T_{r_1, r_2}^+ & 0 & 0 & 0 \\
L_{r_1, r_2}^+ & T_{r_1, p_2-r_2}^- & 0 & 0 \\
R_{r_1, r_2}^+ & 0 & T_{r_1, p_2-r_2}^- & 0 \\
B_{r_1, r_2}^+ & R_{r_1, p_2-r_2}^- & L_{r_1, p_2-r_2}^- & T_{r_1, r_2}^+
\end{bmatrix},
\begin{bmatrix}
T_{p_1-r_1, r_2}^- & 0 & 0 & 0 \\
L_{p_1-r_1, r_2}^- & T_{p_1-r_1, p_2-r_2}^+ & 0 & 0\\
R_{p_1-r_1, r_2}^- & 0 & T_{p_1-r_1, p_2-r_2}^+ & 0\\
B_{p_1-r_1, r_2}^- & R_{p_1-r_1, p_2-r_2}^+ & L_{p_1-r_1, p_2-r_2}^+ & T_{p_1-r_1, r_2}^-
\end{bmatrix},\notag\\
&\begin{bmatrix}
T_{r_1, p_2-r_2}^- & 0 & 0 & 0 \\
L_{r_1, p_2-r_2}^- & T_{r_1, r_2}^+ & 0 & 0\\
R_{r_1, p_2-r_2}^- & 0 & T_{r_1, r_2}^+ & 0 \\
B_{r_1, p_2-r_2}^- & R_{r_1, r_2}^+ & L_{r_1, r_2}^+ & T_{r_1, p_2-r_2}
\end{bmatrix},
\begin{bmatrix}
T_{p_1-r_1, p_2-r_2}^+ & 0 & 0 & 0 \\
L_{p_1-r_1, p_2-r_2}^+ & T_{p_1-r_1, r_2}^- & 0 & 0 \\
R_{p_1-r_1, p_2-r_2}^+ & 0 & T_{p_1-r_1, r_2}^- & 0 \\
B_{p_1-r_1, p_2-r_2}^+ & R_{p_1-r_1, r_2}^- & L_{p_1-r_1, r_2}^- & T_{p_1-r_1, p_2-r_2}^+
\end{bmatrix}
\Bigr) \notag
\end{align}
where
\begin{align}
X_{r_1, r_2}^{\alpha} = \begin{bmatrix}
X_{r_1, r_2}^{\alpha, \uparrow} & 0 & 0 & 0 \\
X_{r_1, r_2}^{\alpha, \leftarrow} & X_{p_1-r_1, r_2}^{-\alpha, \uparrow} & 0 & 0\\
X_{r_1, r_2}^{\alpha, \rightarrow} & 0 &  X_{p_1-r_1, r_2}^{-\alpha, \uparrow} & 0\\
X_{r_1, r_2}^{\alpha, \downarrow} & X_{p_1-r_1, r_2}^{-\alpha, \rightarrow} & X_{p_1-r_1, r_2}^{-\alpha, \leftarrow} & 
X_{r_1, r_2}^{\alpha, \uparrow}
\end{bmatrix}\notag
\end{align}
and $X = T, R, L, B$.
\end{thm}
\endgroup

By using these matrix realizations, we will get a basis of the space of 
symmetric  linear functions
$\SLF(\gpq)$ as the sums of traces of matrix blocks appeared in the matrix realizations.

As it is shown in \cite{R},
The space $\SLF(\gpq)$ is isomorphic to the center of $\gpq$.
This isomorphism is given by
$c \mapsto t^{-1} c \rightharpoonup \lambda$
where $\lambda$ is the left integral
and $t$ is an invertible element in $\gpq$ such that $S^{2} (x) = t x t^{-1}$ for all $x \in \gpq$.
Then we can determine the relations between the basis of symmetric linear functions and
symmetric linear functions determined by the action of the central elements and balancing element of 
$\boldsymbol{\mathfrak{g}_{p, q}}$ on integrals.

In \cite{FGST4}, the space of $q$-characters of $\gpq$
is determined. 
The most important thing is that they construct a basis of the space of $q$-characters.
Since the space of the $q$-characters is  $\{\beta \in \gpq^{\ast} | 
\beta (xy) = \beta (S^{2}(y) x), \ x, y \in \gpq\}$ (c.f. \cite{FGST4})
and the square of the antipode of $\gpq$ is an inner automorphism by the 
balancing element $g$ of $\gpq$ (see \cite{FGST4}),
any symmetric linear function is given by $\beta \leftharpoonup g$ where $\beta$
is a $q$-character. 
Under this correspondence,
any element of our basis of $\SLF(\gpq)$ is mapped to a scalar multiple of an element of
the basis of the space of $q$-characters constructed in \cite{FGST4} (see 
Appendix A).

This paper is organized as follows.
In section 2, we recall the basic definitions and properties of symmetric linear functions and integrals of Hopf algebras.
In section 3, we recall the definition of the algebra 
$\boldsymbol{\mathfrak{g}_{p_1, p_2}}$ and
its integrals and balancing element given in \cite{FGST4}.
We also introduce the list of non-isomorphic simple modules $\boldsymbol{\mathfrak{g}_{p_1, p_2}}$
given in \cite{FGST4}.
In section 4,
we construct indecomposable modules as left ideals of  $\boldsymbol{\mathfrak{g}_{p_1, p_2}}$.
In section 5, we find the primitive idempotents of  $\boldsymbol{\mathfrak{g}_{p_1, p_2}}$ and
show that the indecomposable modules constructed in section 4 give rise to  the list of 
non-isomorphic indecomposable projective modules.
In addition, we give a decomposition of $\boldsymbol{\mathfrak{g}_{p_1, p_2}}$
into subalgebras.
In section 6, we determine the matrix realization of each subalgebra of $\boldsymbol{\mathfrak{g}_{p_1, p_2}}$
as the subalgebras of direct sum of matrix algebras.
In section 7, we construct the symmetric linear functions on $\boldsymbol{\mathfrak{g}_{p_1, p_2}}$
which form a basis of the space of symmetric linear functions on $\boldsymbol{\mathfrak{g}_{p_1, p_2}}$.
We also determine the relations between the basis and the symmetric linear functions 
given by the left integral, the balancing element and the central elements. 
In Appendix, we give a correspondence between the basis of 
$q$-characters of $\gpq$ obtained in \cite{FGST4} 
and the basis of symmetric linear functions on $\gpq$

\section*{Acknowledgment}
The author is grateful to K. Nagatomo for his continuous encouragement and useful comments.

\section{Preliminaries}
In this paper we will always work over the complex number field $\mathbb{C}$.
For any vector space $V$ we denote its dual space $\Hom_{\mathbb{C}}(V, \mathbb{C})$ by $V^*$. 
\subsection{Symmetric linear functions}
Let $A$ be a finite-dimensional associative algebra.
A symmetric linear function $\varphi$ on $A$ is an element of $A^*$ which satisfies $\varphi(ab) = \varphi(ba)$
for all $a, b \in A$.
We denote the space of symmetric linear functions on $A$ by $\SLF(A)$.
If $A$ is a finite-dimensional Hopf algebra,
the space $\SLF(A)$ coincides with the space of cocommutative elements of $A^*$ (see \cite{R}).

\subsection{Integrals and the square of antipode of Hopf algebras}
Let $A$ be a finite-dimensional Hopf algebra with the coproduct $\Delta$, the counit $\varepsilon$ and the antipode $S$.
Any element of the subspaces
\begin{align}
&\mathcal{L}_A = \{ \Lambda \in A \  | \ a \Lambda = \varepsilon (a) \Lambda \ \text{for all} \ a \in A  \}, \notag\\
&\mathcal{R}_A = \{ \Lambda \in A \ | \  \Lambda a = \varepsilon (a) \Lambda \ \text{for all} \ a \in A  \}, \notag
\end{align}
is called a left integral and a right integral of $A$, respectively.
Since $A$ is finite-dimensional, the space
$\mathcal{L}_A$ (respectively $\mathcal{R}_A$) is one-dimensional (cf. \cite{Mont}).
Similarly a left (respectively, right) integral of the dual Hopf algebra $A^*$ is an element $\lambda \in A^*$ which satisfies $p\lambda = p(1)\lambda$
(respectively, $\lambda p = p(1)\lambda$) for all $p \in A^*$.
Equivalently we can see
\begin{align}
&\mathcal{L}_{A^*} = \{ \lambda \in A^* \ | \ (1 \otimes \lambda) \Delta(x) = \lambda(x) \ \text{for all} \ x \in A \}, \notag\\
&\mathcal{R}_{A^*} = \{ \lambda \in A^* \ | \ (\lambda \otimes 1) \Delta(x) = \lambda(x) \ \text{for all} \ x \in A\}. \notag
\end{align}
If $\mathcal{L}_A = \mathcal{R}_A$ the Hopf algebra $A$ is called {\em unimodular}.
\begin{prop}[\cite{R}] \label{prop:center-integ}
Let $A$ be a finite-dimensional unimodular Hopf algebra with the antipode $S$.
Suppose that $\lambda$ is a left integral of $A^*$ and that $\mu$ is a right integral of $A^*$. 
Then
\begin{enumerate}
\item
$\lambda (ab) = \lambda (bS^2(a))$,
\item
$\mu (ab) = \mu(S^2(b)a)$.
\end{enumerate}
\end{prop}

The square of the antipode is called \textit{inner} if there exists an invertible element $t$
such that $S^2(x) = txt^{-1}$ for all $x \in A$.

Denote by $\rightharpoonup$ and $\leftharpoonup$ the left and right actions of $A$ on $A^*$:
\begin{align}
a \rightharpoonup p (b) = p(ba), \ p \leftharpoonup a (b) = p(ab), \notag 
\end{align}
for $p \in A^{\ast}$ and $a, b \in A$.
\begin{prop}[\cite{R}] \label{prop:center-yaho}
Let $A$ be a finite-dimensional unimodular Hopf algebra with the antipode $S$.
Suppose that $\lambda$ is a left integral of $A^*$ and that $\mu$ is a right integral of $A^*$.
If $S^2$ is inner, the linear maps $f_{\ell}$ and $f_{r} : Z(A) \to \SLF(A)$ defined by
$f_{\ell} (c)= t^{-1}c \rightharpoonup \lambda $ and $f_r(c) = \mu \leftharpoonup tc$
are isomorphisms where $Z(A)$ denote the center of $A$.
\end{prop}

\section{The Hopf algebra $\boldsymbol{\mathfrak{g}_{p,q}}$}
\subsection{Definition}
Let $p_1$ and $p_2$ be coprime positive integers.
Set $q = \exp(\frac{\pi \sqrt{-1}}{2 p_1 p_2})$,
$q_1 = q^{2p_2} = \exp(\frac{\pi \sqrt{-1}}{p_1})$, and $q_2 = q^{2p_1} = \exp(\frac{\pi \sqrt{-1}}{p_2})$.
We use the $q$-integers and $q$-binomial coefficients 
\begin{align}
[n]_q = \frac{q^n - q^{-n}}{q - q^{-1}}, \ 
\begin{bmatrix}
m \\ n
\end{bmatrix}_q
= \frac{[m]_q !}{[n]_q! [m-n]_q!}, \
[n]_q! = [n]_q \cdots [2]_q [1]_q. \notag
\end{align}
We also define
\begin{align}
[m]_1 = [m]_{q_1^{p_2}}, \ [m]_2 = [m]_{q_2^{p_1}}, \
\begin{bmatrix}
m \\ n
\end{bmatrix}_{1}
=
\begin{bmatrix}
m \\ n
\end{bmatrix}_{q_1^{p_2}}, 
\begin{bmatrix}
m \\ n
\end{bmatrix}_{2}
=
\begin{bmatrix}
m \\ n
\end{bmatrix}_{q_2^{p_1}}. \notag
\end{align}
Note that 
\begin{align}
[p_1 - m]_1 = (-1)^{p_2 +1} [m]_1 \ \text{and} \ [p_2 - m]_2 = (-1)^{p_1 +1} [m]_2. \notag
\end{align}
$\boldsymbol{\mathfrak{g}_{p_1,p_2}}$ is a Hopf algebra over $\mathbb{C}$ generated by
$e_i, f_i$ and $K^{\pm}$ for $i = 1, 2$ with relations
\begin{align}
&K K^{-1} = K^{-1} K = 1, \notag\\
&e_i^{p_i} = f_i^{p_i} = 0, \ K^{2 p_1 p_2 } = 1, \ (i=1, 2),\notag\\
&K e_i K^{-1} = q_i^2 e_i, \ K f_i K^{- 1}= q_i^{-2} f_i, \ (i=1, 2), \notag\\
&e_1 e_2 = e_2 e_1, \ f_1 f_2 = f_2 f_1, \notag\\
&[e_1, f_1] = \frac{K^{p_2} - K^{-p_2}}{q_1^{p_2} - q_1^{-p_2}}, \ 
[e_2, f_2] = \frac{K^{p_1} - K^{-p_1}}{q_2^{p_1} - q_2^{-p_1}}, \notag
\end{align}
as an algebra.
The coproduct $\Delta$, counit $\varepsilon$, and antipode $S$ are given by
\begin{align}
& \Delta (e_1) = e_1 \otimes 1 + K^{p_2} \otimes e_1, \ \Delta (e_2) = e_2 \otimes K^{p_1} + 1 \otimes e_2, \notag\\
& \Delta (f_1) = f_1 \otimes K^{- p_2} + 1 \otimes f_1, \ \Delta (f_2) = f_2 \otimes 1 + K^{-p_1} \otimes f_2,\notag\\
&\Delta (K) = K \otimes K, \ 
 \varepsilon (e_i) = \varepsilon (f_i) = 0, \ \varepsilon(K) = K, \notag\\
& S (e_1) = - K^{-p_2} e_1, \ S (e_2) = - e_2 K^{-p_1}, \notag\\
& S (f_1) = - f_1 K^{p_2}, \ S (f_2) = - K^{p_1} f_2, \ 
S(K) = K^{-1}. \notag 
\end{align}
\begin{prop}\label{lem:1}
The $2 p_1^3 p_2^3$ elements $e_1^{m_1} e_2^{m_2} f_1^{n_1} f_2^{n_2} K^{\ell}$,
where $0 \le m_i, n_i \le p_i-1$ and $0 \le \ell \le 2p_1 p_2 -1$, form a basis of $\boldsymbol{\mathfrak{g}_{p_1, p_2}}$
as a vector space over $\mathbb{C}$.
\end{prop}
In \cite{FGST4}, it is shown that the restricted quantum groups
$\overline{U}_{q_1^{p_2}}(sl_2) = \langle e_1, f_1, K_1 \rangle$
and $\overline{U}_{q_2^{p_1}}(sl_2) = \langle e_2, f_2, K_2 \rangle$
are embedded in $\boldsymbol{\mathfrak{g}_{p_1, p_2}}$ by
\begin{align}
&e_1 \mapsto e_1, \ f_1 \mapsto f_1, \ K_1 \mapsto K^{p_2}, \notag\\
&e_2 \mapsto e_2, \ f_2 \mapsto f_2, \ K_2 \mapsto K^{p_1} \notag
\end{align}  
and that 
\begin{align}
\boldsymbol{\mathfrak{g}_{p_1, p_2}} \cong \overline{U}_{q_1^{p_2}}(sl_2)
\otimes \overline{U}_{q_2^{p_1}}(sl_2)/ (K_1^{p_2} \otimes 1 - 1 \otimes K_2^{p_2}), \notag
\end{align}
where $(K_1^{p_1} \otimes 1 - 1 \otimes K_2^{p_2})$ is the Hopf ideal generated by 
$(K_1^{p_1} \otimes 1 - 1 \otimes K_2^{p_2})$.

\subsection{Integrals and the square of the antipode}
The space of left and right integrals in $\boldsymbol{\mathfrak{g}_{p_1, p_2}}$ and the space of right integrals on $\boldsymbol{\mathfrak{g}_{p_1, p_2}}$
are determined in \cite{FGST4}.
The space spanned by
\begin{align}
e_1^{p_1-1} e_2^{p_2-1} f_1^{p_1-1} f_2^{p_2-1} \sum_{\ell = 0}^{2 p_1 p_2 -1} K^{\ell}, \notag
\end{align}
coincides with the space of left integrals of $\boldsymbol{\mathfrak{g}_{p_1, p_2}}$,
which also coincides with the space of right integrals of $\boldsymbol{\mathfrak{g}_{p_1, p_2}}$.
This shows that $\boldsymbol{\mathfrak{g}_{p_1, p_2}}$ is unimodular.

Define the linear functions on $\boldsymbol{\mathfrak{g}_{p_1, p_2}}$ by
\begin{align}
&\lambda (e_1^{m_1} e_2^{m_2} f_1^{n_1} f_2^{n_2} K^{\ell}) 
=\delta_{m_1, p_1-1} \delta_{m_2, p_2-1} \delta_{n_1, p_1-1} \delta_{n_2, p_2-1} \delta_{\ell, p_2-p_1}, \label{eq:2.2.1}\\
&\mu (e_1^{m_1} e_2^{m_2} f_1^{n_1} f_2^{n_2} K^{\ell}) 
=\delta_{m_1, p_1-1} \delta_{m_2, p_2-1} \delta_{n_1, p_1-1} \delta_{n_2, p_2-1} \delta_{\ell, p_1-p_2}, \label{eq2.2.2}
\end{align}
for $0 \le m_i \le p_i-1$ and $- (p_1p_2-1) \le \ell \le p_1 p_2$.
Using the induction, we can see that
\begin{align}
&\Delta (e_1^{m_1} e_2^{m_2} f_1^{n_1} f_2^{n_2} K^{\ell})
= \sum_{r_1=0}^{m_1} \sum_{r_2=0}^{m_2}\sum_{s_1=0}^{n_1} \sum_{s_2=0}^{n_2}
q_1^{p_2(m_1-r_1) + p_2 s_1(n_1-r_1) -2p_2 s_1 (m_1-r_1)} \notag\\
& \times q_2^{p_1r_2(m_2-r_2) + p_1s_2(n_2-s_2) -2 p_1r_2(n_2-s_2)} \begin{bmatrix}
m_1 \\ r_1
\end{bmatrix}_1
\begin{bmatrix}
m_2 \\ r_2
\end{bmatrix}_2
\begin{bmatrix}
n_1 \\ s_1
\end{bmatrix}_1
\begin{bmatrix}
n_2 \\ s_2
\end{bmatrix}_2\notag\\
&\times e_1^{r_1} e_2^{r_2} f_1^{s_1} f_2^{s_2} K^{p_2(m_1-r_1) - p_1(n_2-s_2) + \ell} 
\otimes e_1^{m_1-r_1} e_2^{m_2-r_2} f_1^{n_1-s_1} f_2^{n_2-s_2} K^{p_1r_2 - p_2s_1 + \ell}. \notag
\end{align}
\begin{prop}[\cite{FGST4}]\label{prop:integrals}
Each of the spaces of left and right integrals in $\boldsymbol{\mathfrak{g}_{p_1, p_2}}$ is spanned by 
$\lambda$ and $\mu$, respectively.
\end{prop}
It is also shown in \cite{FGST4} that the square of the antipode of $\boldsymbol{\mathfrak{g}_{p_1, p_2}}$
is inner
and that the balancing element of $\boldsymbol{\mathfrak{g}_{p_1, p_2}}$
is given as follows.
\begin{prop}[\cite{FGST4}]\label{prop:inner}
For any $x \in \boldsymbol{\mathfrak{g}_{p_1, p_2}}$, we have
$S^2 (x) = g x g^{-1}$ where $g = K^{p_1-p_2}$.
\end{prop}

\subsection{Simple modules}
The algebra $\boldsymbol{\mathfrak{g}_{p_1, p_2}}$ has $2 p_1 p_2$ non-isomorphic simple modules (see \cite{FGST4}).
Let $\{ X_{r_1, r_2}^{\alpha} \ | \ \alpha = \pm, 1 \le r_1 \le p_1, 1 \le r_2 \le p_2  \}$
be the complete list of non-isomorphic simple modules of  $\boldsymbol{\mathfrak{g}_{p_1, p_2}}$.
The simple module $X_{r_1, r_2}^{\alpha}$ has a basis formed by weight vectors $\mathsf{b}_{n_1, n_2}^{\alpha}(r_1, r_2)$,
$0 \le n_1 \le r_1-1$ and $0 \le n_2 \le r_2-1$ with the action of $\boldsymbol{\mathfrak{g}_{p_1, p_2}}$
defined by
\begin{align}
&K \mathsf{b}_{n_1, n_2}^{\alpha}(r_1, r_2) 
= \alpha q_1^{r_1 - 1 - 2 n_1} q_2^{r_2 - 1 - 2 n_2} \mathsf{b}_{n_1, n_2}^{\alpha}(r_1, r_2), \notag\\
&e_1 \mathsf{b}_{n_1, n_2}^{\alpha}(r_1, r_2) = \begin{cases}
\varphi_1^{\alpha}(n_1, r_1, r_2) \mathsf{b}_{n_1 -1 , n_2}^{\alpha}(r_1, r_2), & n_1 \neq 0, \\
0, & n_1 =0,
\end{cases} \notag\\
&e_2 \mathsf{b}_{n_1, n_2}^{\alpha}(r_1, r_2) = \begin{cases}
\varphi_2^{\alpha}(n_2, r_1, r_2) \mathsf{b}_{n_1  , n_2-1}^{\alpha}(r_1, r_2), & n_2 \neq 0, \\
0, & n_2 =0,
\end{cases}\notag\\
&f_1 \mathsf{b}_{n_1, n_2}^{\alpha}(r_1, r_2)  = \begin{cases}
\mathsf{b}_{n_1 +1, n_2}^{\alpha}(r_1, r_2), & n_1 \neq r_1-1, \\
0, & n_1 = r_1-1,
\end{cases}\notag\\
&f_2 \mathsf{b}_{n_1, n_2}^{\alpha}(r_1, r_2)  = \begin{cases}
\mathsf{b}_{n_1 , n_2 +2}^{\alpha}(r_1, r_2), & n_2 \neq r_2-1, \\
0, & n_2 = r_2-1,
\end{cases}\notag
\end{align}
where
\begin{align}
&\varphi_1^{\alpha}(n_1, r_1, r_2) 
= \alpha^{p_2} (-1)^{r_2-1} [n_1]_1[r_1-n_1]_1, \label{eq:2.1.1}\\
&\varphi_2^{\alpha}(n_2, r_1, r_2) 
= \alpha^{p_1} (-1)^{r_1-1} [n_2]_2[r_2-n_2]_2, \label{eq:2.1.2}
\end{align}
for $1 \le n_i \le r_i-1$.
We note that
\begin{align}
&\varphi_1^{-\alpha}(k_1, p_1-r_1, r_2) = \varphi_1^{\alpha}(k_1, p_1-r_1, p_2-r_2), \label{eq:2.1.3}\\
&\varphi_1^{\alpha}(n_1, r_1, r_2) = \varphi_1^{-\alpha}(n_1, r_1, p_2-r_2), \label{eq:2.1.4}
\end{align}
for $1 \le n_1 \le r_1-1$ and $1 \le k_1 \le p_1-r_1-1$,
and that
\begin{align}
&\varphi_2^{-\alpha}(k_2, r_1, p_2-r_2) = \varphi_2^{\alpha}(k_2, p_1-r_1, p_2-r_2), \label{eq:2.1.5}\\
&\varphi_2^{\alpha}(n_2, r_1, r_2) = \varphi_2^{-\alpha}(n_1, p_1-r_1, r_2), \label{eq:2.1.6}
\end{align}
for $1 \le n_2 \le r_2-1$ and $1 \le k_2 \le p_2-r_2-1$.

\section{Construction of indecomposable left ideals}
In this section we construct left ideals which are isomorphic to $P_{r_1, r_2}^{\alpha}$ whose socle is isomorphic to a simple module
$X_{r_1, r_2}^{\alpha}$ for $(r_1, r_2) \neq (p_1, p_2)$.
The structure of module $P_{r_1, r_2}^{\alpha}$ is described in \cite{FGST4}.

\subsection{Highest weight vectors in $\boldsymbol{\mathfrak{g}_{p_1, p_2}}$}
For $1 \le r_i \le p_i$ and $1 \le s_i \le r_i$,
let us set
\begin{align}
v_{r_1,r_2}^{\alpha}(s_1, s_2)
= \sum_{\ell = 0}^{2 p_1 p_2 -1} (\alpha q_1^{-(r_1-2 s_1 +1)} q_2^{-(r_2 - 2 s_2 + 1)} )^{\ell} K^{\ell}, \notag
\end{align}
and
\begin{align}
b_{n_1, n_2}^{\alpha, \downarrow} (r_1, r_2, s_1, s_2) = 
f_1^{n_1} f_2^{n_2} e_1^{p_1-1}e_2^{p_2-1} f_1^{p_1-s_1} f_2^{p_2-s_2} v_{r_1,r_2}^{\alpha}(s_1, s_2), \notag
\end{align}
for $0 \le n_i \le r_i-1$.
We then see that
\begin{align}
Kb_{n_1, n_2}^{\alpha, \downarrow} (r_1, r_2, s_1, s_2)
= \alpha q_1^{r_1-1-2n_1}q_2^{r_2-1-2n_2}b_{n_1, n_2}^{\alpha, \downarrow} (r_1, r_2, s_1, s_2)
\label{eq:3.1.0}
\end{align}
for $0 \le n_1 \le r_1-1$ and $0 \le n_2 \le r_2-1$.
We also have 
\begin{align}
&f_1 b_{n_1, n_2}^{\alpha, \downarrow} (r_1, r_2, s_1, s_2) = b_{n_1+1, n_2}^{\alpha, \downarrow} (r_1, r_2, s_1, s_2),
\label{eq:3.1.-1}\\
&f_2 b_{n_1, n_2}^{\alpha, \downarrow} (r_1, r_2, s_1, s_2) = b_{n_1, n_2+1}^{\alpha, \downarrow} (r_1, r_2, s_1, s_2)
\label{eq:3.1.-2}
\end{align}
for $0 \le n_1 \le r_1-2$ and $0 \le n_2 \le r_2-2$.
\begin{lem}[\cite{K}]\label{lem:com}
For $1 \le n_i \le p_i -1$, $i, j \in \{1, 2\}$ and $i \neq j$, the following 
 relations hold in $\mathfrak{g}_{p_1, p_2}${\rm :}
\begin{align}
[e_i, f_i^{n_i}] &= [n_i]_i f_i^{n_i-1} \frac{q_i^{-p_j (n_i-1)} K^{p_j} - q_i^{p_j (n_i-1)} K^{-p_j} }
{q_i^{p_j} - q_i^{-p_j}} \notag\\
&= [n_i]_i  \frac{q_i^{p_j (n_i-1)} K^{p_j} - q_i^{-p_j (n_i-1)} K^{-p_j} }
{q_i^{p_j} - q_i^{-p_j}} f_i^{n_i-1}, \notag\\
[e_i^{n_i}, f_i] &= [n_i]_i e_i^{n_i-1} \frac{q_i^{p_j (n_i-1)} K^{p_j} - q_i^{-p_j (n_i-1)} K^{-p_j} }
{q_i^{p_j} - q_i^{-p_j}} \notag\\
&= [n_i]_i  \frac{q_i^{-p_j (n_i-1)} K^{p_j} - q_i^{p_j (n_i-1)} K^{-p_j} }
{q_i^{p_j} - q_i^{-p_j}} e_i^{n_i-1}. \notag
\end{align}
\end{lem}
\begin{lem}\label{lem:irr}
We have
\begin{align}
&e_1 b_{n_1, n_2}^{\alpha, \downarrow} (r_1, r_2, s_1, s_2) \label{eq:3.1.1}\\
&= \begin{cases}
\varphi_1^{\alpha}(n_1, r_1, r_2)  b_{n_1 -1, n_2}^{\alpha, \downarrow} (r_1, r_2, s_1, s_2), & 1 \le n_1 \le r_1-1,\\
0, & n_1=0,
\end{cases}\notag\\
&e_2 b_{n_1, n_2}^{\alpha, \downarrow} (r_1, r_2, s_1, s_2) \label{eq:3.1.2}\\
&= \begin{cases}
\varphi_2^{\alpha}(n_2, r_1, r_2)  b_{n_1 , n_2 -1}^{\alpha, \downarrow} (r_1, r_2, s_1, s_2), & 1 \le n_2 \le r_2-1,\\
0, & n_2=0,
\end{cases}\notag
\end{align}
\end{lem}
\begin{proof}
By Lemma \ref{lem:com},
we obtain
\begin{align}
&e_1 b_{n_1, n_2}^{\alpha, \downarrow} (r_1, r_2, s_1, s_2) \notag\\
=&e_1 f_1^{n_1} f_2^{n_2} e_1^{p_1-1}e_2^{p_2-1}f_1^{p_1-s_1}f_2^{p_2-s_2}v_{r_1, r_2}^{\alpha}(s_1, s_2)\notag\\
=& [n_1]_1 \frac{q_1^{p_2(n_1-1)} K^{p_2} - q_1^{-p_2(n_1-1)}K^{-p_2}}{q_1^{p_2}-q_1^{-p_2}}
b_{n_1-1, n_2}^{\alpha, \downarrow} (r_1, r_2, s_1, s_2)\notag\\
=& \alpha ^{p_2} (-1)^{r_2-1} [n_1]_1[r_1-n_1]  b_{n_1-1, n_2}^{\alpha, \downarrow} (r_1, r_2, s_1, s_2),\notag
\end{align}
for $1 \le n_1 \le r_1-1$.
Similarly we have \eqref{eq:3.1.2}.
\end{proof}
The space spanned by $b_{n_1, n_2}^{\alpha}(r_1, r_2, s_1, s_2)$ with $0 \le n_1 \le r_1-1$
and $0 \le n_2 \le r_2-1$
is isomorphic to the simple module $X_{r_1, r_2}^{\alpha}$.
This fact will be proved in next subsection.
\begin{prop}\label{prop:factor}
\begin{enumerate}
\item
For $1 \le r_1 \le p_1-1$,
\begin{align}
&b_{0, n_2}^{\alpha, \downarrow} (r_1, r_2, s_1, s_2)\notag\\
=& f_1 f_2^{n_2} \sum_{m_1 =1 }^{p_1-r_1} \gamma_{m_1}^{\alpha} (r_1, r_2)
e_1^{p_1 - m_1} e_2^{p_2 -1} f_1^{p_1 - s_1 - m_1} f_2^{p_2 - s_2} v^{\alpha}_{r_1, r_2}(s_1, s_2), \notag
\end{align}
where 
\begin{align}
\gamma_{m_1}^{\alpha} (r_1, r_2)= \prod_{k_1 = p_1-r_1-(m_1-1)}^{p_1-r_1-1}
\varphi_1^{-\alpha}(k_1, p_1-r_1, r_2). \notag
\end{align}
\item
For $1 \le r_2 \le p_2-1$,
\begin{align}
&b_{n_1, 0}^{\alpha, \downarrow} (r_1, r_2, s_1, s_2)\notag\\
=& f_1^{n_1} f_2 \sum_{m_2 = 1}^{p_2-r_2} \delta_{m_2}^{\alpha} (r_1, r_2)
e_1^{p_1 - 1} e_2^{p_2 -m_2} f_1^{p_1 - s_1} f_2^{p_2 - s_2 - m_2} v^{\alpha}_{r_1, r_2}(s_1, s_2), \notag
\end{align}
where 
\begin{align}
\delta_{m_2}^{\alpha} (r_1, r_2)= \prod_{k_2 = p_2-r_2-(m_2-1)}^{p_2-r_2-1}
\varphi_2^{-\alpha}(k_2, r_1, p_2-r_2). \notag
\end{align}
\item
For $1 \le r_1 \le p_1-1$ and $1 \le r_2 \le p_2-1$,
\begin{align}
b_{0, 0}^{\alpha, \downarrow} (r_1, r_2, s_1, s_2)
= &f_1 f_2 \sum_{m_1 =1 }^{p_1-r_1} \sum_{m_2 = 1}^{p_2-r_2}\gamma_{m_1}^{\alpha} (r_1, r_2)
\delta_{m_2}^{\alpha} (r_1, r_2) \notag\\
& \times e_1^{p_1 - m_1} e_2^{p_2 -m_2} f_1^{p_1 - s_1 - m_1} f_2^{p_2 - s_2 - m_2} v^{\alpha}_{r_1, r_2}(s_1, s_2). \notag
\end{align}
\end{enumerate}
\end{prop}
\begin{proof}
We only prove (1).
proofs of others are similar.

By Lemma \ref{lem:com}, we have
\begin{align}
&f_1 e_1^{p_1 - m_1} e_2^{p_2 -1} f_1^{p_1 - s_1 - m_1} f_2^{p_2 - s_2} v^{\alpha}_{r_1, r_2}(s_1, s_2) \notag\\
=& e_1^{p_1 - m_1} e_2^{p_2 -1} f_1^{p_1 - s_1 - m_1+1} f_2^{p_2 - s_2} v^{\alpha}_{r_1, r_2}(s_1, s_2)\notag\\
&\qquad \qquad -  [e_1^{p_1-m_1}, f_1] e_2^{p_2 -1} f_1^{p_1 - s_1 - m_1} f_2^{p_2 - s_2} 
v^{\alpha}_{r_1, r_2}(s_1, s_2)\notag\\
=& e_1^{p_1 - m_1} e_2^{p_2 -1} f_1^{p_1 - s_1 - m_1+1} f_2^{p_2 - s_2} v^{\alpha}_{r_1, r_2}(s_1, s_2)\notag\\
&- \varphi_1^{-\alpha}(m_1, p_1-r_1, r_2)
e_1^{p_1-1-m_1}
e_2^{p_2 -1} f_1^{p_1 - s_1 - m_1} f_2^{p_2 - s_2} v^{\alpha}_{r_1, r_2}(s_1, s_2).\notag
\end{align}
Thus we have
\begin{align}
&f_1 \sum_{m_1 =1 }^{p_1-r_1} \gamma_{m_1}^{\alpha} (r_1, r_2)
e_1^{p_1 - m_1} e_2^{p_2 -1} f_1^{p_1 - s_1 - m_1} f_2^{p_2 - s_2} v^{\alpha}_{r_1, r_2}(s_1, s_2) \notag\\
=& \sum_{m_1 =1 }^{p_1-r_1} \gamma_{m_1}^{\alpha} (r_1, r_2)
e_1^{p_1 - m_1} e_2^{p_2 -1} f_1^{p_1 - s_1 - m_1+1} f_2^{p_2 - s_2} v^{\alpha}_{r_1, r_2}(s_1, s_2)\notag\\
& -\sum_{m_1 =1 }^{p_1-r_1-1} \gamma_{m_1+1}^{\alpha} (r_1, r_2)
e_1^{p_1 - m_1-1} e_2^{p_2 -1} f_1^{p_1 - s_1 - m_1} f_2^{p_2 - s_2} v^{\alpha}_{r_1, r_2}(s_1, s_2)\notag\\
&= e_1^{p_1 - 1} e_2^{p_2 -1} f_1^{p_1 - s_1} f_2^{p_2 - s_2} v^{\alpha}_{r_1, r_2}(s_1, s_2).\notag
\end{align}
\end{proof}
Set
\begin{align}
&\Phi^{\alpha} (r_1, r_2) = 2p_1p_2\prod_{i_1 = 1}^{r_1-1} \varphi_1^{\alpha}(i_1, r_1, r_2)
\prod_{i_2 = 1}^{r_2-1} \varphi_2^{\alpha}(k_2, r_1, r_2) \notag\\
&\qquad \times
\prod_{k_1 = 1}^{p_1-r_1-1} \varphi_1^{-\alpha}(k_1, p_1-r_1, r_2)
\prod_{k_2 = 1}^{p_2-r_2-1} \varphi_2^{-\alpha}(k_2, r_1, p_2-r_2), \notag\\
&\Psi_1^{\alpha}(r_1, r_2) = \sum_{i_1=1}^{r_1-1} \frac{1}{\varphi_1^{\alpha}(i_1, r_1, r_2)}
+ \sum_{j_1 = 1}^{p_1-r_1-1} \frac{1}{\varphi_1^{-\alpha}(j_1, p_1-r_1, r_2)}, \notag\\
&\Psi_2^{\alpha}(r_1, r_2) = \sum_{i_2=1}^{r_2-1} \frac{1}{\varphi_2^{\alpha}(i_2, r_1, r_2)}
+ \sum_{j_2 = 1}^{p_2-r_2-1} \frac{1}{\varphi_2^{-\alpha}(j_2, r_1, p_2-r_2)}. \notag
\end{align}
Then we can see that 
$\Phi^{\alpha} (r_1, r_2) = \Phi^{-\alpha} (p_1-r_1, r_2) = \Phi^{-\alpha} (r_1, p_2-r_2) 
= \Phi^{\alpha} (p_1-r_1, p_2-r_2)$
and that
$\Psi_i^{\alpha}(r_1, r_2) = \Psi_i^{-\alpha}(p_1 - r_1, r_2) 
= \Psi_i^{-\alpha}(r_1, p_2 -r_2)=\Psi_i^{\alpha}(p_1 - r_1, p_2-r_2)$ 
for $i = 1, 2$
 by \eqref{eq:2.1.3}-\eqref{eq:2.1.6}.
Now we set
\begin{align}
B_{n_1, n_2}^{\alpha, \downarrow}(r_1, r_2, s_1, s_2) = \frac{1}{\Phi^{\alpha}(r_1, r_2)} b_{n_1, n_2}^{\alpha, \downarrow}
(r_1, r_2, s_1, s_2).\notag
\end{align}

\subsection{Indecomposable left ideals $P_{r_1, p_2}^{\alpha}(s_1, s_2)$ and $P_{p_1, r_2}^{\alpha}(s_1, s_2)$}
Let us first fix $1 \le r_1 \le p_1-1$ and $r_2 = p_2$.
We set 
\begin{align}
&B_{k_1, n_2}^{\alpha, \leftarrow}(r_1, p_2, s_1, s_2) \label{eq:3.2.1}\\
=& \frac{e_1^{p_1-r_1-1-k_1} f_2^{n_2}\sum_{m_1=1}^{p_1-r_1} \gamma_{m_1}^{\alpha}(r_1, p_2)
e_1^{p_1-m_1} e_2^{p_2-1}f_1^{p_1-s_1-m_1}f_2^{p_2-s_2} v_{r_1, p_2}^{\alpha}(s_1, s_2)}{
\Phi^{\alpha}(r_1, p_2) \prod_{j_1=k_1+1}^{p_1-r_1-1}
\varphi_1^{-\alpha}(j_1, p_1-r_1, p_2)}, \notag
\end{align}
\begin{align}
&B_{n_1, n_2}^{\alpha, \uparrow}(r_1, p_2, s_1, s_2) \label{eq:3.2.2}\\
=&\frac{f_1^{n_1}f_2^{n_2}}{\Phi^{\alpha}(r_1, p_2)} \sum_{m_1=1}^{p_1-r_1} \gamma_{m_1}^{\alpha}(r_1, p_2) e_1^{p_1-1-m_1}
e_2^{p_2-1}f_1^{p_1-s_1-m_1}f_2^{p_2-s_2} v_{r_1, r_2}^{\alpha}(r_1, r_2)\notag\\
&\qquad \qquad \qquad - \Psi_1^{\alpha}(r_1, p_2) B_{n_1, n_2}^{\alpha}(r_1, p_2, s_1, s_2), \notag\\
&B_{k_1, n_2}^{\alpha, \rightarrow}(r_1, p_2, s_1, s_2) = f_1^{r_1+k_1}B_{0, n_2}^{\alpha, \uparrow}(r_1, p_2, s_1, s_2),
\label{eq:3.2.3}
\end{align}
for $0 \le n_1 \le r_1-1$, $0 \le k_1 \le p_1-r_1-1$ and $0 \le n_2 \le p_2-1$.
\begin{lem}\label{lem:r_1-p_2-1}
For $0 \le n_2 \le p_2-1$, the following relations hold\,{\rm :}
\begin{align}
&KB_{k_1, n_2}^{\alpha, \leftarrow}(r_1, p_2, s_1, s_2)
= -\alpha q_1^{p_1-r_1-1-2 k_1}q_2^{r_2-1-2n_2} B_{k_1, n_2}^{\alpha, \leftarrow}(r_1, p_2, s_1, s_2),\label{eq:3.2.4}\\
&KB_{n_1, n_2}^{\alpha, \uparrow}(r_1, p_2, s_1, s_2)
= \alpha q_1^{r_1-1-2 n_1}q_2^{r_2-1-2n_2} B_{n_1, n_2}^{\alpha, \uparrow}(r_1, p_2, s_1, s_2),\label{eq:3.2.5}\\
&KB_{k_1, n_2}^{\alpha, \rightarrow}(r_1, p_2, s_1, s_2)
= -\alpha q_1^{p_1-r_1-1-2 k_1}q_2^{r_2-1-2n_2} B_{k_1, n_2}^{\alpha, \rightarrow}(r_1, p_2, s_1, s_2),\label{eq:3.2.6}\\
&e_1 B_{k_1, n_2}^{\alpha, \leftarrow}(r_1, p_2, s_1, s_2) \label{eq:3.2.7}\\
= &\begin{cases}
\varphi_1^{-\alpha}(k_1, p_1-r_1, p_2) B_{k_1-1, n_2}^{\alpha, \leftarrow}(r_1, p_2, s_1, s_2), & 1 \le k_1 \le p_1-r_1-1,\\
0, & k_1=0,
\end{cases}\notag\\
&f_1 B_{k_1, n_2}^{\alpha, \leftarrow}(r_1, p_2, s_1, s_2) \label{eq:3.2.8}\\
=&\begin{cases}
B_{k_1+1, n_2}^{\alpha, \leftarrow}(r_1, p_2, s_1, s_2), & 0 \le k_1 \le p_1-r_1-2,\\
B_{0, n_2}^{\alpha, \downarrow}(r_1, p_2, s_1, s_2), & k_1=p_1-r_1-1,
\end{cases}\notag\\
&f_1 B_{n_1, n_2}^{\alpha, \downarrow}(r_1, p_2, s_1, s_2) 
=\begin{cases}
B_{n_1+1, n_2}^{\alpha, \downarrow}(r_1, p_2, s_1, s_2), & 0 \le n_1 \le r_1-2,\\
0, & n_1 = r_1-1,
\end{cases}\label{eq:3.2.9}\\
&e_1 B_{n_1, n_2}^{\alpha, \uparrow}(r_1, p_2, s_1, s_2) \label{eq:3.2.10}\\
&=\begin{cases}
\varphi_1^{\alpha}(n_1, r_1, p_2)
B_{n_1-1, n_2}^{\alpha, \uparrow}(r_1, p_2, s_1, s_2)+B_{n_1-1, n_2}^{\alpha, \downarrow}(r_1, p_2, s_1, s_2),
& 1 \le n_1 \le r_1-1,\\
B_{p_1-r_1-1, n_2}^{\alpha, \uparrow}(r_1, p_2, s_1, s_2), & n_1=0,
\end{cases}\notag\\
&f_1 B_{n_1-1, n_2}^{\alpha, \uparrow}(r_1, p_2, s_1, s_2) \label{eq:3.2.11}\\
&=\begin{cases}
B_{n_1+1, n_2}^{\alpha, \uparrow}(r_1, p_2, s_1, s_2), & 0 \le n_1 \le r_1-2,\\
B_{0, n_2}^{\alpha, \rightarrow}(r_1, p_2, s_1, s_2), & n_1=r_1-1,
\end{cases}\notag
\end{align}
\begin{align}
&e_1 B_{k_1, n_2}^{\alpha, \rightarrow}(r_1, p_2, s_1, s_2) \label{eq:3.2.12}\\
&=\begin{cases}
\varphi_1^{-\alpha}(k_1, p_1-r_1, p_2) B_{k_1-1, n_2}^{\alpha, \rightarrow}(r_1, p_2, s_1, s_2),
& 1 \le k_1 \le p_1-r_1-1,\\
B_{r_1-1, n_2}^{\alpha, \downarrow}(r_1, p_2, s_1, s_2), & k_1=0,
\end{cases}\notag\\
&f_1 B_{k_1, n_2}^{\alpha, \rightarrow}(r_1, p_2, s_1, s_2) \label{eq:3.2.13}\\
&=\begin{cases}
B_{k_1+1, n_2}^{\alpha, \rightarrow}(r_1, p_2, s_1, s_2), & 0 \le k_1 \le p_1-r_1-2,\\
0, & k_1=p_1-r_1-1.
\end{cases}\notag
\end{align}
\end{lem}
\begin{proof}
We can  see \eqref{eq:3.2.4}-\eqref{eq:3.2.6}, \eqref{eq:3.2.7}, \eqref{eq:3.2.11} and \eqref{eq:3.2.13}
by the definition.

By Lemma \ref{lem:com}, Proposition \ref{prop:factor} and \eqref{eq:3.2.1},
we obtain \eqref{eq:3.2.8}.
Then we can see $0 = f_1^{p_1} B_{0, n_2}^{\alpha, \leftarrow}(r_1, p_2, s_1, s_2) = B_{r_1, n_2}^{\alpha, \downarrow}
(r_1, p_2, s_1, s_2)$ by \eqref{eq:3.1.-1}, which proves \eqref{eq:3.2.9}.
By Lemma \ref{lem:irr} and Proposition \ref{prop:factor}, we have
$e_1 B_{0, n_2}^{\alpha, \uparrow}(r_1, p_2, s_1, s_2) = B_{p_1-r_1-1, n_2}^{\alpha, \leftarrow}(r_1, p_2, s_1, s_2)$
and then we can see that
\begin{align}
&e_1 B_{n_1, n_2}^{\alpha, \uparrow}(r_1, p_2, s_1, s_2) = (f_1^{n_1} e_1 + [e_1, f_1^{n_1}])
B_{0, n_2}^{\alpha, \uparrow}(r_1, p_2, s_1, s_2)\notag\\
&= \varphi_1^{\alpha}(n_1, r_1, p_2) B_{n_1-1, n_2}^{\alpha, \uparrow}(r_1, p_2, s_1, s_2)
+ B_{n_1-1, n_2}^{\alpha, \downarrow}(r_1, p_2, s_1, s_2) \notag
\end{align}
for $1 \le n_1 \le r_1-1$ by Lemma \ref{lem:com} and \eqref{eq:3.2.8}.
Since $\varphi_1^{\alpha}(r_1 + k_1, r_1, p_2) = \varphi_1^{-\alpha}(k_1, p_1-r_1, p_2)$,
we have \eqref{eq:3.2.12}.
\end{proof}
By Lemma \ref{lem:com} we have the following.
\begin{lem}\label{lem:r_1-p_2-2}
\begin{align}
&e_2 B_{n_1, n_2}^{\alpha, \bullet }(r_1, p_2, s_1, s_2) \label{eq:3.2.14}\\
&=\begin{cases}
\varphi_2^{\alpha}(n_2, r_1, p_2) B_{n_1, n_2-1}^{\alpha, \bullet}(r_1, p_2, s_1, s_2), & 1 \le n_2 \le p_2-1,\\
0, & n_2=0,
\end{cases}\notag\\
&f_2 B_{n_1, n_2}^{\alpha, \bullet}(r_1, p_2, s_1, s_2) \label{eq:3.2.15}\\
=&\begin{cases}
B_{n_1, n_2+1}^{\alpha, \bullet}(r_1, p_2, s_1, s_2), & 0 \le n_2 \le p_2-2,\\
0, & n_2=p_2-1,
\end{cases}\notag
\end{align}
where $0 \le n_1 \le r_1$ if $\bullet \in \{ \uparrow , \downarrow \}$
and $0 \le n_1 \le p_1-r_1-1$ if $\bullet \in \{ \leftarrow, \rightarrow \}$.
\end{lem}
Let $X_{r_1, p_2}^{\alpha}(s_1, s_2)$ be the space spanned by the set
\begin{align}
\{ B_{n_1, n_2}^{\alpha, \downarrow}(r_1, p_2, s_1, s_2) | 0 \le n_1 \le r_1-1, 0 \le n_2 \le p_2-1\}.\notag
\end{align}
Then $X_{r_1, p_2}^{\alpha}(s_1, s_2)$ is isomorphic to the simple module $X_{r_1, p_2}^{\alpha}$
by \eqref{eq:3.1.1}, \eqref{eq:3.2.9} and \eqref{eq:3.2.14}-\eqref{eq:3.2.15}.

By Lemma \ref{lem:irr}, Lemma \ref{lem:r_1-p_2-1} and Lemma \ref{lem:r_1-p_2-2},
we have the following.

\begin{prop}\label{prop:3.2.1}
Let $P_{r_1, p_2}^{\alpha}(s_1, s_2)$ be the space spanned by vectors of the form
\begin{align}
&B_{n_1, n_2}^{\alpha, \uparrow \downarrow }(r_1, p_2, s_1, s_2), 0 \le n_1 \le r_1-1, \ 0 \le n_2 \le p_2-1,\notag\\
&B_{k_1, n_2}^{\alpha,  \rightleftarrows}(r_1, p_2, s_1, s_2), 0 \le k_1 \le p_1-r_1-1, \ 0 \le n_2 \le p_2-1\notag
\end{align}
for $\alpha = \pm$, $1 \le r_1 \le p_1-1$, $1 \le s_1 \le r_1$ and $1 \le s_2 \le p_2$.
Then $P_{r_1, p_2}^{\alpha}(s_1, s_2)$ is a $2 p_1 p_2$-dimensional indecomposable 
left ideal whose socle is isomorphic to the simple module
$X_{r_1, p_2}^{\alpha}$.
In particular, for fixed $1 \le r_1 \le p_1-1$,
the left ideals $P_{r_1, p_2}^{\alpha}(s_1, s_2)$ for $1 \le s_1 \le r_1$ and $1 \le s_2 \le p_2$
are isomorphic to each other.
\end{prop}
By Proposition \ref{prop:3.2.1},
we can write $P_{r_1, p_2}^{\alpha} \cong P_{r_1, p_2}^{\alpha}(s_1, s_2)$
for $1 \le r_1 \le p_1-1$ and $\alpha = \pm$.
We denote a basis of $P_{r_1, p_2}^{\alpha}$ by
\begin{align}
&\mathsf{b}_{n_1, n_2}^{\alpha, \uparrow \downarrow} (r_1, p_2), \ 0 \le n_1 \le r_1-1, \ 0 \le n_2 \le p_2-1,\notag\\
&\mathsf{b}_{k_1, n_2}^{\alpha, \rightleftarrows}(r_1, p_2), \ 0 \le k_1 \le p_1-r_1-1, \ 0 \le n_2 \le p_2-1. \notag
\end{align}
This basis corresponds to the basis $B^{\alpha, \bullet}$ of $P_{r_1, p_2}^{\alpha}(s_1, s_2)$.

We now fix $r_1=p_1$, $1 \le r_2 \le p_2-1$ and set
\begin{align}
&B^{\alpha, \leftarrow}_{n_1, k_2}(p_1, r_2, s_1, s_2)\label{eq:3.2.16}\\
=&
\frac{f_1^{n_1}e_2^{p_2-r_2-1-k_2} \sum_{m_2=1}^{p_2-r_2}\delta_{m_2}^{\alpha}(p_1, r_2)
e_1^{p_1-1}e_2^{p_2-m_2}f_1^{p_1-s_1}f_2^{p_2-s_2-m_2}v_{p_1, r_2}^{\alpha}(s_1, s_2)}{\Phi^{\alpha}(p_1, r_2) 
\prod_{j_2 = k_2+1}^{p_2-r_2-1} \varphi_2^{-\alpha}(j_2, p_1, p_2-r_2)}, \notag\\
&B^{\alpha, \uparrow}_{n_1, n_2}(p_1, r_2, s_1, s_2) \label{eq:3.2.17}\\
&=\frac{f_1^{n_1}f_2^{n_2}}{\Phi^{\alpha}(p_1, r_2)}
\sum_{m_2=1}^{p_2-r_2}\delta_{m_2}^{\alpha}(p_1, r_2)
e_1^{p_1-1}e_2^{p_2-1-m_2}f_1^{p_1-s_1}f_2^{p_2-s_2-m_2} v_{p_1, r_2}^{\alpha}(s_1, s_2) \notag\\
&\qquad \qquad \qquad - \Psi_2^{\alpha}(p_1, r_2) B_{n_1, n_2}^{\alpha, \downarrow}(p_1, r_2, s_1, s_2), \notag\\
&B^{\alpha, \rightarrow}_{n_1, k_2}(p_1, r_2, s_1, s_2) = f_2^{r_2+k_2} B^{\alpha, \uparrow}_{n_1, 0}(p_1, r_2, s_1, s_2),
\label{eq:3.2.18}
\end{align}
for $0 \le n_1 \le p_1-1$, $0 \le n_2 \le r_2-1$ and $0 \le k_2 \le p_2-r_2-1$.
Then we have the following.
\begin{prop}\label{prop:3.2.2}
Let $P_{p_1, r_2}^{\alpha}(s_1, s_2)$ be the space spanned by vectors of the form
\begin{align}
&B_{n_1, n_2}^{\alpha, \uparrow \downarrow }(p_1, r_2, s_1, s_2), 0 \le n_1 \le p_1-1, \ 0 \le n_2 \le r_2-1,\notag\\
&B_{n_1, k_2}^{\alpha, \rightleftarrows}(p_1, r_2, s_1, s_2), 0 \le n_1 \le p_1-1, \ 0 \le k_2 \le p_2-r_2-1\notag 
\end{align}
for $\alpha = \pm$, $1 \le r_2 \le p_2-1$, $1 \le s_1 \le p_1$ and $1 \le s_2 \le r_2$.
Then $P_{p_1, r_2}^{\alpha}(s_1, s_2)$ is a $2 p_1 p_2$-dimensional
indecomposable left ideal whose socle is isomorphic to the simple module
$X_{p_1, r_2}^{\alpha}$.
In particular, for fixed $1 \le r_2 \le p_2-1$,
the modules $P_{p_1, r_2}^{\alpha}(s_1, s_2)$ for $1 \le s_1 \le p_1$ and $1 \le s_2 \le r_2$
are isomorphic to each other.
\end{prop}
By Proposition \ref{prop:3.2.2},
we can also write $P_{p_1, r_2}^{\alpha} \cong P_{p_1, r_2}^{\alpha}(s_1, s_2)$
for $\alpha = \pm$ and $1 \le r_2 \le p_2-1$.
We denote the basis of $P_{p_1, r_2}^{\alpha}$ which corresponds to
the basis $B^{\alpha, \bullet}(s_1, s_2)$ of $P_{p_1, r_2}^{\alpha}(s_1, s_2)$ is denoted by
\begin{align}
&\mathsf{b}_{n_1, n_2}^{\alpha, \uparrow \downarrow} (p_1, r_2), \ 0 \le n_1 \le p_1-1, \ 0 \le n_2 \le r_2-1,\notag\\
&\mathsf{b}_{n_1, k_2}^{\alpha, \rightleftarrows}(r_1, p_2), \ 0 \le k_1 \le p_1-1, \ 0 \le n_2 \le p_2-r_2-1, \notag
\end{align}

\subsection{Indecomposable left ideal $P_{r_1, r_2}^{\alpha}$}
Let us fix $1 \le r_1 \le p_1-1$ and $1 \le r_2 \le p_2-1$.
Set
\begin{align}
&B^{\alpha, \leftarrow}_{k_1, n_2}(r_1, r_2, s_1, s_2)
=\frac{e_1^{p_1-r_1-1-k_1} f_2^{n_2}}{\Phi^{\alpha}(r_1, r_2)
\prod_{j_1 = k_1 + 1}^{p_1 - r_1 -1}\varphi_1^{\alpha}(j_1, p_1-r_1, r_2)}
\label{eq:3.3.1}\\
&\qquad \qquad \qquad \qquad  \times \sum_{m_1=1}^{p_1-r_1} \gamma_{m_1}^{\alpha} (r_1, r_2)
e_1^{p_1-m_1}e_2^{p_2-1}f_1^{p_1-s_1-m_1}f_2^{p_2-s_2}  v_{r_1, r_2}^{\alpha}(s_1, s_2), \notag\\
&B^{\alpha, \uparrow}_{n_1, n_2}(r_1, r_2, s_1, s_2) \label{eq:3.3.2}\\
=& \frac{f_1^{n_1}f_2^{n_2}}{\Phi^{\alpha}(r_1, r_2)} \{
\sum_{m_1=1}^{p_1-r_1} \gamma_{m_1}^{\alpha} (r_1, r_2) 
e_1^{p_1-1-m_1}e_2^{p_2-1}f_1^{p_1-s_1-m_1}f_2^{p_2-s_2}  v_{r_1, r_2}^{\alpha}(s_1, s_2) \} \notag\\
&\qquad \qquad \qquad  - \Psi_1^{\alpha}(r_1, r_2) B_{n_1, n_2}^{\alpha, \downarrow}(r_1, r_2, s_1, s_2) 
, \notag\\
&B^{\alpha, \rightarrow}_{k_1, n_2}(r_1, r_2, s_1, s_2) = f_1^{r_1+k_1} 
B^{\alpha, \uparrow}_{0, n_2}(r_1, r_2, s_1, s_2),\label{eq:3.3.3}\\
&L^{\alpha, \downarrow}_{n_1, k_2}(r_1, r_2, s_1, s_2) \label{eq:3.3.4}\\
=& \frac{f_1^{n_1} e_2^{p_2-r_2-1-k_2}}{\Phi^{\alpha}(r_1, r_2)
\prod_{j_2= k_2+1}^{p_2-r_2-1}\varphi_2^{-\alpha}(j_2, r_1, p_2-r_2) } \notag\\
&\times \sum_{m_2=1}^{p_2-r_2}\delta_{m_2}^{\alpha}(r_1, r_2)
e_1^{p_1-1}e_2^{p_2-m_2}f_1^{p_1-s_1}f_2^{p_2-s_2-m_2} v_{r_1, r_2}^{\alpha}(s_1, s_2), \notag\\
&L^{\alpha, \leftarrow}_{k_1, k_2}(r_1, r_2, s_1, s_2) \label{eq:3.3.5}\\
=& \frac{e_1^{p_1-r_1-1-k_1} e_2^{p_2-r_2-1-k_2}}{\Phi^{\alpha}(r_1, r_2)
\prod_{j_1= k_1+1}^{p_1-r_1-1}\prod_{j_2= k_2+1}^{p_2-r_2-1}
\varphi_1^{-\alpha}(j_1, p_1-r_1, r_2) \varphi_2^{-\alpha}(j_2, r_1, p_2-r_2)} \notag\\
&\times \sum_{m_1=1}^{p_1-r_1}\sum_{m_2=1}^{p_2-r_2}
\gamma_{m_1}^{\alpha}(r_1, r_2) \delta_{m_2}^{\alpha}(r_1, r_2) 
e_1^{p_1-m_1}e_2^{p_2-m_2}f_1^{p_1-s_1-m_1}f_2^{p_2-s_2-m_2}  v_{r_1, r_2}^{\alpha}(s_1, s_2), \notag\\
&L^{\alpha, \uparrow}_{n_1, k_2}(r_1, r_2, s_1, s_2) \label{eq:3.3.6}\\
&= \frac{f_1^{n_1} e_2^{p_2-r_2-1-k_2}}{\Phi^{\alpha}(r_1, r_2)
\prod_{j_2= k_2+1}^{p_2-r_2-1}
\varphi_2{-\alpha}(j_2, r_1, p_2-r_2)} \notag\\
&\times \{ \sum_{m_1=1}^{p_1-r_1}\sum_{m_2=1}^{p_2-r_2}\gamma_{m_1}^{\alpha}(r_1, r_2) \delta_{m_2}^{\alpha}(r_1, r_2) 
e_1^{p_1-1-m_1}e_2^{p_2-m_2}f_1^{p_1-s_1-m_1}f_2^{p_2-s_2-m_2} v_{r_1, r_2}^{\alpha}(s_1, s_2) \}\notag\\
&\qquad \qquad \qquad - \Psi_1^{\alpha}(r_1, r_2) L_{n_1, k_2}^{\alpha, \downarrow} (r_1, r_2, s_1, s_2) , \notag\\
&L^{\alpha, \rightarrow}_{k_1, k_2}(r_1, r_2, s_1, s_2) = 
f_1^{r_1+k_1} L^{\alpha, \uparrow}_{0, k_2}(r_1, r_2, s_1, s_2), \label{eq:3.3.7}
\end{align}
\begin{align}
&T_{n_1, n_2}^{\alpha, \downarrow}(r_1, r_2, s_1, s_2) \label{eq:3.3.8}\\
=& \frac{f_1^{n_1}f_2^{n_2}}{\Phi^{\alpha}(r_1, r_2)} \{\sum_{m_2=1}^{p_2-r_2}\delta_{m_2}^{\alpha}(r_1 r_2)e^{p_1-1}e_2^{p_2-1-m_2}f_1^{p_1-s_1}f_2^{p_2-s_2-m_2}v_{r_1, r_2}^{\alpha}(s_1, s_2) \}\notag\\
&\qquad \qquad \qquad \qquad - \Psi_2^{\alpha}(r_1, r_2) B_{n_1, n_2}^{\alpha, \downarrow}(r_1, r_2, s_1, s_2) , \notag\\
&T_{k_1, n_2}^{\alpha, \leftarrow}(r_1, r_2, s_1, s_2)
=\frac{e_1^{p_1-r_1-1-k_1}f_2^{n_2}}{\Phi^{\alpha}(r_1, r_2)
\prod_{j_1=k_1+1}^{p_1-r_1-1}\varphi_1^{-\alpha}(j_1, p_1-r_1, r_2)}\label{eq:3.3.9}\\ 
&\times \{ \sum_{m_1=1}^{p_1-r_1}\sum_{m_2=1}^{p_2-r_2}
\gamma_{m_1}^{\alpha}(r_1, r_2)\delta_{m_2}^{\alpha}(r_1, r_2)
e_1^{p_1-m_1}e_2^{p_2-1-m_2}f_1^{p_1-s_1-m_1}f_2^{p_2-s_2-m_2}v_{r_1, r_2}^{\alpha}(r_1, r_2) \}\notag\\
&\qquad \qquad - \Psi_2^{\alpha}(r_1, r_2) B_{k_1, n_2}^{\alpha, \leftarrow}(r_1, r_2, s_1, s_2) ,\notag\\
&T_{n_1, n_2}^{\alpha, \uparrow}(r_1, r_2, s_1, s_2)
= \frac{f_1^{n_1}f_2^{n_2}}{\Phi^{\alpha}(r_1, r_2)}\label{eq:3.3.10}\\
&\times \Bigl\{
\sum_{m_1=1}^{p_1-r_1}\sum_{m_2=1}^{p_2-r_2}\gamma_{m_1}^{\alpha}(r_1, r_2)\delta_{m_2}^{\alpha}(r_1, r_2)
e_1^{p_1-1-m_1}e_2^{p_2-1-m_2}f_1^{p_1-s_1-m_1}f_2^{p_2-s_2-m_2}v_{r_1, r_2}^{\alpha}(s_1, s_2)\notag\\
&-\Psi_1^{\alpha}(r_1, r_2) \sum_{m_2=1}^{p_2-r_2}\delta_{m_2}^{\alpha}(r_1 r_2)e^{p_1-1}e_2^{p_2-1-m_2}f_1^{p_1-s_1}f_2^{p_2-s_2-m_2}v_{r_1, r_2}^{\alpha}(s_1, s_2)\notag\\
& - \Psi_2^{\alpha}(r_1, r_2)
\sum_{m_1=1}^{p_1-r_1} \gamma_{m_1}^{\alpha} (r_1, r_2) 
e_1^{p_1-1-m_1}e_2^{p_2-1}f_1^{p_1-s_1-m_1}f_2^{p_2-s_2}  v_{r_1, r_2}^{\alpha}(s_1, s_2)  \notag\\
&+ \Psi_1^{\alpha}(r_1, r_2) \Psi_2^{\alpha}(r_1, r_2)
e_1^{p_1-1}e_2^{p_2-1}f_1^{p_1-s_1} f_2^{p_2-s_2}v_{r_1, r_1}^{\alpha}(s_1, s_2) \Bigr\}, \notag\\
&T_{k_1, r_2}^{\alpha, \rightarrow}(r_1, r_2, s_1, s_2)
= f_1^{r_1+k_1}T_{0, n_2}^{\alpha, \uparrow}(r_1, r_2, s_1, s_2), \label{eq:3.3.11}\\
&R_{n_1, k_2}^{\alpha, \downarrow}(r_1, r_2, s_1, s_2) 
= f_2^{r_2+k_2}T_{n_1, 0}^{\alpha, \downarrow}(r_1, r_2, s_1, s_2), \label{eq:3.3.12}\\
&R_{k_1, k_2}^{\alpha, \leftarrow}(r_1, r_2, s_1, s_2) 
= f_2^{r_2+k_2}T_{k_1, 0}^{\alpha, \leftarrow}(r_1, r_2, s_1, s_2), \label{eq:3.3.13}\\
&R_{n_1, k_2}^{\alpha, \uparrow}(r_1, r_2, s_1, s_2) 
= f_2^{r_2+k_2}T_{n_1, 0}^{\alpha, \uparrow}(r_1, r_2, s_1, s_2), \label{eq:3.3.14}\\
&R_{k_1, k_2}^{\alpha, \rightarrow}(r_1, r_2, s_1, s_2) 
= f_2^{r_2+k_2}T_{k_1, 0}^{\alpha, \rightarrow}(r_1, r_2, s_1, s_2), \label{eq:3.3.15}
\end{align}
for $0 \le n_i \le r_i-1$ and $0 \le k_i \le p_i-r_i-1$.
By Lemma \ref{lem:com} and Proposition \ref{prop:factor}, we can see
\begin{align}
&e_1 B_{k_1, n_2}^{\alpha, \leftarrow}(r_1, r_2, s_1, s_2)\label{eq:3.3.16}\\
=&\begin{cases}
\varphi_1^{-\alpha}(k_1, p_1-r_1, r_2) B_{k_1-1, n_2}^{\alpha, \leftarrow}(r_1, r_2, s_1, s_2), & 1 \le k_1 \le p_1-r_1-1\\
0, & k_1=0,
\end{cases}\notag
\end{align}
\begin{align}
&f_1 B_{k_1, n_2}^{\alpha, \leftarrow}(r_1, r_2, s_1, s_2) \label{eq:3.3.17}\\
=& \begin{cases}
B_{k_1+1, n_2}^{\alpha, \leftarrow}(r_1, r_2, s_1, s_2), & 0 \le k_1 \le p_1-r_1-2,\\
B_{0, n_2}^{\alpha, \downarrow}(r_1, r_2, s_1, s_2), & k_1=p_1-r_1-1.
\end{cases}\notag
\end{align}
Thus we have $B_{r_1, n_2}^{\alpha, \downarrow}(r_1, r_2, s_1, s_2) 
= f_1^{p_1} B_{0, n_2}^{\alpha, \leftarrow}(r_1, r_2, s_1, s_2) =0$
for $0 \le n_2 \le r_2-1$, which shows
\begin{align}
f_1 B_{n_1, n_2}^{\alpha, \downarrow}(r_1, r_2, s_1, s_2) = \begin{cases}
B_{n_1, n_2}^{\alpha, \downarrow}(r_1, r_2, s_1, s_2), & 0 \le n_1 \le r_1-2, \\
0, & n_1=r_1-1.
\end{cases}\label{eq:3.3.18}
\end{align}
By Lemma \ref{lem:com}, Proposition \ref{prop:factor} and \eqref{eq:3.3.18},
we have
\begin{align}
&e_1 B_{n_1, n_2}^{\alpha, \uparrow}(r_1, r_2, s_1, s_2) \label{eq:3.3.19}\\
=& \begin{cases}
\varphi_1^{\alpha}(n_1, r_1, r_2) B_{n_1-1, n_2}^{\alpha, \uparrow}(r_1, r_2, s_1, s_2)
+B_{n_1-1, n_2}^{\alpha, \downarrow}(r_1, r_2, s_1, s_2), & 1 \le n_1 \le r_1-1,\\
B_{p_1-r_1-1, n_2}^{\alpha, \leftarrow}(r_1, r_2, s_1, s_2), & n_1=0,
\end{cases}\notag\\
&f_1 B_{n_1, n_2}^{\alpha, \uparrow}(r_1, r_2, s_1, s_2) \label{eq:3.3.20}\\
=& \begin{cases}
B_{n_1+1, n_2}^{\alpha, \uparrow}(r_1, r_2, s_1, s_2), & 0 \le n_1 \le r_1-2,\\
B_{0, n_2}^{\alpha, \rightarrow}(r_1, r_2, s_1, s_2), & n_1=r_1-1,
\end{cases}\notag\\
&e_1 B_{k_1, n_2}^{\alpha, \rightarrow}(r_1, r_2, s_1, s_2)\label{eq:3.3.21}\\
=&\begin{cases}
\varphi_1^{-\alpha}(k_1, p_1-r_1, r_2)B_{k_1-1, n_2}^{\alpha, \rightarrow}
(r_1, r_2, s_1, s_2), & 1 \le k_1 \le p_1-r_1-1\\
B_{r_1-1, n_2}^{\alpha, \downarrow}(r_1, r_2, s_1, s_2), & k_1=0,
\end{cases}\notag\\
&f_1 B_{k_1, n_2}^{\alpha, \rightarrow}(r_1, r_2, s_1, s_2)\label{eq:3.3.22}\\
=& \begin{cases}
B_{k_1+1, n_2}^{\alpha, \rightarrow}(r_1, r_2, s_1, s_2), & 0 \le k_1 \le p_1-r_1-2,\\
0, & k_1=p_1-r_1-1,
\end{cases}\notag
\end{align}
and we also have
\begin{align}
&e_2 B_{n_1, n_2}^{\alpha, \bullet}(r_1, r_2, s_1, s_2) \label{eq:3.3.23}\\
=& \begin{cases}
\varphi_2^{\alpha}(n_2, r_1, r_2) B_{n_1, n_2-1}^{\alpha, \bullet}(r_1, r_2, s_1, s_2), & 1 \le n_2 \le r_2-1,\\
0, & n_2=0,
\end{cases}\notag
\end{align}
where $0 \le n_1 \le r_1-1$ if $\bullet \in \{\uparrow, \downarrow\}$
and $0 \le n_1 \le p_1 -r_1-1$ if $\bullet \in \{\leftarrow, \rightarrow\}$.
By Lemma \ref{lem:com} and \eqref{eq:3.3.4}-\eqref{eq:3.3.7}, we obtain
\begin{align}
&e_1 L_{n_1, k_2}^{\alpha, \downarrow}(r_1, r_2, s_1, s_2) \label{eq:3.3.24}\\
=& \begin{cases}
\varphi_1^{\alpha}(n_1, r_1, r_2) L_{n_1-1, k_2}^{\alpha, \downarrow}(r_1, r_2, s_1, s_2)
, & 1 \le n_1 \le r_1-1,\\
0, & n_1=0,
\end{cases}\notag
\end{align}
\begin{align}
&e_1 L_{k_1, k_2}^{\alpha, \leftarrow}(r_1, r_2, s_1, s_2) \label{eq:3.3.25}\\
=& \begin{cases}
\varphi_{1}^{\alpha}(k_1, r_1, p_2-r_2) L_{k_1-1, k_2}^{\alpha, \leftarrow}(r_1, r_2, s_1, s_2), & 1 \le k_1 \le p_1-r_1-1,\\
0, & k_1=0,
\end{cases}\notag\\
&f_1 L_{k_1, k_2}^{\alpha, \leftarrow}(r_1, r_2, s_1, s_2)\label{eq:3.3.26}\\
=&\begin{cases}
L_{k_1+1, k_2}^{\alpha, \leftarrow}(r_1, r_2, s_1, s_2), & 0 \le k_1 \le p_1-r_1-2,\\
L_{0, k_2}^{\alpha, \downarrow}(r_1, r_2, s_1, s_2), & k_1=p_1-r_1-1,
\end{cases}\notag
\end{align}
where the last relation follows from Proposition \ref{prop:factor}.
By \eqref{eq:3.3.4} and \eqref{eq:3.3.26}, we see $L_{r_1, k_2}^{\alpha, \downarrow}(r_1, r_2, s_1, s_2)
=f_1^{p_1}L_{0, k_2}^{\alpha, \leftarrow}(r_1, r_2, s_1, s_2) = 0$.
Hence we have
\begin{align}
&f_1 L_{n_1, k_2}^{\alpha, \downarrow}(r_1, r_2, s_1, s_2) \label{eq:3.3.27}\\
=& \begin{cases}
L_{n_1+1, k_2}^{\alpha, \downarrow}(r_1, r_2, s_1, s_2), & 0 \le n_1 \le r_1-2,\\
0, & n_1=r_1-1.
\end{cases}\notag
\end{align}
We can see that
\begin{align}
&e_1 L_{n_1, k_2}^{\alpha, \uparrow}(r_1, r_2, s_1, s_2)\label{eq:3.3.28}\\
 =& \begin{cases}
\varphi_1^{\alpha}(n_1, r_1, r_2) L_{n_1-1, k_2}^{\alpha, \uparrow}(r_1, r_2, s_1, s_2)
+L_{n_1-1, k_2}^{\alpha, \downarrow}(r_1, r_2, s_1, s_2), & 1 \le n_1 \le r_1-1, \\
L_{p_1-r_1-1, k_2}^{\alpha, \leftarrow}(r_1, r_2, s_1, s_2), & k_1=0,
\end{cases}\notag\\
&f_1 L_{n_1, k_2}^{\alpha, \uparrow}(r_1, r_2, s_1, s_2)=\begin{cases}
L_{n_1+1, k_2}^{\alpha, \uparrow}(r_1, r_2, s_1, s_2), & 0 \le n_1 \le r_1-2,\\
L_{0, k_2}^{\alpha, \rightarrow}(r_1, r_2, s_1, s_2), & n_1=p_1-r_1-1.
\end{cases}\label{eq:3.3.29}
\end{align}
by Lemma \ref{lem:com} and we also have
\begin{align}
&e_1 L_{k_1, k_2}^{\alpha, \rightarrow}(r_1, r_2, s_1, s_2) \label{eq:3.3.30}\\
=& \begin{cases}
\varphi_1^{-\alpha}(k_1, p_1-r_1, r_2 ) L_{k_1-1, k_2}^{\alpha, \uparrow}(r_1, r_2, s_1, s_2), & 1 \le k_1 \le p_1-r_1-1,\\
L_{0, k_2}^{\alpha, \downarrow}(r_1, r_2, s_1, s_2), & k_1=0,
\end{cases}\notag\\
&f_1 L_{k_1, k_2}^{\alpha, \rightarrow}(r_1, r_2, s_1, s_2) \label{eq:3.3.31}\\
=& \begin{cases}
L_{k_1+1, k_2}^{\alpha, \rightarrow}(r_1, r_2, s_1, s_2), & 0 \le k_1 \le p_1-r_1-2,\\
0, & k_1=p_1-r_1-1,
\end{cases}\notag
\end{align}
by Lemma \ref{lem:com} and Proposition \ref{prop:factor}, in addition,
the last relation follows from \eqref{eq:3.3.7} and $f_1^{p_1}=0$.
Moreover, we have
\begin{align}
&e_2 L_{k_1, k_2}^{\alpha, \bullet}(r_1, r_2, s_1, s_2)\label{eq:3.3.32}\\
 =& \begin{cases}
\varphi_2^{-\alpha}(r_1, p_2-r_2) L_{k_1, k_2-1}^{\alpha, \bullet}(r_1, r_2, s_1, s_2), & 1 \le k_2 \le p_2-r_2-1,\\
0, & k_2=0,
\end{cases}\notag\\
&f_2 L_{k_1, k_2}^{\alpha, \bullet}(r_1, r_2, s_1, s_2) \label{eq:3.3.33}\\
=& \begin{cases}
L_{k_1, k_2+1}^{\alpha, \bullet}(r_1, r_2, s_1, s_2), & 0 \le k_2 \le p_2-r_2-2,\\
B_{k_1, 0}^{\alpha, \bullet}(r_1, r_2, s_1, s_2), & k_2=p_2-r_2-1,
\end{cases}\notag
\end{align}
where $0 \le k_1 \le r_1-1$ if $\bullet \in \{\uparrow, \downarrow\}$
and $0 \le k_1 \le p_1-r_1-1$ if $\bullet \in \{\leftarrow, \rightarrow\}$.
By \eqref{eq:3.3.33}, we can see $b_{k_1, r_2}^{\alpha, \bullet}(r_1, r_2, s_1, s_2)
= f_2^{p_2} l_{k_1, 0}^{\alpha, \bullet}(r_1, r_2, s_1, s_2)=0$.
Therefore we have
\begin{align}
f_2 B_{n_1, n_2}^{\alpha, \bullet}(r_1, r_2, s_1, s_2) = \begin{cases}
B_{n_1, n_2+1}^{\alpha, \bullet}(r_1, r_2, s_1, s_2), & 0 \le n_2 \le r_1-2,\\
0, & n_2=r_2-1,
\end{cases}\label{eq:3.3.34}
\end{align}
where $0 \le n_1 \le r_1-1$ if $\bullet \in \{\uparrow, \downarrow\}$
and $0 \le n_1 \le p_1-r_1-1$ if $\bullet \in \{\leftarrow, \rightarrow\}$.
The following relations are obtained by similar argument used to show \eqref{eq:3.3.16}-\eqref{eq:3.3.22}:
\begin{align}
&T_{n_1, n_2}^{\alpha, \downarrow}(r_1, r_2, s_1, s_2)\label{eq:3.3.35}\\
=&\begin{cases}
\varphi_1^{\alpha}(n_1, r_1, r_2) T_{n_1-1, n_2}^{\alpha, \downarrow}(r_1, r_2, s_1, s_2), & 1 \le n_1 \le r_1-1,\\
0, & n_1=0,
\end{cases}\notag\\
&f_1 T_{n_1, n_2}^{\alpha, \downarrow}(r_1, r_2, s_1, s_2)\label{eq:3.3.36}\\
=&\begin{cases}
T_{n_1+1, n_2}^{\alpha, \downarrow}(r_1, r_2, s_1, s_2), & 0 \le n_1 \le r_1-2, \\
0, & n_1=r_1-1,
\end{cases}\notag
\end{align}
\begin{align}
&e_1 T_{k_1, n_2}^{\alpha, \leftarrow}(r_1, r_2, s_1, s_2)\label{eq:3.3.37}\\
=&\begin{cases}
\varphi_{1}^{-\alpha}(k_1, p_1-r_1, r_2) T_{k_1-1, n_2}^{\alpha, \downarrow}(r_1, r_2, s_1, s_2), & 1 \le k_1 \le p_1-r_1-1,\\
0, & k_1=0,
\end{cases}\notag\\
&f_1 T_{k_1, n_2}^{\alpha, \leftarrow}(r_1, r_2, s_1, s_2)\label{eq:3.3.38}\\
=&\begin{cases}
T_{k_1+1, n_2}^{\alpha, \leftarrow}(r_1, r_2, s_1, s_2), & 0 \le k_1 \le p_1-r_1-2,\\
T_{0, n_2}^{\alpha, \downarrow}(r_1, r_2, s_1, s_2), & k_1=p_1-r_1-1,
\end{cases}\notag
\end{align}
\begin{align}
&e_1 T_{n_1, n_2}^{\alpha, \uparrow}(r_1, r_2, s_1, s_2)\label{eq:3.3.39}\\
 =& \begin{cases}
\varphi_1^{\alpha}(n_1, r_1, r_2) T_{n_1-1, n_2}^{\alpha, \uparrow}(r_1, r_2, s_1, s_2)
+T_{n_1-1, n_2}^{\alpha, \downarrow}(r_1, r_2, s_1, s_2), & 1 \le n_1 \le r_1-1,\\
T_{p_1-r_1-1, n_2}^{\alpha, \leftarrow}(r_1, r_2, s_1, s_2), & n_1=0,
\end{cases}\notag\\
&f_1 T_{n_1, n_2}^{\alpha, \uparrow}(r_1, r_2, s_1, s_2)\label{eq:3.3.40}\\
=&\begin{cases}
T_{n_1+1, n_2}^{\alpha, \uparrow}(r_1, r_2, s_1, s_2), & 0 \le n_1 \le r_1-2,\\
T_{0, n_2}^{\alpha, \rightarrow}(r_1, r_2, s_1, s_2), & n_1=r_1-1,
\end{cases}\notag\\
&e_1 T_{k_1, n_2}^{\alpha, \rightarrow}(r_1, r_2, s_1, s_2)\label{eq:3.3.41}\\
=&\begin{cases}
\varphi_1^{-\alpha}(k_1, p_1-r_1, r_2) T_{k_1-1, n_2}^{\alpha, \rightarrow}(r_1, r_2, s_1, s_2), & 1 \le k_1 \le p_1-r_1-1,\\
T_{r_1-1, n_2}^{\alpha, \downarrow}(s_1, s_2), & k_1=0,
\end{cases}\notag\\
&f_1 T_{k_1, n_2}^{\alpha, \rightarrow}(r_1, r_2, s_1, s_2)\label{eq:3.3.42}\\
 =& \begin{cases}
T_{k_1+1, n_2}^{\alpha, \rightarrow}(r_1, r_2, s_1, s_2), & 0 \le k_1 \le p_1-r_1-2,\\
0, & k_1=p_1-r_1-1.
\end{cases}\notag
\end{align}
By \eqref{eq:3.3.8}-\eqref{eq:3.3.15}, Lemma \ref{lem:com} and Proposition \ref{prop:factor},
we obtain
\begin{align}
&e_2 T_{n_1, n_2}^{\alpha, \bullet}(s_1, s_2) \label{eq:3.3.43}\\
&=\begin{cases}
\varphi_{2}^{\alpha}(n_2, r_1, r_2) T_{n_1, n_2-1}^{\alpha, \bullet}(r_1, r_2, s_1, s_2)
+B_{n_1-1, n_2}^{\alpha, \bullet}(r_1, r_2, s_1, s_2), & 1 \le n_2 \le r_2-1,\\
L_{n_1, p_2-r_2-1}^{\alpha, \bullet}(r_1, r_2, s_1, s_2), & n_2=0,
\end{cases}\notag\\
&f_2 T_{n_1, n_2}^{\alpha, \bullet}(r_1, r_2, s_1, s_2) 
=\begin{cases}
T_{n_1, n_2+1}^{\alpha, \bullet}(r_1, r_2, s_1, s_2) , & 0 \le n_2 \le r_2-2,\\
R_{n_1, 0}^{\alpha, \bullet}(r_1, r_2, s_1, s_2), & n_2=r_2-1,
\end{cases}\label{eq:3.3.44}
\end{align}
where $0 \le n_1 \le r_1-1$ if $\bullet \in \{\uparrow, \downarrow\}$
and $0 \le n_1 \le p_1-r_1-1$ if $\bullet \in \{\leftarrow, \rightarrow\}$.
By \eqref{eq:3.3.12}-\eqref{eq:3.3.15} and \eqref{eq:3.3.35}-\eqref{eq:3.3.42},
we can see that
\begin{align}
&e_1 R_{n_1, k_2}^{\alpha, \uparrow}(r_1, r_2, s_1, s_2) \label{eq:3.3.45}\\
=&\begin{cases}
\varphi_1^{\alpha}(n_1, r_1, r_2) R_{n_1-1, k_2}^{\alpha, \uparrow}(r_1, r_2, s_1, s_2) 
+R_{n_1-1, k_2}^{\alpha, \downarrow}(r_1, r_2, s_1, s_2), & 1 \le n_1 \le r_1-1, \\
R_{p_1-r_1-1, k_2}^{\alpha, \leftarrow}(r_1, r_2, s_1, s_2), & n_1=0, 
\end{cases}\notag\\
&f_1 R_{n_1, k_2}^{\alpha, \uparrow}(r_1, r_2, s_1, s_2) \label{eq:3.3.46}\\
=&\begin{cases}
R_{n_1+1, k_2}^{\alpha, \uparrow}(r_1, r_2, s_1, s_2), & 0 \le n_1 \le r_1-2, \\
R_{0, k_2}^{\alpha, \rightarrow}(r_1, r_2, s_1, s_2), & n_1=r_1-1,  
\end{cases}\notag
\end{align}
\begin{align}
&e_1 R_{k_1, k_2}^{\alpha, \leftarrow}(r_1, r_2, s_1, s_2)\label{eq:3.3.47}\\
 =& \begin{cases}
\varphi_1^{-\alpha}(k_1, p_1-r_1, r_2) R_{k_1-1, k_2}^{\alpha, \leftarrow}(r_1, r_2, s_1, s_2), & 1 \le k_1 \le p_1-r_1-1,\\
0, & k_1=0, 
\end{cases}\notag\\
&f_1 R_{k_1, k_2}^{\alpha, \leftarrow}(r_1, r_2, s_1, s_2) \label{eq:3.3.48}\\
=& \begin{cases}
R_{k_1+1, k_2}^{\alpha, \leftarrow}(r_1, r_2, s_1, s_2), & 0 \le k_1 \le p_1-r_1-2,\\
R_{0, k_2}^{\alpha, \downarrow}(r_1, r_2, s_1, s_2), & k_1=p_1-r_1-1, 
\end{cases}\notag\\
&e_1 R_{k_1, k_2}^{\alpha, \rightarrow}(r_1, r_2, s_1, s_2) \label{eq:3.3.49}\\
=& \begin{cases}
\varphi_1^{-\alpha}(k_1, p_1-r_1, r_2) R_{k_1-1, k_2}^{\alpha, \rightarrow}(r_1, r_2, s_1, s_2), & 1 \le k_1 \le p_1-r_1-1,\\
R_{r_1-1, k_2}^{\alpha, \downarrow}(r_1, r_2, s_1, s_2), & k_1=0,
\end{cases} \notag\\
&f_1 R_{k_1, k_2}^{\alpha, \rightarrow}(r_1, r_2, s_1, s_2)\label{eq:3.3.50}\\
 =& \begin{cases}
R_{k_1+1, k_2}^{\alpha, \rightarrow}(r_1, r_2, s_1, s_2), & 0 \le k_1 \le p_1-r_1-2,\\
0, & k_1=p_1-r_1-1,
\end{cases}\notag\\
&e_1 R_{n_1, k_2}^{\alpha, \downarrow}(r_1, r_2, s_1, s_2) \label{eq:3.3.51}\\
=& \begin{cases}
\varphi_1^{\alpha}(n_1, r_1, r_2) R_{n_1-1, k_2}^{\alpha, \downarrow}(r_1, r_2, s_1, s_2), & 1 \le n_1 \le r_1-1,\\
0, & n_1=0,
\end{cases}\notag\\
&f_1 R_{n_1, k_2}^{\alpha, \downarrow}(r_1, r_2, s_1, s_2) \label{eq:3.3.52}\\
=& \begin{cases}
R_{n_1+1, k_2}^{\alpha, \downarrow}(r_1, r_2, s_1, s_2), & 0 \le n_1 \le r_1-2,\\
0, & n_1=0.
\end{cases}\notag
\end{align}
By Lemma \ref{lem:com}, \eqref{eq:3.3.12}-\eqref{eq:3.3.15} and \eqref{eq:3.3.35}-\eqref{eq:3.3.42},
one has
\begin{align}
&e_2 R_{n_1, k_2}^{\alpha, \bullet}(r_1, r_2, s_1, s_2) \label{eq:3.3.53}\\
=& \begin{cases}
\varphi_2^{-\alpha}(k_2, r_1, p_2-r_2) R_{n_1, k_2-1}^{\alpha, \bullet}(r_1, r_2, s_1, s_2), & 1 \le k_2 \le p_2-r_2-1,\\
B_{n_1, r_2-1}^{\alpha, \bullet}(r_1, r_2, s_1, s_2), & k_2=0,
\end{cases}\notag\\
&f_2 R_{n_1, k_2}^{\alpha, \bullet}(r_1, r_2, s_1, s_2) \label{eq:3.3.54}\\
=& \begin{cases}
R_{n_1, k_2+1}^{\alpha, \bullet}(r_1, r_2, s_1, s_2), & 0 \le k_2 \le p_2-r_2-2,\\
0, & k_2=p_2-r_2-1,
\end{cases}\notag
\end{align}
where $0 \le n_1 \le r_1-1$ if $\bullet \in \{\uparrow, \downarrow\}$
and $0 \le n_1 \le p_1-r_1-1$ if $\bullet \in \{\leftarrow, \rightarrow\}$.

Let $X_{r_1, r_2}^{\alpha}(s_1, s_2)$ be the space spanned by the vectors
\begin{align}
\left\{ B_{n_1, n_2}^{\alpha}(r_1, r_2, s_1, s_2); 1 \le n_i \le r_i-1 \right\}. \notag
\end{align}
By \eqref{eq:3.1.0}-\eqref{eq:3.1.2}, \eqref{eq:3.3.18} and \eqref{eq:3.3.34},
$X_{r_1, r_2}^{\alpha}(s_1, s_2)$ is a left module and is isomorphic to the simple module $X_{r_1, r_2}^{\alpha}$.  

\begin{prop}\label{prop:3.3.1}
Let $P_{r_1, r_2}^{\alpha}(s_1, s_2)$ be the space spanned by the vectors 
\begin{align}
&T_{n_1, n_2}^{\alpha, \uparrow}(r_1, r_2, s_1, s_2), \  T_{k_1, n_2}^{\alpha, \leftarrow}(r_1, r_2, s_1, s_2), \
T_{k_1, n_2}^{\alpha, \rightarrow}(r_1, r_2, s_1, s_2), \ T_{n_1, n_2}^{\alpha, \downarrow}(r_1, r_2, s_1, s_2),\notag\\
&L_{n_1, k_2}^{\alpha, \uparrow}(r_1, r_2, s_1, s_2), \  L_{k_1, k_2}^{\alpha, \leftarrow}(r_1, r_2, s_1, s_2), \
L_{k_1, k_2}^{\alpha, \rightarrow}(r_1, r_2, s_1, s_2), \ L_{k_1, n_2}^{\alpha, \downarrow}(r_1, r_2, s_1, s_2),\notag\\
&R_{n_1, k_2}^{\alpha, \uparrow}(r_1, r_2, s_1, s_2), \  R_{k_1, k_2}^{\alpha, \leftarrow}(r_1, r_2, s_1, s_2), \
R_{k_1, k_2}^{\alpha, \rightarrow}(r_1, r_2, s_1, s_2), \ R_{n_1, k_2}^{\alpha, \downarrow}(r_1, r_2, s_1, s_2),\notag\\
&B_{n_1, n_2}^{\alpha, \uparrow}(r_1, r_2 s_1, s_2), \ B_{k_1, n_2}^{\alpha, \leftarrow}(r_1, r_2, s_1, s_2), \
B_{k_1, n_2}^{\alpha, \rightarrow}(r_1, r_2, s_1, s_2), \
B_{n_1, n_2}^{\alpha, \downarrow}(r_1, r_2, s_1, s_2),\notag
\end{align}
for $0 \le n_i \le r_i-1$ and $0 \le k_i \le p_i-r_i-1$.
Then $P_{r_1, r_2}^{\alpha}(s_1, s_2)$ is a $4 p_1 p_2$-dimensional
indecomposable left ideal whose socle is isomorphic to the simple module
$X^{\alpha}_{r_1, r_2}$, in particular, for fixed $1 \le r_i \le p_i-1$ and $\alpha = \pm$,
the modules $P_{r_1, r_2}^{\alpha}(s_1, s_2)$ for $1 \le s_i \le r_i$
are isomorphic to each other.
\end{prop}
By Proposition \ref{prop:3.3.1},
we can write $P_{r_1, r_2}^{\alpha} \cong P_{r_1, r_2}^{\alpha}(s_1, s_2)$.
A basis of $P_{r_1, r_2}^{\alpha}$ 
which corresponds to the basis $T^{\alpha, \bullet}(s_1, s_2), L^{\alpha, \bullet}(s_1, s_2),
R^{\alpha, \bullet}(s_1, s_2), B^{\alpha, \bullet}(s_1, s_2)$ is denoted by
\begin{align}
\mathsf{t}^{\alpha, \bullet}(r_1, r_2), \mathsf{l}^{\alpha, \bullet}(r_1, r_2), \mathsf{r}^{\alpha, \bullet}(r_1, r_2), 
\mathsf{b}^{\alpha, \bullet}(r_1, r_2). \notag
\end{align}
\section{Projective modules and primitive idempotents}
In this section we will construct primitive idempotents of $\boldsymbol{\mathfrak{g}_{p_1, p_2}}$
and show that the simple module $X_{p_1, p_2}^{\alpha}$ and indecomposable modules
constructed in Section 3 are projective.
Moreover, we give the block decomposition of  $\boldsymbol{\mathfrak{g}_{p_1, p_2}}$.
\subsection{Primitive idempotents}
In order to calculate the products of elements in $\boldsymbol{\mathfrak{g}_{p_1, p_2}}$, 
the following lemma is useful.
\begin{lem}\label{lem:eigenvalue}
\begin{enumerate}
\item
Let $w$ be an element in $\boldsymbol{\mathfrak{g}_{p_1, p_2}}$ with $K w = \alpha q_1^{r_1-1-2 n_1} q_2^{r_2-1-2 n_2}w$
for $0 \le n_1 \le r_1-1$ and $0 \le n_2 \le r_2-1$.
Then 
\begin{align}
v_{r_1, r_2}^{\alpha}(s_1, s_2) w = \begin{cases}
2p_1 p_2 w, & (n_1, n_2) = (s_1-1, s_2-1), \\
0, & \text{otherwise}.
\end{cases} \notag 
\end{align} 
\item 
Let $w$  be an element in $\boldsymbol{\mathfrak{g}_{p_1, p_2}}$ 
with $K w = - \alpha q_1^{p_1 - r_1-1-2 k_1} q_2^{r_2-1-2 n_2}w$
for $0 \le k_1 \le p_1-r_1-1$ and $0 \le n_2 \le r_2-1$.
Then $v_{r_1, r_2}^{\alpha}(s_1, s_2) w = 0$.
\item
Let $w$ be an element in $\boldsymbol{\mathfrak{g}_{p_1, p_2}}$
with $K w = - \alpha q_1^{r_1-1-2 n_1} q_2^{p_2-r_2-1-2 k_2}w$
for $0 \le n_1 \le r_1-1$ and $0 \le k_2 \le p_2-r_2-1$.
Then $v_{r_1, r_2}^{\alpha}(s_1, s_2) w = 0$.
\item
Let $w$  be an element in $\boldsymbol{\mathfrak{g}_{p_1, p_2}}$ 
with $K w = \alpha q_1^{p_1-r_1-1-2 k_1} q_2^{p_2-r_2-1-2 k_2}w$
for $0 \le k_1 \le p_1 - r_1-1$ and $0 \le k_2 \le p_2-r_2-1$.
Then $v_{r_1, r_2}^{\alpha}(s_1, s_2) w = 0$.
\end{enumerate}
\end{lem}
\begin{proof}
We will show (1) and (2).
The proofs of (3) and (4) are similar to the proof of (2).

(1)
We can see 
\begin{align}
v_{r_1, r_2}^{\alpha}(s_1, s_2) w &= \sum_{\ell = 0}^{2 p_1 p_2} (\alpha q_1^{-(r_1-2 s_1+1)}
q_2^{- (r_2-2 s_2+1)})^{\ell} (\alpha q_1^{r_1-1-2 n_1} q_2^{r_2-2 n_2-2}) w \notag\\
&= \sum_{\ell=0}^{2 p_1 p_2-1} (q_1^{2(s_1-n_1-1)} q_2^{2(s_2-n_2-1)})^{\ell} w \notag\\
&= \sum_{\ell=0}^{2 p_1 p_2-1} (q^{4 \left\{p_2 (s_1-n_1-1) + p_1 (s_2-n_2-1) \right\}})^{\ell} w\notag 
\end{align}
Since $q = \exp (\sqrt{-1} \pi / 2 p_1 p_2)$,
the sum does not vanish if and only if 
\begin{align}
p_2 (s_1-n_1-1) + p_1 (s_2-n_2-1) = m p_1 p_2 \notag
\end{align}
for some $m \in \mathbb{Z}$.
Since 
\begin{align}
n_1 = s_1 - 1 + \frac{p_1}{p_2} ( s_2 - n_2 - 1) - m p_1 p_2 \label{eq-1.1.1.1}
\end{align}
and $p_1$, $p_2$ are coprime,
the right hand side of \eqref{eq-1.1.1.1} does not belong to $\{0, 1, \dots , r_1-1\}$ if $s_2 - n_2 - 1 \neq 0$
and  $m \neq 0$.
Thus we have $(n_1, n_2) = (s_1-1, s_2-1)$.

(2)
We have
\begin{align}
v_{r_1, r_2}^{\alpha}(s_1, s_2) w &= \sum_{\ell = 0}^{2 p_1 p_2-1} (\alpha q_1^{-(r_1-2 s_1+1)}
q_2^{- (r_2-2 s_2+1)})^{\ell} (- \alpha q_1^{p_1-r_1-1-2 k_1} q_2^{r_2-2 n_2-2}) w\notag\\
&= \sum_{\ell=0}^{2p_1p_2-1} q^{-4 \ell (p_2(r_1+k_1-s_1+1) - p_1 (s_2-n_2-1) )} w .\notag
\end{align}
Then the sum does not vanish if and only if $p_2(r_1+k_1-s_1+1) - p_1 (s_2-n_2-1) = m p_1 p_2 $ for some $m \in \mathbb{Z}$.
Since $1 \le k_1 + r_1 -s_1 +1 \le p_1-1$ and $-(r_2-1) \le s_2-n_2-1 \le r_2 - 1$,
we have the assertion.
\end{proof}

Let $X_{p_1, p_2}^{\alpha}(s_1, s_2)$ be the space spanned by the vectors
\begin{align}
\left\{B_{n_1, n_2}^{\alpha, \downarrow}(p_1, p_2, s_1, s_2) ; 0 \le n_i \le p_i-1, \ i=1, 2\right\}.\notag
\end{align}
We can see that $X_{p_1, p_2}^{\alpha}(s_1, s_2)$ is isomorphic to the simple module $X_{p_1, p_2}^{\alpha}$.
\begin{prop}\label{prop:4.1.1}
The element $B_{s_1-1, s_2-1}^{\alpha, \downarrow}(p_1, p_2, s_1, s_2)$ is a primitive idempotent of
 $\boldsymbol{\mathfrak{g}_{p_1, p_2}}$
for $\alpha = \pm$ and $1 \le s_i \le p_i-1$.
In particular, the simple module $X_{p_1, p_2}^{\alpha}$ is projective.
\end{prop}
\begin{proof}
By the action of $\boldsymbol{\mathfrak{g}_{p_1, p_2}}$ on simple modules and Lemma \ref{lem:eigenvalue},
we have
\begin{align}
&( B_{s_1-1, s_2-1}^{\alpha, \downarrow}(p_1, p_2, s_1, s_2) )^2 \notag\\
=&\frac{2 p_1 p_2 f_1^{s_1-1}f_2^{s_2-1}}
{\Phi^{\alpha}(p_1, p_2)}
e_1^{p_1-1}e_{2}^{p_2-1}f_1^{p_1-s_1}f_2^{p_2-s_2} B_{s_1-1, s_2-1}^{\alpha, \downarrow}(p_1, p_2, s_1, s_2)\notag\\
=&\frac{2 p_1 p_2}{\Phi^{\alpha}(p_1, p_2)} 
\prod_{i_1=1}^{p_1-1}\prod_{i_2=1}^{p_2-1} 
\varphi_{1}^{\alpha}(i_1, p_1, p_2) \varphi_{2}^{\alpha}(i_2, p_1, p_2)
B_{s_1-1, s_2-1}^{\alpha, \downarrow}(p_1, p_2, s_1, s_2)\notag\\
=& B_{s_1-1, s_2-1}^{\alpha, \downarrow}(p_1, p_2, s_1, s_2).\notag
\end{align}
\end{proof}

By similar methods used in \cite{Ar}, we have the following.  
\begin{prop}\label{prop:4.2.1}
The elements $B_{s_1-1, s_2-1}^{\alpha, \uparrow}(r_1, p_2, s_1, s_2)$ and 
$B_{s_1-1, s_2-1}^{\alpha, \uparrow}(p_1, r_2, s_1, s_2)$
are primitive idempotents.
In particular, indecomposable modules $P_{r_1, p_2}^{\alpha}$ and $P_{p_1, r_2}^{\alpha}$
are projective.
\end{prop}

We also have the following.

\begin{prop}\label{prop:4.3.1}
The elements 
\begin{align}
T_{s_1-1, s_2-1}^{\alpha, \uparrow}(r_1, r_2, s_1, s_2), \ 1 \le s_1 \le p_1, \ 
 1 \le s_2 \le p_2, 
\end{align}
are primitive idempotents of $\boldsymbol{\mathfrak{g}_{p_1, p_2}}$.
In particular, the indecomposable module $P_{r_1, r_2}^{\alpha}$ is a projective module.
\end{prop}
\begin{proof}
It is clear that $T_{s_1-1, s_2-2}^{\alpha, \uparrow}(r_1, r_2, s_1, s_2)$ generate the left ideal
$P_{r_1, r_2}^{\alpha}(s_1, s_2)$.
By Lemma \ref{lem:eigenvalue}, we can see that
\begin{align}
2 p_1 p_2 e_1^{p_1-1} e_2^{p_2-1}f_1^{p_1-s_1} f_2^{p_2-s_2}T_{s_1-1, s_2-1}^{\alpha, \uparrow}(r_1, r_2, s_1, s_2)
= \Phi^{\alpha}(r_1, r_2) B_{s_1-1, s_2-1}^{\alpha, \downarrow}(r_1, r_2, s_1, s_2). \notag
\end{align}
We also have 
\begin{align}
&2p_1p_2\sum_{m_2=1}^{p_2-r_2}\delta_{m_2}^{\alpha}(r_1 r_2)e^{p_1-1}e_2^{p_2-1-m_2}f_1^{p_1-s_1}f_2^{p_2-s_2-m_2}
T_{s_1-1, s_2-1}^{\alpha, \uparrow}(r_1, r_2, s_1, s_2)\notag\\
&=2p_1p_2\Bigl\{ \delta_{p_2-r_2}^{\alpha}(r_1, r_2)
e_1^{p_1-1} e_2^{r_2-1}T_{p_1-r_1-1, r_2-1}^{\alpha, \rightarrow}(r_1, r_2, s_1, s_2) \notag\\
&+ \sum_{m_2=1}^{p_2-r_2-1} \delta_{m_2}^{\alpha}(r_1, r_2) R_{p_1-r_1-1, p_2-r_2-1-m_2}^{\alpha, \rightarrow}
(r_1, r_2, s_1, s_2)
\Bigr\}\notag\\
&=\Phi^{\alpha}(r_1, r_2) \Bigl\{
T_{0, 0}^{\alpha, \downarrow}(r_1, r_2, s_1, s_2) + \sum_{i_2=1}^{r_2-1}\frac{1}{\varphi_2^{\alpha}(i_2, r_1, r_1)}
B_{0, 0}^{\alpha, \downarrow}(r_1, r_2, s_1, s_2)\notag\\
&+ \sum_{k_2=2}^{p_2-r_2-1}\frac{1}{\varphi_2^{-\alpha}(k_2, r_1, p_2-r_2)} B_{0, 0}^{\alpha, \downarrow}
(r_1, r_2, s_1, s_2)
\Bigr\}\notag
\end{align}
and 
\begin{align}
&2p_1p_2\sum_{m_1=1}^{p_1-r_1}\gamma_{m_1}^{\alpha}(r_1 r_2)e^{p_1-1-m_1}e_2^{p_2-1}f_1^{p_1-s_1-m_1}f_2^{p_2-s_2}
T_{s_1-1, s_2-1}^{\alpha, \uparrow}(r_1, r_2, s_1, s_2)\notag\\
&=2p_1p_2\Bigl\{ \gamma_{p_1-r_1}^{\alpha}(r_1, r_2)
e_1^{r_1-1} e_2^{p_2-1}R_{r_1-1, p_2-r_2-1}^{\alpha, \uparrow}(r_1, r_2, s_1, s_2) \notag\\
&+ \sum_{m_1=1}^{p_1-r_1-1} \gamma_{m_1}^{\alpha}(r_1, r_2) R_{p_1-r_1-1-m_1, p_2-r_2-1}^{\alpha, \rightarrow}
(r_1, r_2, s_1, s_2)
\Bigr\}\notag\\
&=\Phi^{\alpha}(r_1, r_2) \Bigl\{
B_{0, 0}^{\alpha, \uparrow}(r_1, r_2, s_1, s_2) + \sum_{i_1=1}^{r_1-1}\frac{1}{\varphi_1^{\alpha}(i_1, r_1, r_1)}
B_{0, 0}^{\alpha, \downarrow}(r_1, r_2, s_1, s_2)\notag\\
&+ \sum_{k_1=1}^{p_1-r_1-1}\frac{1}{\varphi_1^{-\alpha}(k_1, p_1-r_1, r_2)} B_{0, 0}^{\alpha, \downarrow}
(r_1, r_2, s_1, s_2)
\Bigr\}. \notag
\end{align}
By Lemma \ref{lem:eigenvalue} and \eqref{eq:3.3.2}, we have
\begin{align}
&2 p_1 p_2 \left\{\sum_{m_1=1}^{p_1-r_1}\sum_{m_2=1}^{p_2-r_2}
\gamma_{m_1}^{\alpha}(r_1, r_2) \delta_{m_2}^{\alpha}(r_1, r_2)
e_1^{p_1-1-m_1}e_2^{p_2-1-m_2}f_1^{p_1-s_1-m_1} f_2^{p_2-s_2-m_2} \right\} \notag\\
&\qquad \qquad \qquad \times
T_{s_1-1, s_2-1}^{\alpha, \uparrow}(r_1, r_2, s_1, s_2) \notag\\
&=2 p_1 p_2\Bigl\{
\gamma_{p_1-r_1}^{\alpha}(r_1, r_2) \delta_{p_2-r_2}^{\alpha}(r_1, r_2)e_1^{r_1-1}e_2^{r_2-1}
T_{r_1-1, r_2-1}^{\alpha, \uparrow}(r_1, r_2, s_1, s_2) \notag\\
&+ \gamma_{p_1-r_1}^{\alpha}(r_1, r_2) \sum_{m_2=1}^{p_2-r_2-1}
\delta_{m_2}^{\alpha}(r_1, r_2)
e_1^{r_1-1}e_2^{p_2-1-m_2} R_{r_1-1, p_2-r_2 -1-m_2}^{\alpha, \uparrow}(r_1, r_2, s_1, s_2)\notag\\
&+ \delta_{p_2-r_2}^{\alpha}(r_1, r_2) \sum_{m_1=1}^{p_1-r_1-1}\gamma_{m_1}^{\alpha}(r_1, r_2)
e_1^{p_1-1-m_1} e_2^{r_2-1}T_{p_1-r_1 -1 - m_1, r_2-1}^{\alpha, \rightarrow}(r_1, r_2, s_1, s_2)\notag\\
&\sum_{m_1=1}^{p_1-r_1-1}\sum_{m_2=1}^{p_2-r_2-1}
\gamma_{m_1}^{\alpha}(r_1, r_2) \delta_{m_2}^{\alpha}(r_1, r_2)
e_1^{p_1-1-m_1} e_2^{p_2-1-m_2}\notag\\
&\qquad \qquad \qquad \qquad \times
R_{p_1-r_1-1-m_1, p_2-r_2-1 -m_2}^{\alpha, \rightarrow}(r_1, r_2, s_1, s_2)
\Bigr\}.\notag
\end{align}
Then we have
\begin{align}
&2 p_1 p_2 \gamma_{p_1-r_1}^{\alpha}(r_1, r_2) \delta_{p_2-r_2}^{\alpha}(r_1, r_2) e_1^{r_1-1}e_2^{r_2-1}
T_{r_1-1, r_2-1}^{\alpha, \uparrow}(r_1, r_2, s_1, s_2)\notag\\
&=2 p_1 p_2 \gamma_{p_1-r_1}^{\alpha}(r_1, r_2) \delta_{p_2-r_2}^{\alpha}(r_1, r_2)
e_2^{r_2-1} \Bigl\{\prod_{i_1=1}^{r_1-1} \varphi_{1}^{\alpha}(i_1, r_1, r_2)
T_{0, r_2-1}^{\alpha, \uparrow}(r_1, r_2, s_1, s_2) \notag\\
&\qquad + \sum_{i_1=1}^{r_1-1}\prod_{\genfrac{}{}{0pt}{}{j_1=1}
{j_1 \neq i_1}}^{r_1-1} \varphi_{1}^{\alpha}(j_1, r_1, r_2)
T_{0, r_2-1}^{\alpha, \downarrow}(r_1, r_2, s_1, s_2)  \Bigr\}\notag\\
&=\Phi^{\alpha}(r_1, r_2) \Bigl\{
T_{0, 0}^{\alpha, \uparrow}(r_1, r_2, s_1, s_2)+ 
\sum_{i_2=1}^{r_2-1} \frac{1}{\varphi_{2}^{\alpha}(i_2, r_1, r_2)}
B_{0, 0}^{\alpha, \uparrow}(r_1, r_2, s_1, s_2) \notag\\
&+ \sum_{i_1=1}^{r_1-1} \frac{1}{\varphi_{1}^{\alpha}(i_1, r_1, r_2)}
T_{0, 0}^{\alpha, \downarrow}(r_1, r_2, s_1, s_2) \notag\\
&+ \sum_{i_1=1}^{r_1-1}
\frac{1}{\varphi_{1}^{\alpha}(i_1, r_1, r_2)}
\sum_{i_2=1}^{r_2-1} \frac{1}{\varphi_{2}^{\alpha}(j_2, r_1, r_2)}
B_{0, 0}^{\alpha, \downarrow}(r_1, r_2, s_1, s_2) \Bigr\},\notag
\end{align}
and
\begin{align}
&2 p_1 p_2 
\gamma_{p_1-r_1}^{\alpha}(r_1, r_2)\sum_{m_2=1}^{p_2-r_2-1}
\delta_{m_2}^{\alpha}(r_1, r_2)
e_1^{r_1-1}e_2^{p_2-1-m_2}
R_{r_1-1, p_2-r_2 -1-m_2}^{\alpha, \uparrow}(r_1, r_2, s_1, s_2)\notag\\
&=\Phi^{\alpha}(r_1, r_2)
\Bigl\{ \sum_{k_2=1}^{p_2-r_2-1}\frac{1}{\varphi_2^{-\alpha}(k_2, r_1, p_2-r_2)}
B_{0, 0}^{\alpha, \uparrow}(r_1, r_2, s_1, s_2)\notag\\
& \qquad + \sum_{i_1=1}^{r_1-1}\frac{1}{\varphi_1^{\alpha}(i_1, r_1, r_2)}
\sum_{k_2=1}^{p_2-r_2-1}\frac{1}{\varphi_2^{-\alpha}(k_2, r_1, p_2-r_2)}
B_{0, 0}^{\alpha, \downarrow}(r_1, r_2, s_1, s_2). \notag
\end{align}
Similarly we have
\begin{align}
&2p_1 p_2 \delta_{p_2-r_2}^{\alpha}(r_1, r_2) \sum_{m_1=1}^{p_1-r_1-1} \gamma_{m_1}^{\alpha}(r_1, r_2)
e_1^{p_1-m_1-1}e_2^{r_2-1} T_{p_1-r_1-1-m_1, r_2-1}^{\alpha, \rightarrow}(r_1, r_2, s_1, s_2)\notag\\
&=\Phi^{\alpha}(r_1, r_2)\Bigl\{
\sum_{k_1=1}^{p_1-r_1-1} \frac{1}{\varphi_{1}^{-\alpha}(p_1-r_1, r_2)}
T_{0, 0}^{\alpha, \downarrow}(r_1, r_2, s_1, s_2)\notag\\
&+ \sum_{k_1=1}^{p_1-r_1-1} \frac{1}{\varphi_{1}^{-\alpha}( k_1, p_1-r_1, r_2)}
\sum_{i_1=1}^{r_1-1}\frac{1}{\varphi_{1}^{\alpha}(i_1, r_1, r_2)}
B_{0, 0}^{\alpha, \downarrow}(r_1, r_2, s_1, s_2), \notag
\end{align}
and 
\begin{align}
&2 p_1 p_2 
\sum_{m_1=1}^{p_1-r_1-1}\sum_{m_2=1}^{p_2-r_2-1}
\gamma_{m_1}^{\alpha}(r_1, r_2) \delta_{m_2}^{\alpha}(r_1, r_2)
e_1^{p_1-1-m_1} e_2^{p_2-1-m_2}\notag\\
&\qquad \qquad \qquad \qquad \times
R_{p_1-r_1-1-m_1, p_2-r_2-1 -m_2}^{\alpha, \rightarrow}(r_1, r_2, s_1, s_2)\notag\\
&=\Phi^{\alpha}(r_1, r_2)
\sum_{k_1=1}^{p_1-r_1-1}\frac{1}{\varphi_{1}^{-\alpha}(p_1-r_1, r_2)}
\sum_{k_2=1}^{p_2-r_2-1}\frac{1}{\varphi_{2}^{-\alpha}(r_1, p_2-r_2)}
B_{0, 0}^{\alpha, \downarrow}(r_1, r_2, s_1, s_2). \notag
\end{align}
We finally have
\begin{align}
&\left( T_{s_1-1, s_2-1}^{\alpha, \uparrow}(r_1, r_2, s_1, s_2)\right)^{2}\notag\\
&=
T_{s_1-1, s_2-1}^{\alpha, \uparrow}(r_1, r_2, s_1, s_2)
+ \Psi_1^{\alpha}(r_1, r_2) T_{s_1-1, s_2-1}^{\alpha, \downarrow}(r_1, r_2, s_1, s_2)\notag\\
&+ \Psi_2^{\alpha}(r_1, r_2) B_{s_1-1, s_2-1}^{\alpha, \uparrow}(r_1, r_2, s_1, s_2)
+ \Psi_1^{\alpha}(r_1, r_2)   \Psi_2^{\alpha}(r_1, r_2) B_{s_1-1, s_2-1}^{\alpha, \downarrow}(r_1, r_2, s_1, s_2)\notag\\
&- \Psi_1^{\alpha}(r_1, r_2) T_{s_1-1, s_2-1}^{\alpha, \downarrow}(r_1, r_2, s_1, s_2)
- \Psi_1^{\alpha}(r_1, r_2) \Psi_2^{\alpha}(r_1, r_2) B_{s_1-1, s_2-1}^{\alpha, \downarrow}(r_1, r_2, s_1, s_2)\notag\\
&-\Psi_2^{\alpha}(r_1, r_2) B_{s_1-1, s_2-1}^{\alpha, \uparrow}(r_1, r_2, s_1, s_2)
- \Psi_1^{\alpha}(r_1, r_2) \Psi_2^{\alpha}(r_1, r_2) B_{s_1-1, s_2-1}^{\alpha, \downarrow}(r_1, r_2, s_1, s_2)\notag\\
&+ \Psi_1^{\alpha}(r_1, r_2) \Psi_2^{\alpha}(r_1, r_2) B_{s_1-1, s_2-1}^{\alpha, \downarrow}(r_1, r_2, s_1, s_2)
\notag\\
&= T_{s_1-1, s_2-1}^{\alpha, \uparrow}(r_1, r_2), \notag
\end{align}
which shows that $T_{s_1-1, s_2-1}^{\alpha}(r_1, r_2, s_1, s_2)$ is idempotent.
\end{proof}

\subsection{Block decomposition}
We can see that the elements
\begin{align}
&C_1 = - q_1^{p_2} K^{-p_2} - q_1^{-p_2} K^{p_2} - (q_1^{p_2} - q_1^{-p_2})^2 e_1 f_1, \label{eq:4.4.1}\\
&C_2 = - q_2^{p_1} K^{-p_1} - q_2^{-p_1} K^{p_1} - (q_2^{p_1} - q_2^{-p_1})^2 e_2 f_2 \label{eq:4.4.2}
\end{align}
are central elements \cite{FGST4}.
The central element
$C_1$ acts as a  scalar $\alpha^{p_2} (-1)^{r_2} (q_1^{p_2 r_1} + q_1^{-p_2r_1})$ on the simple module
$X_{r_1, r_2}^{\alpha}$ and 
the central element $C_2$ acts as a scalar $\alpha^{p_1 (-1)^{r_1}} ( q_2^{p_1r_2} + q_2^{-p_1r_2})$
on the simple module $X_{r_1, r_2}^{\alpha}$.
We set $\beta_1 (r_1, r_2) = \alpha^{p_2}(-1)^{r_2} (q_1^{p_2r_1} + q_1^{-p_2r_1})$
and $\beta_2^{\alpha}(r_1, r_2) = \alpha^{p_1} (-1)^{r_1} (q_2^{p_1r_2} + q_2^{-p_1r_2})$.

By direct calculations, we obtain the following.
\begin{lem}\label{lem:4.5.1}
We have
\begin{align}
&(C_1 - \beta_{1}^{\alpha}(r_1, p_2))^2 B^{\alpha, \uparrow}_{s_1-1, s_2-1}(r_1, p_2, s_1, s_2) = 0,\notag\\
&(C_2 - \beta_2^{\alpha}(r_1, p_2)) B^{\alpha, \uparrow}_{s_1-1, s_2-1}(r_1, p_2, s_1, s_2) = 0, \notag\\
&(C_1 - \beta_{1}^{\alpha}(p_1, r_2)) B^{\alpha, \uparrow}_{s_1-1, s_2-1}(p_1, r_2, s_1, s_2) = 0,\notag\\
&(C_2 - \beta_2^{\alpha}(p_1, r_2))^2 B^{\alpha, \uparrow}_{s_1-1, s_2-1}(p_1, r_2, s_1, s_2) = 0, \notag\\
&(C_1 - \beta_{1}^{\alpha}(r_1, r_2))^2 T^{\alpha, \uparrow}_{s_1-1, s_2-1}(r_1, r_2, s_1, s_2) = 0,\notag\\
&(C_2 - \beta_2^{\alpha}(r_1, r_2))^2 T^{\alpha, \uparrow}_{s_1-1, s_2-1}(r_1, r_2, s_1, s_2) = 0. \notag
\end{align}
\end{lem}
Set
\begin{align}
&Q(p_1, p_2) = \left\{ v \in \boldsymbol{\mathfrak{g}_{p_1, p_2}} |
(C_1 - 2)v = 0, \ (C_2-2)v = 0\right\},\notag\\
&Q(0, p_2) = \left\{ v \in \boldsymbol{\mathfrak{g}_{p_1, p_2}} |
(C_1 - (-1)^{p_2}2)v = 0, \ (C_2- (-1)^{p_1}2)v = 0\right\},\notag\\
&Q(r_1, p_2) = \left\{ v \in \boldsymbol{\mathfrak{g}_{p_1, p_2}} |
(C_1 - \beta_{1}^{+}(r_1, p_2))^2 v=0, \ (C_2-(-1)^{p_1+r_1} 2)v = 0 \right\},\notag\\
&Q(p_1, r_2) = \left\{ v \in \boldsymbol{\mathfrak{g}_{p_1, p_2}} |
(C_1 - (-1)^{p_2+r_2}2) v=0, \ (C_2-\beta_2^+(p_1, r_2))^2v = 0 \right\},\notag\\
&Q(r_1, r_2) = \left\{v \in \boldsymbol{\mathfrak{g}_{p_1, p_2}} |
(C_1 - \beta_1^+(r_1, r_2))^2v = 0, \ (C_2- \beta_2^+(r_1, r_2))^2v=0 \right\}\notag
\end{align}
Since $C_1$ and $C_2$ are  the central elements of  $\boldsymbol{\mathfrak{g}_{p_1, p_2}}$,
we can see that the space $Q(r_1, r_2)$ is two-sided ideal of  $\boldsymbol{\mathfrak{g}_{p_1, p_2}}$.
For $1 \le r_1 \le p_1-1$,
we have $\beta_{1}^+(r_1, p_2) = \beta_1^{-}(p_1-r_1, p_2)$
and $\beta_2^-(p_1-r_1, p_2) = (-1)^{p_1} (-1)^{p_1-r_1} (-1)^{p_1} 2= (-1)^{p_1+r_1}2$.
Similarly we have $\beta_{2}^+(p_1, r_2) = \beta_2^{-}(p_1, p_2-r_2)$
and $\beta_1^-(p_1, p_2-r_2) = (-1)^p_2(-1)^{p_2-r_2} (-1)^{p_2} 2= (-1)^{p_2+r_2}2$
for $1 \le r_2 \le p_2-1$.
We also have
\begin{align}
&\beta_{1}^+(r_1, r_2) = \beta_{1}^{-}( p_1-r_1, r_2) = \beta_{1}^{-}(r_1, p_2-r_1) = \beta_{1}^{+}(p_1-r_1, p_2-r_2),\notag\\
&\beta_{2}^+(r_1, r_2) = \beta_{2}^{-}( p_1-r_1, r_2) = \beta_{2}^{-}(r_1, p_2-r_1) = \beta_{2}^{+}(p_1-r_1, p_2-r_2)\notag
\end{align}
for $1 \le r_1 \le p_1-1$ and $1 \le r_2 \le p_2-1$.
By Lemma \ref{lem:4.5.1} we have injective maps
\begin{align}
&X_{p_1, p_2}^+ \hookrightarrow Q(p_1, p_2), \notag\\
&X_{p_1, p_2}^- \hookrightarrow  Q(0, p_2), \notag\\
&P_{r_1, p_2}^+, P_{p_1-r_1, p_2}^- \hookrightarrow Q(r_1, p_1),\notag\\
&P_{p_1, r_2}^+, P_{p_1, p_2-r_2}^- \hookrightarrow Q(p_1, r_2), \notag\\
&P_{r_1, r_2}^+, P_{r_1, p_2-r_2}^-, P_{p_1-r_1, r_2}^-, P_{p_1-r_1, p_2-r_2}^+ \hookrightarrow Q(r_1, r_2).\notag
\end{align}
Note that
$(\beta_1^+(r_1, p_2), (-1)^{p_1+r_1}2) = (\beta_1^+(r_1^{\prime}, p_2), (-1)^{p_1+r_1^{\prime}}2)$
if and only if $r_1 = r_1^{\prime}$ for $1 \le r_1, r_1^{\prime} \le p_1-1$.
We also have $(\beta_{1}^+(r_1, r_2), \beta_2^+(r_1, r_2)) = (\beta_1^+(r_1^{\prime}, r_2^{\prime}),
\beta_2^+(r_1^{\prime}, r_2^{\prime}))$ if and only if 
$(r_1^{\prime}, r_2^{\prime}) = (p_1-r_1, p_2-r_2), (r_1, r_2)$ for $1 \le r_1, r_1^{\prime} \le p_1-1$
and $1 \le r_2, r_2^{\prime} \le p_2-1$.
Thus we have the decomposition of  $\boldsymbol{\mathfrak{g}_{p_1, p_2}}$
into two sided ideals.
\begin{prop}[\cite{FGST4}]\label{prop:4.5.2}
We have a decomposition of $\gpq$ into two-sided ideals{\rm :}
\begin{align}
\boldsymbol{\mathfrak{g}_{p_1, p_2}} = \bigoplus_{(r_1, r_2) \in I} Q(r_1, r_2) \oplus \bigoplus_{r_1=1}^{p_1-1}Q(r_1, p_2)
\oplus \bigoplus_{r_2=1}^{p_2-1} Q(p_1, r_2) \oplus Q(p_1, p_2) \oplus Q(0, p_2)\notag
\end{align}
where
\begin{align}
I=\left\{(r_1, r_2) | 1 \le r_1 \le p_1-1, 1 \le r_2 \le p_2-1, p_2r_1+ p_1r_2 \le p_1p_2 \right\}.\notag
\end{align}
\end{prop}
\begin{cor}\label{cor:4.5.3}
\begin{enumerate}
\item
The elements $B_{s_1-1, s_2-1}^{+, \downarrow}(p_1, p_2, s_1, s_2)$ for $1 \le s_1 \le p_1-1$ and 
$1 \le s_2 \le p_2-1$
are mutually orthogonal primitive idempotents of $Q(p_1, p_2)$.
\item
The elements $B_{s_1-1, s_2-1}^{-, \downarrow}(p_1, p_2, s_1, s_2)$, for $1 \le s_1 \le p_1-1$ and 
$1 \le s_2 \le p_2-1$,
are mutually orthogonal primitive idempotents of $Q(0, p_2)$.
\item
The elements 
\begin{align}
&B_{s_1-1, s_2-1}^{+, \uparrow}(r_1, p_2, s_1, s_2), \ 1 \le s_1 \le r_1, \ 1 \le s_2 \le p_2,\notag\\
&B_{t_1-1, s_2-2}^{-, \uparrow}(p_1-r_1, p_2, t_1, s_2), \ 1 \le t_1 \le p_1-r_1, \ 1 \le s_2 \le p_2,\notag
\end{align}
are mutually orthogonal primitive idempotents of $Q(r_1, p_2)$.
\item
The elements 
\begin{align}
&B_{s_1-1, s_2-1}^{+, \uparrow}(p_1, r_2, s_1, s_2), \ 1 \le s_1 \le p_1, \ 1 \le s_2 \le r_2,\notag\\
&B_{s_1-1, t_2-2}^{-, \uparrow}(p_1, p_2-r_2, s_1, t_2), \ 1 \le s_1 \le p_1, \ 1 \le t_2 \le p_2-r_2,\notag
\end{align}
are mutually orthogonal primitive idempotents of $Q(p_1, r_2)$.
\item
The elements 
\begin{align}
&T_{s_1-1, s_2-1}^{+, \uparrow}(r_1, r_2, s_1, s_2), \ 1 \le s_1 \le r_1, \ 1 \le s_2 \le r_2, \notag\\
&T_{t_1-1, s_2-1}^{-, \uparrow}(p_1-r_1, r_2, t_1, s_2), \ 1 \le t_1 \le p_1-r_1, \ 1 \le s_2 \le r_2,\notag\\
&T_{s_1-1, t_2-1}^{-, \uparrow}(r_1, p_2-r_2, s_1, t_2), \ 1 \le s_1 \le r_1, \ 1 \le t_2 \le p_2-r_2,\notag\\
&T_{t_1-1, t_2-1}^{+, \uparrow}(p_1-r_1, p_2-r_2, t_1, t_2), \ 1 \le t_1 \le p_1-r_1, \ 1 \le t_2 \le p_2-r_2,\notag
\end{align}
are mutually orthogonal primitive idempotents of $Q(r_1, r_2)$.
\end{enumerate}
\end{cor}
\begin{proof}
The corollary follows from  Lemma \ref{lem:eigenvalue}, Proposition \ref{prop:4.1.1}, Proposition \ref{prop:4.2.1},
Proposition \ref{prop:4.3.1} and Lemma \ref{lem:4.5.1}.
\end{proof}

\begin{thm}\label{thm:4.5.4}
The following equality holds{\rm :}
\begin{align}
&\boldsymbol{\mathfrak{g}_{p_1, p_2}} \notag\\ 
=& \bigoplus_{s_1=1}^{p_1}\bigoplus_{s_2=1}^{p_2} X_{p_1, p_2}^{+}(s_1, s_2)
\oplus \bigoplus_{s_1=1}^{p_1}\bigoplus_{s_2=1}^{p_2} X_{p_1, p_2}^{-}(s_1, s_2)\notag\\
&\oplus
\bigoplus_{r_1=1}^{p_1-1} \bigoplus_{s_2=1}^{p_2}
\left\{ \bigoplus_{s_1=1}^{r_1}P_{r_1, p_2}^{\alpha}(s_1, s_2) 
\oplus \bigoplus_{t_1=1}^{p_1-r_1}P_{p_1-r_1, p_2}^{-}(t_1, s_2) \right\} \notag\\
&\oplus
\bigoplus_{r_2=1}^{p_2-1}\bigoplus_{s_1=1}^{p_1}
\left\{ \bigoplus_{s_2=1}^{r_2} P_{p_1, r_2}^{+}(s_1, s_2) \oplus
\bigoplus_{t_2=1}^{p_2-r_2} P_{p_1, p_2-r_2}^{-}(s_1, t_2) \right\}\notag\\
&\oplus 
\bigoplus_{(r_1, r_2) \in I}
\biggl\{\bigoplus_{s_1=1}^{r_1}\bigoplus_{s_2=1}^{r_2}
P_{r_1, r_2}^{+}(s_1, s_2) \oplus \bigoplus_{s_1=1}^{r_1} \bigoplus_{t_2=1}^{p_2-r_2}
P_{r_1, p_2-r_2}^-(s_1, t_2) \notag\\
&\oplus \bigoplus_{t_1=1}^{p_1-r_1} \bigoplus_{s_2=1}^{r_2} P_{p_1-r_1, r_2}^- (t_1, s_2)
\oplus \bigoplus_{t_1=1}^{p_1-r_1}\bigoplus_{t_2=1}^{p_2-r_2} P_{p_1-r_1, p_2-r_2}^+ (t_1, t_2)
\biggr\} \notag
\end{align}
\end{thm}
\begin{proof}
By Corollary \ref{cor:4.5.3}, the right hand side is contained in $\boldsymbol{\mathfrak{g}_{p_1, p_2}}$.
The dimension of the right hand side is 
\begin{align}
2 p_1^2 p_2^2 + 2 p_1^2 p_2^2(p_1-1) + 2 p_1^2 p_2^2(p_2-1) + 2 p_1^2 p_2^2 (p_1-1) (p_2-1) = 2 p_1^3 p_2^3.\notag
\end{align} 
Since $\dim \boldsymbol{\mathfrak{g}_{p_1, p_2}} = 2 p_1^3 p_2^3$,
we have the equalitiy.
\end{proof}

\section{matrix realization}
Let $A$ be a finite-dimensional associative algebra
and $\{P_i \ | \ 1 \le i \le n \}$ the complete set of non-isomorphic indecomposable projective modules of $A$.
Then the map defined by
\begin{align}
A \to \bigoplus_{i=1}^{n} \End (P_i), \ a \mapsto (\rho_i(a)) \notag
\end{align}
is an algebra monomorphism
where $\rho :A \to \End(P_i)$ is representation.

In this section, 
we give a matrix realization of each subalgebra $Q_{r_1, r_2}$ of $\boldsymbol{\mathfrak{g}_{p_1, p_2}}$
by using the above fact.
\subsection{matrix realization of $Q(p_1, p_2)$ and $Q(0, p_2)$}
By Corollary \ref{cor:4.5.3} and Theorem \ref{thm:4.5.4},
we have
\begin{align}
&Q(p_1, p_2) = \bigoplus_{s_1=1}^{p_1} \bigoplus_{s_2=1}^{p_2}X_{p_1, p_2}^+(s_1, s_2),\notag\\
&Q(0, p_2) = \bigoplus_{s_1=1}^{p_1} \bigoplus_{s_2=1}^{p_2}X_{p_1, p_2}^-(s_1, s_2).\notag
\end{align}
By Lemma \ref{lem:eigenvalue}, the action of the basis of $B_{n_1, n_2}^{\alpha}(p_1, p_2, s_1, s_2)$
on the simple module $X_{p_1, p_2}^{\alpha}$ is given by
\begin{align}
&B_{m_1, m_2}^{\alpha}(p_1, p_2, s_1, s_2)
\mathsf{b}_{n_1, n_2}^{\alpha}(p_1, p_2) \label{eq:5.1.1}\\
&\qquad \qquad =\begin{cases}
\mathsf{b}_{m_1, m_2}^{\alpha}(p_1, p_2), & (n_1, n_2) = (s_1-1, s_2-1), \\
0, & \text{otherwise}.
\end{cases}\notag
\end{align}
\begin{thm}\label{thm:5.1.1.1}
The subalgebras $Q(p_1, p_2)$ and $Q(0, p_2)$ are isomorphic to the matrix algebra $M_{p_1p_2}(\mathbb{C})$.
The isomorphism is given by
\begin{align}
B_{m_1-1, m_2-1}^{\alpha}(p_1, p_2, s_1, s_2) \mapsto E(m_1+ p_1( m_2 - 1 ), s_1 + p_1 (s_2- 1)), \notag
\end{align} 
for $1 \le m_i, s_i \le p_i$
where $E(a, b)$ is a matrix unit of $M_{p_1p_2}(\mathbb{C})$.
\end{thm}
\subsection{matrix realization of $Q(r_1, p_2)$ and $Q(p_1, r_2)$}
By Corollary \ref{cor:4.5.3} and Theorem \ref{thm:4.5.4},
we have
\begin{align}
&Q(r_1, p_2) = \bigoplus_{s_1=1}^{r_1} \bigoplus_{s_2=1}^{p_2} P_{r_1, p_2}^{+}(s_1, s_2)
\oplus \bigoplus_{t_1=1}^{p_1-r_1}  \bigoplus_{s_2=1}^{p_2} P_{p_1-r_1, p_2}^{-}(t_1, s_2), \notag\\
&Q(r_1, p_2) = \bigoplus_{s_1=1}^{p_1} \bigoplus_{s_2=1}^{r_2} P_{p_1, r_2}^{+}(s_1, s_2)
\oplus \bigoplus_{s_1=1}^{p_1}  \bigoplus_{t_2=1}^{p_2-r_2} P_{p_1, p_2 -r_2}^{-}(s_1, t_2). \notag
\end{align}
The action of $Q(r_1, p_2)$ on $P_{r_1, p_2}^+$ is give in the following table:
\begin{center}
\small{
\begin{tabular}{c|cccc}
\multicolumn{5}{c}{} \\
$Q(r_1, p_2) \backslash P_{r_1, p_2}^+$    
& $\mathsf{b}_{s_1-1, s_2-1}^{+, \uparrow}$   
& $\mathsf{b}_{t_1-1, s_2-1}^{+, \leftarrow}$
& $\mathsf{b}_{t_1-1, s_2-1}^{+, \rightarrow}$ 
& $\mathsf{b}_{s_1-1, s_2-1}^{+, \downarrow}$ \\ \hline
$B_{n_1, n_2}^{+, \uparrow}(r_1,p_2, s_1, s_2)$        
& $\mathsf{b}_{n_1, n_2}^{+,\uparrow}$
& $0$     & $0$     
& $\mathsf{b}_{n_1, n_2}^{+, \downarrow}$      \\ 
$B_{k_1, n_2}^{+, \leftarrow}(r_1,p_2, s_1, s_2)$        
& $\mathsf{b}_{k_1, n_2}^{+, \leftarrow}$   & $0$     & $0$     & $0$          \\ 
$B_{k_1, n_2}^{+, \rightarrow}(r_1,p_2, s_1, s_2)$        
& $\mathsf{b}_{k_1, b_2}^{+, \rightarrow}$ 
& $0$     & $0$     & $0$          \\ 
$B_{n_1, n_2}^{+, \downarrow}(r_1,p_2, s_1, s_2)$        
& $\mathsf{b}_{n_1, n_2}^{+, \downarrow}$
& $0$     & $0$     & $0$          \\ 
$B_{k_1, n_2}^{-, \uparrow}(p_1-r_1,p_2, t_1, s_2)$      & $0$       
& $\mathsf{b}_{k_1, n_2}^{+, \leftarrow}$ 
&$\mathsf{b}_{k_1, _2}^{+, \rightarrow}$ & $0$          \\ 
$B_{n_1, n_2}^{-, \leftarrow}(p_1-r_1,p_2, t_1, s_2)$      & $0$       & $0$     
& $\mathsf{b}_{n_1, n_2}^{+, \downarrow}$
& $0$          \\ 
$B_{n_1, n_2}^{-, \rightarrow}(p_1-r_1,p_2, t_1, s_2)$      & $0$       &
$\mathsf{b}_{n_1, n_2}^{+, \downarrow}$ & $0$     & $0$          \\ 
$B_{k_1, n_2}^{-, \downarrow}(p_1-s_1,p_2, t_1, s_2)$        & $0$       & $0$     & $0$     & $0$          \\ 
\end{tabular}
}
\end{center}

Remark that the actions which do not appear in the above table are all zero.
Substituting $+, r_1$ with $-, p_1-r_1$ in above table,
we have the actions of the basis of  $Q(r_1, p_2)$ on $P_{p_1-r_1, p_2}^-$.

Let us expand $v \in Q(r_1, p_2)$ as follows:
\begin{align}
&v = \sum_{s_1=1}^{r_1} \sum_{s_2=1}^{p_2} \sum_{n_2=1}^{p_2} \Bigl\{
\sum_{\bullet \in \{\uparrow, \downarrow\}} \sum_{n_1=1}^{r_1} b_{r_1, p_2}^{+, \bullet}
(n_1 + (n_2-1)r_1, s_1 + (s_2-1)r_1) (v) \notag\\
&B_{n_1-1, n_2-1}^{+, \bullet}(r_1, p_2, s_1, s_2) \notag\\
&+ \sum_{\bullet \in \{\leftarrow, \rightarrow\}} \sum_{k_1=1}^{p_1-r_1}
b_{r_1, p_2}^{+, \bullet}(k_1 + (n_2-1)(p_1-r_1), s_1 + (s_2-1)r_1)(v) \notag\\
&B_{k_1-1, n_2-1}^{+, \bullet}(r_1, p_2, s_1, s_2) \Bigr\}
\notag\\
&+\sum_{t_1=1}^{p_1-r_1} \sum_{s_2=1}^{p_2} \sum_{n_2=1}^{p_2} \Bigl\{
\sum_{\bullet \in \{\uparrow, \downarrow\}} \sum_{k_1=1}^{p_1-r_1} b_{p_1-r_1, p_2}^{-, \bullet}
(k_1 + (n_2-1)(p_1-r_1), s_1 + (s_2-1)(p_1-r_1)) (v)\notag\\
&\times B_{k_1-1, n_2-1}^{+, \bullet}(p_1-r_1, p_2, t_1, s_2) \notag\\
&+ \sum_{\bullet \in \{\leftarrow, \rightarrow\}} \sum_{n_1=1}^{r_1}
b_{p_1-r_1, p_2}^{-, \bullet}(n_1 + (n_2-1)r_1, s_1 + (s_2-1)r_1)(v) \notag\\
& \times B_{n_1-1, n_2-1}^{, \bullet}
(p_1-r_1, p_2, t_1, s_2) \Bigr\}. \notag
\end{align}
We define
\begin{align}
&B_{r_1, p_2}^{+, \bullet}(v) = [b_{r_1, p_2}^{+, \bullet}
(m_1, m_2) (v)] \in M_{r_1 p_2}(\mathbb{C}), \notag\\
&B_{r_1, p_2}^{+, \circ}(v) = [b_{r_1, p_2}^{+, \bullet}(m_1, m_2)(v)] \in M_{(p_1-r_1)p_2, r_1 p_2}(\mathbb{C}), \notag\\
&B_{p_1-r_1, p_2}^{-, \bullet} = [b_{p_1-r_1, p_2}^{-, \bullet}(m_1, m_2)(v)] \in M_{(p_1-r_1)p_2}(\mathbb{C}), \notag\\
&B_{p_1-r_1, p_2}^{-, \circ} = [b_{p_1-r_1, p_2}^{-, \circ}(m_1, m_2)(v)] \in M_{r_1 p_2, (p_1-r_1)p_2}(\mathbb{C})\notag
\end{align}
where $\bullet \in \{\uparrow, \downarrow\}$ and $\circ \in \{\leftarrow, \rightarrow\}$.
Then, by the action of $Q(r_1, p_2)$ on the projective modules $P_{r_1, p_2}^+$ and $P_{p_1-r_1, p_2}^-$,
we have the following result.
\begin{thm}\label{thm:5.2.1} 
For $1 \le r_1 \le p_1-1$, $Q(r_1, p_2)$ is isomorphic to the subalgebra of 
$M_{2p_1p_2}(\mathbb{C}) \oplus M_{2p_1p_2}(\mathbb{C})$ with the shape:
\begin{align}
\Bigl(  
\begin{bmatrix}
B^{+, \uparrow}_{r_1, p_2} & 0 & 0 & 0 \\
B^{+, \leftarrow}_{r_1, p_2} & B^{-, \uparrow}_{p_1-r_1, p_2} & 0 & 0 \\
B^{+, \rightarrow}_{r_1, p_2} & 0 & B^{-, \uparrow}_{p_1-r_1, p_2} & 0 \\
B^{+, \downarrow}_{r_1, p_2} & B^{-, \rightarrow}_{p_1-r_1, p_2} & B^{-, \leftarrow}_{p_1-r_1, p_2} 
& B^{+, \uparrow}_{r_1, p_2}
\end{bmatrix},
\begin{bmatrix}
B_{p_1-r_1, p_2}^{-, \uparrow} & 0 & 0 & 0 \\
B_{p_1-r_1, p_2}^{-, \leftarrow} & B_{r_1, p_2}^{+, \uparrow} & 0 & 0 \\
B_{p_1-r_1, p_2}^{-, \rightarrow} & 0 & B_{r_1, p_2}^{+, \uparrow} & 0 \\
B_{p_1-r_1, p_2}^{-, \downarrow} & B_{r_1, p_2}^{+, \rightarrow} & B_{r_1, p_2}^{-, \leftarrow} 
& B_{p_1-r_1, p_2}^{-, \uparrow}
\end{bmatrix}        \Bigr). \notag
\end{align}
\end{thm}

By using similar argument for the case $Q(r_1, p_2)$, we have the following:
\begin{thm}\label{thm:5.2.2}
For $1 \le r_1 \le p_1-1$, the algebra $Q(r_1, p_2)$ is isomorphic to the subalgebra of 
$M_{2p_1p_2}(\mathbb{C}) \oplus M_{2p_1p_2}(\mathbb{C})$:
\begin{align}
\Bigl(  
\begin{bmatrix}
B^{+, \uparrow}_{p_1, r_2} & 0 & 0 & 0 \\
B^{+, \leftarrow}_{p_1, r_2} & B^{-, \uparrow}_{p_1, p_2-r_2} & 0 & 0 \\
B^{+, \rightarrow}_{p_1, r_2} & 0 & B^{-, \uparrow}_{p_1, p_2-r_2} & 0 \\
B^{+, \downarrow}_{p_1, r_2} & B^{-, \rightarrow}_{p_1, p_2 - r_2} & 
B^{-, \leftarrow}_{p_1, p_2 - r_2} & B^{+, \uparrow}_{p_1, r_2}
\end{bmatrix},
\begin{bmatrix}
B^{-, \uparrow}_{p_1, p_2-r_2} & 0 & 0 & 0 \\
B^{-, \leftarrow}_{p_1, p_2-r_2} & B^{+, \uparrow}_{p_1, r_2} & 0 & 0 \\
B^{-, \rightarrow}_{p_1, p_2-r_2} & 0 & B^{+, \uparrow}_{p_1, r_2} & 0 \\
B^{-}_{p_1, p_2-r_2} & B^{+, \rightarrow}_{p_1, r_2} & B^{+, \leftarrow}_{p_1, r_2} & 
B^{-, \uparrow}_{p_1, p_2-r_2}
\end{bmatrix}        \Bigr). \notag
\end{align}
\end{thm}
\subsection{matrix realization of $Q(r_1, r_2)$}
By Corollary \ref{cor:4.5.3} and Theorem \ref{thm:4.5.4},
we have
\begin{align}
&Q(r_1, r_2) = \bigoplus_{s_1=1}^{r_1}\bigoplus_{s_2=1}^{r_2} P_{r_1, r_2}^+(s_1, s_2)
\oplus \bigoplus_{s_1=1}^{r_1}\bigoplus_{t_2=1}^{p_2-r_2} P_{r_1, p_2-r_2}^-(s_1, t_2)\notag\\
&\qquad \qquad \oplus \bigoplus_{t_1=1}^{p_1-r_1}\bigoplus_{s_2=1}^{r_2} P_{p_1-r_1, r_2}^-(t_1, s_2)
\oplus \bigoplus_{t_1=1}^{p_1-r_1}\bigoplus_{t_2=1}^{p_2-r_2} P_{p_1-r_1, p_2-r_2}^+(t_1, t_2).\notag
\end{align}
Since $Q(r_1, r_2)$ has  four non-isomorphic indecomposable projective modules,
there is an algebra monomorphism
\begin{align}
Q(r_1, r_2) \to \End_{\mathbb{C}}(P_{r_1, r_2}^+) \oplus \End_{\mathbb{C}}(P_{r_1, p_2-r_2}^-)
\oplus \End_{\mathbb{C}}(P_{p_1-r_1, r_2}^-) \oplus \End_{\mathbb{C}}(P_{p_1-r_1, p_2-r_2}^+).\notag
\end{align}

By Lemma \ref{lem:eigenvalue}, the definition of the basis of $Q(r_1, r_2)$
and the structure of projective modules, we can determine the action of $Q(r_1, r_2)$ on the projective module as in Table 1
(the actions do not appear in Table 1 are zero).

\newpage
\begin{table}
\caption{The actions of the basis of $Q(r_1, r_2)$ on $P_{r_1, r_2}^+$.}
\begin{center}
\small{
\begin{tabular}{c|cccc}
\multicolumn{5}{c}{} \\
$P_{r_1, r_2}^{+}(s_1, s_2) \backslash P_{r_1, r_2}^+$    
& $\mathsf{t}_{s_1-1, s_2-1}^{+, \uparrow}$   
& $\mathsf{t}_{s_1-1, s_2-1}^{+, \downarrow}$
& $\mathsf{b}_{s_1-1, s_2-1}^{+, \uparrow}$ 
& $\mathsf{b}_{s_1-1, s_2-1}^{+, \downarrow}$ \\ \hline
$T_{m_1, m_2}^{+, \uparrow}(r_1,r_2, s_1, s_2)$        
& $\mathsf{t}_{m_1, m_2}^{+,\uparrow}$
& $\mathsf{t}_{m_1, m_2}^{+,\downarrow}$     
&$\mathsf{b}_{m_1, m_2}^{+,\uparrow}$     
&$\mathsf{b}_{m_1, m_2}^{+, \downarrow}$      \\ 
$T_{k_1, n_2}^{+, \bullet}(r_1,r_2, s_1, s_2)$
& $\mathsf{t}_{k_1, n_2}^{+, \bullet}$
& $0$     
& $\mathsf{b}_{k_1, m_2}^{+, \bullet}$     
& $0$          \\ 
$X_{m_1, k_2}^{+, \uparrow}(r_1,r_2, s_1, s_2)$        
& $\mathsf{x}_{m_1, k_2}^{+, \uparrow}$ 
& $\mathsf{x}_{m_1, k_2}^{+, \downarrow}$     
& $0$     
& $0$          \\ 
$X_{m_1, k_2}^{+, \bullet}(r_1,r_2, s_1, s_2)$        
&$\mathsf{x}_{m_1, k_2}^{+, \bullet}$
& $0$     
& $0$     
& $0$          \\ 
$B_{m_1, m_2}^{+, \uparrow}(r_1,r_2, s_1, s_2)$   
& $\mathsf{b}_{m_1, m_2}^{+, \uparrow}$       
& $\mathsf{b}_{m_1, m_2}^{+, \downarrow}$ 
&$0$ 
& $0$          \\ 
$B_{k_1, n_2}^{+, \bullet}(r_1,r_2, s_1, s_2)$      
&$\mathsf{b}_{k_1, m_2}^{+, \bullet}$       
& $0$     
& $0$
& $0$          \\ 
\end{tabular}
\begin{tabular}{c|cccc}
\multicolumn{5}{c}{} \\
$P_{p_1-r_1, r_2}^{-}(t_1, s_2) \backslash P_{r_1, r_2}^+$    
& $\mathsf{t}_{t_1-1, s_2-1}^{+, \leftarrow}$   
& $\mathsf{t}_{t_1-1, s_2-1}^{+, \rightarrow}$
& $\mathsf{b}_{t_1-1, s_2-1}^{+, \leftarrow}$ 
& $\mathsf{b}_{t_1-1, s_2-1}^{+, \rightarrow}$ \\ \hline
$T_{k_1, m_2}^{-, \uparrow}(p_1-r_1,r_2, t_1, s_2)$        
& $\mathsf{t}_{k_1, m_2}^{+,\leftarrow}$
& $\mathsf{t}_{k_1, m_2}^{+,\rightarrow}$     
&$\mathsf{b}_{k_1, m_2}^{+,\leftarrow}$     
&$\mathsf{b}_{k_1, m_2}^{+, \rightarrow}$      \\ 
$T_{m_1, m_2}^{-, \leftarrow}(p_1-r_1,r_2, t_1, s_2)$
& $0$
& $\mathsf{t}_{m_1, m_2}^{+, \downarrow}$     
& $0$     
& $\mathsf{b}_{m_1, m_2}^{+, \downarrow}$          \\ 
$T_{m_1, m_2}^{-, \rightarrow}(p_1-r_1,r_2, t_1, s_2)$
& $\mathsf{t}_{m_1, m_2}^{+, \downarrow}$
& $0$     
& $\mathsf{b}_{m_1, m_2}^{+, \downarrow}$     
& $0$          \\ 
$X_{k_1, k_2}^{-, \uparrow}(p_1-r_1,r_2, t_1, s_2)$        
& $\mathsf{x}_{k_1, k_2}^{+, \leftarrow}$ 
& $\mathsf{x}_{k_1, k_2}^{+, \rightarrow}$     
& $0$     
& $0$          \\ 
$X_{m_1, k_2}^{-, \leftarrow}(p_1-r_1,r_2, t_1, s_2)$ 
&$0$
& $\mathsf{x}_{m_1, k_2}^{+, \downarrow}$     
& $0$     
& $0$          \\ 
$X_{m_1, k_2}^{-, \rightarrow}(p_1-r_1,r_2, t_1, s_2)$ 
& $\mathsf{x}_{m_1, k_2}^{+, \downarrow}$       
& $0$ 
&$0$ 
& $0$          \\ 
 \end{tabular}
\begin{tabular}{c|cccc}
\multicolumn{5}{c}{} \\
$P_{r_1, p_2-r_2}^{-}(s_1, t_2) \backslash P_{r_1, r_2}^+$    
& $\mathsf{l}_{s_1-1, t_2-1}^{+, \uparrow}$   
& $\mathsf{l}_{s_1-1, t_2-1}^{+, \downarrow}$
& $\mathsf{r}_{s_1-1, t_2-1}^{+, \uparrow}$ 
& $\mathsf{r}_{s_1-1, t_2-1}^{+, \downarrow}$ \\ \hline
$T_{m_1, k_2}^{-, \uparrow}(r_1, p_2-r_2, s_1, t_2)$        
& $\mathsf{l}_{m_1, k_2}^{+,\uparrow}$
& $\mathsf{l}_{m_1, k_2}^{+,\downarrow}$     
&$\mathsf{r}_{m_1, k_2}^{+,\uparrow}$     
&$\mathsf{r}_{m_1, k_2}^{+, \downarrow}$      \\ 
$T_{k_1, k_2}^{-, \bullet}(r_1, p_2-r_2, s_1, t_2)$
& $\mathsf{l}_{k_1, k_2}^{+, \bullet}$
& $0$     
& $\mathsf{r}_{k_1, k_2}^{+, \bullet}$     
& $0$          \\ 
$L_{m_1, m_2}^{-, \uparrow}(r_1, p_2-r_2, s_1, t_2)$        
& $0$ 
& $0$     
& $\mathsf{b}_{m_1, m_2}^{+, \uparrow}$     
& $\mathsf{b}_{m_1, m_2}^{+, \downarrow}$          \\ 
$L_{k_1, m_2}^{-, \bullet}(r_1, p_2-r_2, s_1, t_2)$ 
&$0$
& $0$     
& $\mathsf{b}_{k_1, m_2}^{+, \bullet}$     
& $0$          \\ 
$R_{m_1, m_2}^{-, \uparrow}(r_1, p_2-r_2, s_1, t_2)$        
& $\mathsf{b}_{m_1, m_2}^{+, \uparrow}$     
& $\mathsf{b}_{m_1, m_2}^{+, \downarrow}$
& $0$ 
& $0$     \\ 
$R_{k_1, m_2}^{-, \bullet}(r_1, p_2-r_2, s_1, t_2)$ 
&$\mathsf{b}_{k_1, m_2}^{+, \bullet}$
& $0$     
& $0$     
& $0$ \\
 \end{tabular}
\begin{tabular}{c|cccc}
\multicolumn{5}{c}{} \\
$P_{p_1-r_1, p_2-r_2}^{+}(t_1, t_2) \backslash P_{r_1, r_2}^+$    
& $\mathsf{l}_{t_1-1, t_2-1}^{+, \leftarrow}$   
& $\mathsf{l}_{t_1-1, t_2-1}^{+, \rightarrow}$
& $\mathsf{r}_{t_1-1, r_2-1}^{+, \leftarrow}$ 
& $\mathsf{r}_{t_1-1, t_2-1}^{+, \rightarrow}$ \\ \hline
$T_{k_1, k_2}^{+, \uparrow}(p_1-r_1, p_2-r_2, t_1, t_2)$        
& $\mathsf{l}_{k_1, k_2}^{+,\leftarrow}$
& $\mathsf{l}_{k_1, k_2}^{+,\rightarrow}$     
&$\mathsf{r}_{k_1, k_2}^{+,\leftarrow}$     
&$\mathsf{r}_{k_1, k_2}^{+, \rightarrow}$      \\ 
$T_{m_1, k_2}^{+, \leftarrow}(p_1-r_1, p_2-r_2, t_1, t_2)$
& $0$
& $\mathsf{l}_{m_1, k_2}^{+, \downarrow}$     
& $0$     
& $\mathsf{r}_{m_1, k_2}^{+, \downarrow}$          \\ 
$T_{m_1, k_2}^{+, \rightarrow}(p_1-r_1, p_2-r_2, t_1, t_2)$
& $\mathsf{l}_{m_1, k_2}^{+, \downarrow}$
& $0$     
& $\mathsf{r}_{m_1, k_2}^{+, \downarrow}$     
& $0$          \\ 
$L_{k_1, m_2}^{+, \uparrow}(p_1-r_1, p_2-r_2, t_1, t_2)$        
& $0$ 
& $0$     
& $\mathsf{b}_{k_1, m_2}^{+, \leftarrow}$     
& $\mathsf{b}_{k_1, m_2}^{+, \rightarrow}$          \\ 
$L_{m_1, m_2}^{+, \leftarrow}(p_1-r_1, p_2-r_2, t_1, t_2)$ 
&$0$
& $0$     
& $0$     
& $\mathsf{b}_{m_1, m_2}^{+, \downarrow}$          \\ 
$L_{m_1, m_2}^{+, \rightarrow}(p_1-r_1, p_2-r_2, t_1, t_2)$ 
& $0$       
& $0$ 
&$\mathsf{b}_{m_1, m_2}^{+, \downarrow}$ 
& $0$       \\   
$R_{k_1, m_2}^{+, \uparrow}(p_1-r_1, p_2-r_2, t_1, t_2)$
& $\mathsf{b}_{k_1, m_2}^{+, \leftarrow}$     
& $\mathsf{b}_{k_1, m_2}^{+, \rightarrow}$ 
&$0$
&$0$ \\
$R_{m_1, m_2}^{+, \leftarrow}(p_1-r_1, p_2-r_2, t_1, t_2)$ 
&$0$
& $\mathsf{b}_{m_1, m_2}^{+, \downarrow}$     
& $0$     
& $0$          \\ 
$R_{m_1, m_2}^{+, \rightarrow}(p_1-r_1, p_2-r_2, t_1, t_2)$ 
& $\mathsf{b}_{m_1, m_2}^{+, \downarrow}$       
& $0$ 
&$0$ 
& $0$   \\
 \end{tabular}
}
\end{center}
\end{table}

Let $\rho_{r_1, r_2}^{+}$ be the representation of $Q(r_1, r_2)$ on $P_{r_1, r_2}^+$.
Let us expand $v \in Q(r_1, r_2)$ as
\begin{align}
v = v^+_{r_1, r_2} + v_{p_1-r_1, r_2}^- + v_{r_1, p_2-r_2}^- + v_{p_1-r_1, p_2-r_2}^+ \notag
\end{align}
where
\begin{align}
&v_{r_1, r_2}^+ = \sum_{s_1=1}^{r_1}\sum_{s_2=1}^{r_2} \notag\\
&\Bigl\{ \sum_{\bullet \in \{\uparrow, \downarrow\}} \sum_{n_1=1}^{r_1} \sum_{n_2=1}^{r_2} \Bigl(
t_{r_1, r_2}^{+, \bullet}(n_1 + r_1 (n_2-1), s_1 + (s_2 -1)r_2) (v) T_{n_1-1, n_2-1}^{+, \bullet} (r_1, r_2, s_1, s_2)\notag\\
&+ b_{r_1, r_2}^{+, \bullet}(n_1 + r_1 (n_2-1), s_1 + (s_2 -1)r_1) (v)T_{n_1-1, n_2-1}^{+, \bullet} (r_1, r_2, s_1, s_2)
\Bigr) \notag\\
&+ \sum_{\bullet \in \{\leftarrow, \rightarrow\}} \sum_{k_1=1}^{p_1-r_1} \sum_{n_2=1}^{r_2}\notag\\
& \Bigl(
t_{r_1, r_2}^{+, \bullet}(k_1 + (p_1-r_1) (n_2-1), s_1 + (s_2 -1)r_1) (v) T_{k_1-1, n_2-1}^{+, \bullet} (r_1, r_2, s_1, s_2)\notag\\
&+ b_{r_1, r_2}^{+, \bullet}(k_1 + (p_1-r_1) (n_2-1), s_1 + (s_2 -1)r_1) (v)B_{k_1-1, n_2-1}^{+, \bullet} (r_1, r_2, s_1, s_2)
\Bigr)\notag\\
&\sum_{\bullet \in \{\uparrow, \downarrow\}} \sum_{n_1=1}^{r_1} \sum_{k_2=1}^{p_2-r_2} \Bigl(
r_{r_1, r_2}^{+, \bullet}(n_1 + r_1 (k_2-1), s_1 + (s_2 -1)r_1) (v) R_{n_1-1, k_2-1}^{+, \bullet} (r_1, r_2, s_1, s_2)\notag\\
&+ b_{r_1, r_2}^{+, \bullet}(n_1 + r_1 (k_2-1), s_1 + (s_2 -1)r_1) (v)L_{n_1-1, k_2-1}^{+, \bullet} (r_1, r_2, s_1, s_2)
\Bigr)\notag\\
&+\sum_{\bullet \in \{\leftarrow, \rightarrow\}} \sum_{k_1=1}^{p_1-r_1} 
 \sum_{k_2=1}^{p_2-r_2} \notag\\
& \Bigl(
r_{r_1, r_2}^{+, \bullet}(k_1 + (p_1-r_1) (k_2-1), s_1 + (s_2 -1)r_1) (v) R_{k_1-1, k_2-1}^{+, \bullet} (r_1, r_2, s_1, s_2)\notag\\
&+ b_{r_1, r_2}^{+, \bullet}(k_1 + (p_1-r_1) (k_2-1), s_1 + (s_2 -1)r_1) (v)L_{k_1-1, k_2-1}^{+, \bullet} (r_1, r_2, s_1, s_2)
\Bigr)
\Bigr\}\notag
\end{align}
and set
\begin{align}
&T_{r_1, r_2}^{+, \bullet}(v) = [t_{r_1, r_2}^{+, \bullet}(m_1, m_2) (v)], 
B_{r_1, r_2}^{+, \bullet}(v) = [b_{r_1, r_2}^{+, \bullet}(m_1, m_2) (v)] 
\in M_{r_1 r_2} (\mathbb{C}),  \notag\\
&T_{r_1, r_2}^{+, \circ}(v) = [t_{r_1, r_2}^{+, \circ}(m_1, m_2) (v)], 
B_{r_1, r_2}^{+, \circ} (v) = [b_{r_1, r_2}^{+, \circ}(m_1, m_2) (v)] \in M_{(p_1-r_1) r_2, r_1 r_2} 
(\mathbb{C}),  \notag\\
&R_{r_1, r_2}^{+, \bullet}(v) = [r_{r_1, r_2}^{+, \bullet}(m_1, m_2) (v)], 
L_{r_1, r_2}^{+, \bullet} = [l_{r_1, r_2}^{+, \bullet}(m_1, m_2) (v)] \in M_{r_1 (p_2-r_2), r_1 r_2} (\mathbb{C}), 
 \notag\\
&R_{r_1, r_2}^{+, \circ}(v) = [r_{r_1, r_2}^{+, \circ}(m_1, m_2) (v)], 
L_{r_1, r_2}^{+, \circ} (v) = [l_{r_1, r_2}^{+, \circ}(m_1, m_2) (v)] \in M_{(p_1-r_1) (p_2-r_2), r_1 r_2} (\mathbb{C}), 
 \notag
\end{align}
for $\bullet \in \{\uparrow, \downarrow\}$ and $\circ \in \{\leftarrow, \rightarrow\}$.
Substituting $(+, r_1, r_2)$ with $(-, p_1-r_1, r_2)$, $(-, r_1, p_2-r_2)$ and $(+, p_1-r_1, p_2-r_2)$,
we can define $v_{p_1-r_1, r_2}^-$, $v_{r_1, p_2-r_2}^-$ and $v_{p_1-r_1, p_2-r_2}^+$ and corresponding matrices.
By Table 1,
the linear map $\rho_{r_1, r_2}^+(v)$ for $v \in Q(r_1, r_2)$ can be expressed as follows
\begin{align}
\begin{bmatrix}
T_{r_1, r_2}^+(v) & 0 & 0 & 0 \\
L_{r_1, r_2}^+(v) & T_{r_1, p_2-r_2}^-(v) & 0 & 0 \\
R_{r_1, r_2}^+(v) & 0 & T_{r_1, p_2-r_2}^- (v)& 0 \\
B_{r_1, r_2}^+(v) & R_{r_1, p_2-r_2}^-(v) & L_{r_1, p_2-r_2}^-(v) & T_{r_1, r_2}^+(v)
\end{bmatrix}\notag
\end{align}
with
\begin{align}
X_{r_1, r_2}^{+}(v) = \begin{bmatrix}
X_{r_1, r_2}^{+, \uparrow}(v) & 0 & 0 & 0 \\
X_{r_1, r_2}^{+, \leftarrow}(v) & X_{p_1-r_1, r_2}^{-, \uparrow}(v) & 0 & 0\\
X_{r_1, r_2}^{+, \rightarrow}(v) & 0 &  X_{p_1-r_1, r_2}^{-, \uparrow}(v) & 0\\
X_{r_1, r_2}^{+, \downarrow}(v) & X_{p_1-r_1, r_2}^{-, \rightarrow}(v) & X_{p_1-r_1, r_2}^{-, \leftarrow}(v) & 
X_{r_1, r_2}^{+, \uparrow}(v)
\end{bmatrix}, \notag
\end{align}
where
$X \in \{T, L, R, B\}$.

Then we have the following result.

\begin{thm}\label{thm:5.3.1}
The algebra of $Q(r_1, r_2)$ with $(r_1, r_2) \in I$ is isomorphic to the subalgebra of 
$M_{4 p_1 p_2}(\mathbb{C}) \oplus M_{4 p_1 p_2}(\mathbb{C}) \oplus
M_{4 p_1 p_2}(\mathbb{C}) \oplus M_{4 p_1 p_2}(\mathbb{C})$ given by
\begin{align}
&\Bigl( \begin{bmatrix}
T_{r_1, r_2}^+ & 0 & 0 & 0 \\
L_{r_1, r_2}^+ & T_{r_1, p_2-r_2}^- & 0 & 0 \\
R_{r_1, r_2}^+ & 0 & T_{r_1, p_2-r_2}^- & 0 \\
B_{r_1, r_2}^+ & R_{r_1, p_2-r_2}^- & L_{r_1, p_2-r_2}^- & T_{r_1, r_2}^+
\end{bmatrix},
\begin{bmatrix}
T_{p_1-r_1, r_2}^- & 0 & 0 & 0 \\
L_{p_1-r_1, r_2}^- & T_{p_1-r_1, p_2-r_2}^+ & 0 & 0\\
R_{p_1-r_1, r_2}^- & 0 & T_{p_1-r_1, p_2-r_2}^+ & 0\\
B_{p_1-r_1, r_2}^- & R_{p_1-r_1, p_2-r_2}^+ & L_{p_1-r_1, p_2-r_2}^+ & T_{p_1-r_1, r_2}^-
\end{bmatrix},\notag\\
&\begin{bmatrix}
T_{r_1, p_2-r_2}^- & 0 & 0 & 0 \\
L_{r_1, p_2-r_2}^- & T_{r_1, r_2}^+ & 0 & 0\\
R_{r_1, p_2-r_2}^- & 0 & T_{r_1, r_2}^+ & 0 \\
B_{r_1, p_2-r_2}^- & R_{r_1, r_2}^+ & L_{r_1, r_2}^+ & T_{r_1, p_2-r_2}
\end{bmatrix},
\begin{bmatrix}
T_{p_1-r_1, p_2-r_2}^+ & 0 & 0 & 0 \\
L_{p_1-r_1, p_2-r_2}^+ & T_{p_1-r_1, r_2}^- & 0 & 0 \\
R_{p_1-r_1, p_2-r_2}^+ & 0 & T_{p_1-r_1, r_2}^- & 0 \\
B_{p_1-r_1, p_2-r_2}^+ & R_{p_1-r_1, r_2}^- & L_{p_1-r_1, r_2}^- & T_{p_1-r_1, p_2-r_2}^+
\end{bmatrix}
\Bigr) \notag
\end{align}
where
\begin{align}
X_{r_1, r_2}^{\alpha} = \begin{bmatrix}
X_{r_1, r_2}^{\alpha, \uparrow} & 0 & 0 & 0 \\
X_{r_1, r_2}^{\alpha, \leftarrow} & X_{p_1-r_1, r_2}^{-\alpha, \uparrow} & 0 & 0\\
X_{r_1, r_2}^{\alpha, \rightarrow} & 0 &  X_{p_1-r_1, r_2}^{-\alpha, \uparrow} & 0\\
X_{r_1, r_2}^{\alpha, \downarrow} & X_{p_1-r_1, r_2}^{-\alpha, \rightarrow} & X_{p_1-r_1, r_2}^{-\alpha, \leftarrow} & 
X_{r_1, r_2}^{\alpha, \uparrow}
\end{bmatrix}\notag
\end{align}
and $X = T, R, L, B$.
\end{thm}
\section{Symmetric linear functions on $\boldsymbol{\mathfrak{g}_{p_1, p_2}}$}
\subsection{The center}
The structure of the center
of $\boldsymbol{\mathfrak{g}_{p_1, p_2}}$ is described in \cite{FGST4}.
Recall that $\boldsymbol{\mathfrak{g}_{p_1, p_2}}$ decomposes into subalgebras as
\begin{align}
\boldsymbol{\mathfrak{g}_{p_1, p_2}}= \bigoplus_{(r_1, r_2) \in I} Q(r_1, r_2) \oplus
\bigoplus_{r_1=1}^{p_1-1} Q(r_1, p_2) \oplus \bigoplus_{r_2=1}^{p_2-1} Q(p_1, r_2) 
\oplus Q(p_1, p_2) \oplus Q(0, p_2). \notag 
\end{align}
\begin{prop}[\cite{FGST4}]\label{prop:6.1.1}
\begin{align}
\dim Z(Q(r_1, r_2)) = \begin{cases}
9, & (r_1, r_2) \in I,\\
3, & 1 \le r_1 \le p_1-1, r_2=p_2,\\
3, & r_1=p_1, 1 \le r_2 \le p_2-1, \\
1, & (r_1, r_2) = (p_1, p_2).
\end{cases}\notag
\end{align}
\end{prop}
By Proposition \ref{prop:center-yaho} we can see that the space of symmetric linear functions on $\boldsymbol{\mathfrak{g}_{p_1, p_2}}$
is $\frac{1}{2} (3p_1-1) (3p_2-1)$-dimensional.

The action of the center on the projective modules is also determined in \cite{FGST4}.
For each $Q(r_1, r_2)$, there is a central idempotent $e(r_1, r_2)$ which acts 
as an  identity on $Q(r_1, r_2)$
and as zero on $Q(s_1, s_2)$ for $(r_1, r_2) \neq (s_1, s_2)$.
For $Q(r_1, r_2)$ with $(r_1, r_2) \in I$, there are eight central elements as follows;
\begin{align}
&v^{\nearrow}(r_1, r_2), \ v^{\nearrow}(r_1, r_2) \mathsf{t}^{\bullet}_{n_1, n_2} = \mathsf{b}^{\bullet}_{n_1, n_2}, \ 
\bullet \in \{\uparrow, \downarrow, \leftarrow, \rightarrow\}
\ \text{in} \ P^{+}_{r_1, r_2} \ \text{and} \ P_{p_1-r_1, r_2}^-, \notag\\
&v^{\nwarrow}(r_1, r_2), \ v^{\nwarrow}(r_1, r_2) \mathsf{x}_{n_1, n_2}^{\uparrow} = \mathsf{x}_{n_1, n_2}^{\downarrow}, \ 
\mathsf{x} \in \{\mathsf{t}, \mathsf{b}, \mathsf{l}, \mathsf{r}\} \ 
\text{in} \ P_{r_1, r_2}^+ \ \text{and} \ P_{r_1, p_2-r_2}^-, \notag\\
&v^{\swarrow}(r_1, r_2), \ 
v^{\swarrow}(r_1, r_2) \mathsf{t}_{n_1, n_2}^{\bullet} = \mathsf{b}_{n_1, n_2}^{\bullet}, \ 
\bullet \in \{\uparrow, \downarrow,\leftarrow, \rightarrow\} \ \text{in} \ 
P_{r_1, p_2-r_2}^- \ \text{and} \ P_{p_1-r_1, p_2-r_2}^+,\notag\\
&v^{\searrow}(r_1, r_2), \ 
v^{\searrow}(r_1, r_2) \mathsf{x}_{n_1, n_2}^{\uparrow} = \mathsf{x}_{n_1, n_2}^{\downarrow}, \
\mathsf{x} \in \{\mathsf{t}, \mathsf{b}, \mathsf{l}, \mathsf{r}\}
\ \text{in} \ P_{p_1-r_1, r_2}^{-} \ \text{and} \ P_{p_1-r_1, p_2-r_2}^+,\notag\\
&w^{\uparrow}(r_1, r_2), \
w^{\uparrow}(r_1, r_2) \mathsf{t}_{n_1, n_2}^{\uparrow}= \mathsf{b}_{n_1, n_2}^{\downarrow} \ 
\text{in} \ P_{r_1, r_2}^+, \notag\\
&w^{\rightarrow}(r_1, r_2), \ w^{\rightarrow}(r_1, r_2) 
\mathsf{t}_{n_1, n_2}^{\uparrow} = \mathsf{b}_{n_1, n_2}^{\downarrow} \ 
\text{in} \ P_{p_1-r_1, r_2}^-,\notag\\
&w^{\leftarrow}(r_1, r_2), \
w^{\leftarrow} (r_1, r_2) \mathsf{t}_{n_1, n_2}^{\uparrow} = \mathsf{b}_{n_1, n_2}^{\downarrow} \ 
\text{in} \ P_{r_1, p_2-r_2}^-,\notag\\
&w^{\downarrow}(r_1, r_2), \ 
w^{\downarrow}(r_1, r_2) \mathsf{t}_{n_1, n_2}^{\uparrow} = \mathsf{b}_{n_1, n_2}^{\downarrow} \ 
\text{in} \ P_{p_1-r_1, p_2-r_2}^+.\notag 
\end{align}
There are  two central elements $w^{\pm}(r_1, p_2)$ in $Q(r_1, p_2)$
and we have
\begin{align}
&w^{+}(r_1, p_2) \mathsf{b}_{n_1, n_2}^{\uparrow} = \mathsf{b}^{\downarrow}_{n_1, n_2} \ \text{in} \ P_{r_1, p_2}^+, \notag\\
&w^{-}(r_1, p_2) \mathsf{b}_{n_1, n_2}^{\uparrow} = \mathsf{b}^{\downarrow}_{n_1, n_2} \ \text{in} \ P_{p_1-r_1, p_2}^-.\notag
\end{align}
Similarly we can determine the actions of the central elements $w^{\pm}(p_1, r_2)$ in $Q(p_1, r_2)$:
\begin{align}
&w^{+}(p_1, r_2) \mathsf{b}_{n_1, n_2}^{\uparrow} = \mathsf{b}^{\downarrow}_{n_1, n_2}
 \ \text{in} \ P_{p_1, r_2}^+, \notag\\
&w^{-}(p_1, r_2) \mathsf{b}_{n_1, n_2}^{\uparrow} = \mathsf{b}^{\downarrow}_{n_1, n_2}
 \ \text{in} \ P_{p_1, p_2 - r_2}^-.\notag
\end{align} 
\subsection{Symmetric linear functions}
By Theorem \ref{thm:5.1.1.1}, the subalgebras $Q(p_1, p_2)$ and $Q(0, p_2)$ are isomorphic to the matrix algebra
$M_{p_1p_2}(\mathbb{C})$.
Thus we can define the linear functions $\tau_{p_1, p_2}$ on  $Q(p_1, p_2)$
and $\tau_{0, p_2}$ on $Q(0, p_2)$ as the trace on  $M_{p_1p_2}(\mathbb{C})$.
It is clear that the linear functions $\tau_{p_1, p_2}$ and $\tau_{0, p_2}$ are symmetric.

By Theorem \ref{thm:5.2.1},
we define the linear functions on $Q(r_1, p_2)$ by
\begin{align}
&\tau^+_{r_1, p_2}(v) = \trace B_{r_1, p_2}^{+, \uparrow}(v), \
\tau^-_{r_1, p_2}(v) = \trace B_{p_1-r_1, p_2}^{-, \uparrow}(v), \notag\\
&\chi_{r_1, p_2}(v) = \trace B_{r_1, p_2}^{+, \downarrow}(v) + \trace B_{p_1-r_1, p_2}^{-, \downarrow}(v) \notag
\end{align}
and the linear functions on $Q(p_1, r_2)$ by
\begin{align}
&\tau^+_{p_1, r_2}(v) = \trace B_{p_1, r_2}^{+, \uparrow}(v), \
\tau^-_{p_1, r_2}(v) = \trace B_{p_1, p_2-r_2}^{-, \uparrow}(v), \notag\\
&\chi_{p_1, r_2}(v) = \trace B_{p_1, r_2}^{+, \downarrow}(v) + \trace B_{p_1, p_2-r_2}^{-, \downarrow}(v) \notag
\end{align}
for $1 \le r_1 \le p_1-1$ and $1 \le r_2 \le p_2-1$.
These linear functions are symmetric (see \cite{Ar}).

We also define the linear functions on $Q(r_1, r_2)$ for $(r_1, r_2) \in I$:
\begin{align}
&\tau^{\uparrow}_{r_1, r_2}(v) = \trace T_{r_1, r_2}^{+, \uparrow}(v),\
\tau^{\downarrow}_{r_1, r_2}(v) = \trace T_{p_1-r_1, p_2-r_2}^{+, \uparrow}(v), \notag\\
&\tau^{\rightarrow}_{r_1, r_2}(v) = \trace T_{p_1-r_1, r_2}^{-, \uparrow}(v), \
\tau^{\leftarrow}_{r_1, r_2}(v) = \trace T_{r_1, p_2-r_2}^{-, \uparrow}(v), \notag\\
&\upsilon^{\nearrow}_{r_1, r_2}(v) = \trace T_{r_1, r_2}^{+, \downarrow}(v) + \trace T_{p_1-r_1, r_2}^{-, \downarrow}(v),
\
\upsilon^{\swarrow}_{r_1, r_2}(v) = \trace T_{r_1, p_2-r_2}^{-, \downarrow}(v) 
+ \trace T_{p_1-r_1, p_2-r_2}^{+, \downarrow}(v),
\notag\\
&\upsilon^{\nwarrow}_{r_1, r_2} (v) = \trace B_{r_1, r_2}^{+, \uparrow}(v) + \trace B_{r_1, p_2-r_2}^{-, \uparrow}(v),
\
\upsilon^{\searrow}_{r_1, r_2}(v) = \trace B_{p_1-r_1, r_2}^{-, \uparrow}(v) + \trace B_{p_1-r_1, p_2-r_2}^{+, \uparrow}(v),
\notag\\
&\chi_{r_1, r_2}(v) = \trace B_{r_1, r_2}^{+, \downarrow}(v) + \trace B_{r_1, p_2-r_2}^{-, \downarrow}(v) + 
\trace B_{p_1-r_1, r_2}^{-, \downarrow}(v) + \trace B_{p_1-r_1, p_2-r_2} ^{+, \downarrow}(v) \notag  
\end{align}
for $v \in Q(r_1, r_2)$.
\begin{thm}\label{thm:7.1.1}
\begin{enumerate}
\item
Each of the linear function $\tau_{p_1, p_2}$ and $\tau_{0, p_2}$ 
is a basis of the space of symmetric linear functions on $Q(p_1, p_2)$
and $Q(0, p_2)$, respectively.
\item
The linear functions
\begin{align}
\tau_{r_1, p_2}^{\pm}, \ \chi_{r_1, p_2} \notag
\end{align}
form a basis of the space of symmetric linear functions on $Q(r_1, p_2)$.
\item
The linear functions
\begin{align}
\tau_{p_1, r_2}^{\pm}, \ \chi_{p_1, r_2} \notag
\end{align}
form a basis of the space of symmetric linear functions on $Q(p_1, r_2)$.
\item
The linear functions
\begin{align}
&\tau_{r_1, r_2}^{\bullet}, \ \upsilon_{r_1, r_2}^{\circ}, \ \chi_{r_1, r_2}, \ \bullet \in \{\uparrow, \downarrow, \leftarrow,
\rightarrow\}, \ \circ \in \{\nearrow, \nwarrow, \searrow, \swarrow\}\notag
\end{align}
form a basis of the space of symmetric linear functions on $Q(r_1, r_2)$ for $i \in I$.
\end{enumerate}
\end{thm}
\begin{proof}
We only prove (4).
By Theorem \ref{thm:5.3.1}, $\tau_{r_1, r_2}^{\bullet}$ with $\bullet \in \{\uparrow, \downarrow, \leftarrow, \rightarrow\}$
is symmetric.
For $v, w \in Q(r_1, r_2)$, we can see
\begin{align}
T_{r_1, r_2}^{+, \downarrow}(v w) &= T_{r_1, r_2}^{+, \downarrow}(v) T_{r_1, r_2}^{+, \uparrow}(w)
+ T_{p_1-r_1, r_2}^{-, \rightarrow}(v) T_{r_1, r_2}^{+, \leftarrow}(w) \notag\\
&+ T_{p_1-r_1, r_2}^{-, \leftarrow}(v)
T_{r_1, r_2}^{+, \rightarrow}(w) + T_{r_1, r_2}^{+, \uparrow} (v) T_{r_1, r_2}^{+, \uparrow}(w), \notag\\
T_{p_1-r_1, r_2}^{-, \downarrow}(v w) &=
T_{p_1-r_1, r_2}^{-, \downarrow}(v) T_{p_1-r_1}^{-, \uparrow}(w) 
+T_{r_1, r_2}^{+, \rightarrow}(v) T_{p_1-r_1, r_2}^{-, \leftarrow}(w)\notag\\
&+ T_{r_1, r_2}^{+, \leftarrow}(v) T_{p_1-r_1, r_2}^{-, \rightarrow}(w)
+ T_{p_1-r_1, r_2}^{-, \uparrow}(v) T_{p_1-r_1, r_2}^{-, \downarrow}(w).\notag
\end{align}
Thus $\upsilon_{r_1, r_2}^{\nearrow}$ is symmetric.
Similarly we can see that $\upsilon_{r_1, r_2}^{\swarrow}$ is symmetric.
We can also see that 
\begin{align}
B_{r_1, r_2}^{+, \uparrow}(v w) & = 
B_{r_1, r_2}^{+, \uparrow} (v) T_{r_1, r_2}^{+, \uparrow}(w)
+R_{r_1, p_2-r_2}^{-, \uparrow}(v) L_{r_1, r_2}^{+, \uparrow}(w)\notag\\
&+L_{r_1, p_2-r_2}^{-, \uparrow}(v) R_{r_1, r_2}^{+, \uparrow}(w)
+ T_{r_1, r_2}^{+, \uparrow}(v) B_{r_1, r_2}^{+, \uparrow}(w), \notag\\
B_{r_1, p_2-r_2}^{-, \uparrow}(v w)
&=B_{r_1, p_2-r_2}^{-, \uparrow}(v) T_{r_1, p_2-r_2}^{-, \uparrow}(w)
+R_{r_1, r_2}^{+, \uparrow}(v) L_{r_1, p_2-r_2}^{-, \uparrow}(w)\notag\\
&+L_{r_1, r_2}^{+, \uparrow}(v) R_{r_1, p_2-r_2}^{-, \uparrow}(w)
+T_{r_1, p_2-r_2}^{-, \uparrow}(v) B_{r_1, p_2-r_2}^{-, \uparrow}(w), \notag
\end{align}
which shows $\upsilon_{r_1, r_2}^{\nwarrow}$ is symmetric.
Similarly we can see that
$\upsilon_{r_1, r_2}^{\searrow}$ is symmetric.

Now we have
\begin{align}
B_{r_1, r_2}^{+, \downarrow}(v w) &=
B_{r_1, r_2}^{+, \downarrow}(v) T_{r_1, r_2}^{+, \uparrow}(w)
+B_{p_1-r_1, r_2}^{-, \rightarrow}(v) T_{r_1, r_2}^{+, \leftarrow}(w)\notag\\
&+B_{p_1-r_1, r_2}^{-, \leftarrow}(v) T_{r_1, r_2}^{+, \rightarrow}(w)
+B_{r_1, r_2}^{+, \uparrow}(v) T_{r_1, r_2}^{+, \downarrow}(w)\notag\\
&+R_{r_1, p_2-r_2}^{-, \downarrow} (v) L_{r_1, r_2}^{+, \uparrow} (w)
+ R_{p_1-r_1, p_2-r_2}^{+, \rightarrow} (v) L_{r_1, r_2}^{+, \leftarrow} (w)\notag\\
&+ R_{p_1-r_1, p_2-r_2}^{+, \leftarrow} (v) L_{r_1, r_2}^{+, \rightarrow}(w)
+ R_{r_1, p_2-r_2}^{-, \uparrow}(v) L_{r_1, r_2}^{+, \uparrow}(w) \notag\\
&+ L_{r_1, p_2-r_2}^{-, \downarrow}(v) R_{r_1, r_2}^{+, \uparrow} (w) 
+ L_{p_1-r_1, p_2-r_2}^{+, \rightarrow}(v) R_{r_1, r_2}^{+, \leftarrow} (w)\notag\\
&+L_{p_1-r_1, p_2-r_2}^{+, \leftarrow} (v) R_{r_1, r_2}^{+, \rightarrow}(w)
+ L_{r_1, p_2-r_2}^{-, \uparrow} (v) R_{r_1, r_2}^{+, \downarrow} (w)\notag\\
&+T_{r_1, r_2}^{+, \downarrow}(v) B_{r_1, r_2}^{+, \uparrow}(w)
+T_{p_1-r_1, r_2}^{-, \rightarrow} (v) B_{r_1, r_2}^{+, \leftarrow}(w) \notag\\
&+ T_{p_1-r_1, r_2}^{-, \leftarrow} (v) B_{r_1, r_2}^{+, \rightarrow} (w)
+ T_{r_1, r_2}^{+, \uparrow} (v) B_{r_1, r_2}^{+, \downarrow} (w).\notag
\end{align}
Substituting $(+, r_1, r_2)$ with $(-, p_1-r_1, r_2)$, $(-, r_1, p_2-r_2)$ and $(+, p_1-r_1, p_2-r_2)$,
we can describe $B_{p_1-r_1, r_2}^{-, \downarrow}(vw)$, $B_{r_1, p_2-r_2}^{-, \downarrow}(vw)$ 
and $B_{p_1-r_1, p_2-r_2}^{+, \downarrow}(vw)$ as linear combinations of traces
matrix blocks in the matrix realization of $Q(r_1, r_2)$ for $(r_1, r_2) \in 
 I$.
Since the trace is symmetric, we can see that $\chi_{r_1, r_2}$ is symmetric.
\end{proof}
\subsection{Integrals and symmetric linear functions}
Recall that the left and right integrals on $\boldsymbol{\mathfrak{g}_{p_1, p_2}}$ is given by
\begin{align}
&\lambda (e_1^{m_1} e_2^{m_2} f_1^{n_1} f_2^{n_2} K^{\ell}) 
=\delta_{m_1, p_1-1} \delta_{m_2, p_2-1} \delta_{n_1, p_1-1} \delta_{n_2, p_2-1} \delta_{\ell, p_1-p_2}, \notag\\
&\mu (e_1^{m_1} e_2^{m_2} f_1^{n_1} f_2^{n_2} K^{\ell}) 
=\delta_{m_1, p_1-1} \delta_{m_2, p_2-1} \delta_{n_1, p_1-1} \delta_{n_2, p_2-1} \delta_{\ell, p_1+p_2}, \notag
\end{align}
for $0 \le m_i \le p_i-1$ and $0 \le \ell \le 2p_1 p_2-1$.
By Proposition \ref{prop:center-yaho} and Proposition \ref{prop:inner},
we can see 
\begin{align}
g^{-1} \rightharpoonup \lambda (e_1^{m_1} e_2^{m_2} f_1^{n_1} f_2^{n_2} K^{\ell}) 
&= \mu \leftharpoonup g (e_1^{m_1} e_2^{m_2} f_1^{n_1} f_2^{n_2} K^{\ell}) \notag\\
&= \delta_{m_1, p_1-1} \delta_{m_2, p_2-1} \delta_{n_1, p_1-1} \delta_{n_2, p_2-1} \delta_{\ell, 0}. \notag
\end{align}
The following is the generalization of Lemma \ref{lem:com} (see \cite{L}).
\begin{lem}\label{lem:7.7}
For $1 \le m_i, n_i \le p_i-1$ with $i = 1, 2$, 
\begin{align}
[e_i^{m_i}, f_i^{n_i}] = \sum_{j = 1}^{\min (m_i, n_i)} e_i^{m_i-j} f_i^{n_i-j} F_{j}^{m_i, n_i}(K), \notag
\end{align}
where $F_{j}^{m_i, n_i} (z) \in \mathbb{C}[z, z^{-1}]$.
\end{lem}
By Lemma \ref{lem:7.7}, 
the term $e_{1}^{p_1-1} e_2^{p_2-1}f_1^{p_1-1}f_2^{p_2-1}$ appears in 
$B_{s_1-1, s_2-1}^{\pm, \downarrow}(p_1, p_2, s_1, s_2)$ only for $Q(p_1, p_2)$ and $Q(0, p_2)$.
Thus we can see the following.
\begin{prop}
\begin{align}
&g^{-1} e(p_1, p_2) \rightharpoonup \lambda = \frac{1}{\Phi^+(r_1, r_2)} \tau_{p_1, p_2}, \notag\\
&g^{-1} e(0, p_2) \rightharpoonup \lambda = \frac{1}{\Phi^-(p_1, p_2)} \tau_{0, p_2}.\notag
\end{align}
\end{prop}
For $Q(r_1, p_2)$, the term $e_{1}^{p_1-1} e_2^{p_2-1}f_1^{p_1-1}f_2^{p_2-1}$ appears in
\begin{align}
B_{s_1-1, s_2-1}^{+, \uparrow \downarrow}(r_1, p_2, s_1, s_2) \ \text{and} \ B_{t_1-1, s_2-1}^{-, \uparrow \downarrow}
(p_1, r_2, s_1, s_2) \notag
\end{align}
by \eqref{eq:3.2.2}, \eqref{eq:3.2.17} and Lemma \ref{lem:7.7}.
\begin{prop}
\begin{enumerate}
\item
For $1 \le r_1 \le p_1-1$, we have
\begin{align}
&g^{-1} e(r_1, p_2) \rightharpoonup \lambda = \frac{1}{\Phi^+(r_1, p_2)}
\left\{\chi_{r_1, p_2} - \Psi_1^+(r_1, p_2)\left( \tau_{r_1, p_2}^+ + \tau_{r_1, p_2}^-\right)
\right\}, \notag\\
&g^{-1} w(r_1, p_2)^{\pm} \rightharpoonup \lambda = \frac{1}{\Phi^+(r_1, p_2)}
\tau_{r_1, p_2}^{\pm}.\notag
\end{align}
\item
For $1 \le r_2 \le p_2-1$, we have
\begin{align}
&g^{-1} e(p_1, r_2) \rightharpoonup \lambda = \frac{1}{\Phi^+(p_1, r_2)}
\left\{\chi_{p_1, r_2} - \Psi_1^+(p_1, r_2)\left( \tau_{p_1, r_2}^+ + \tau_{p_1, r_2}^-\right)
\right\}, \notag\\
&g^{-1} w(p_1, r_2)^{\pm} \rightharpoonup \lambda = \frac{1}{\Phi^+(p_1, r_2)}
\tau_{p_1, r_2}^{\pm}.\notag
\end{align}
\end{enumerate}
\end{prop}
For $Q(r_1, r_2)$ with $(r_1, r_2) \in I$,
we can see that the term $e_{1}^{p_1-1} e_2^{p_2-1}f_1^{p_1-1}f_2^{p_2-1}$
appears in 
\begin{align}
&T_{s_1-1, s_2-1}^{+, \uparrow \downarrow}(r_1, r_2, s_1, s_2), B_{s_1-1, s_2-1}^{+, \uparrow \downarrow}(r_1, r_2, s_1, s_2),
\notag\\
&T_{t_1-1, s_2-1}^{-, \uparrow \downarrow}(p_1-r_1, r_2, t_1, s_2), 
B_{t_1-1, s_2-1}^{-, \uparrow \downarrow}(p_1-r_1, r_2, t_1, s_2), \notag\\
&T_{s_1-1, t_2-1}^{-, \uparrow \downarrow}(r_1, p_2-r_2, s_1, t_2), 
B_{s_1-1, s_2-1}^{+, \uparrow \downarrow}(r_1, p_2-r_2, s_1, t_2), \notag\\
&T_{t_1-1, t_2-1}^{+, \uparrow \downarrow}(p_1-r_1, p_2-r_2, t_1, t_2), 
B_{t_1-1, t_2-1}^{+, \uparrow \downarrow}(p_1-r_1, p_2-r_2, t_1, t_2). \notag
\end{align}
by
\eqref{eq:3.3.2}, \eqref{eq:3.3.8},  \eqref{eq:3.3.10} and Lemma \ref{lem:7.7}.
Thus we have the following.
\begin{prop}
For $(r_1, r_2) \in I$, we have
\begin{align}
&g^{-1} e(r_1, r_2) \rightharpoonup \lambda = \frac{1}{\Phi^{+}(r_1, r_2)}\Bigl\{
\Psi_1^{+}(r_1, r_2) \Psi_2^+(r_1, r_2) \sum_{\bullet \in \{\uparrow, \downarrow, \leftarrow, \rightarrow\}}
\tau_{r_1, r_2}^{\bullet}\notag\\
&- \Psi_2^+(r_1, r_2) \sum_{\bullet \in \{\nearrow, \swarrow \}} \upsilon_{r_1, r_2}^{\bullet}
-\Psi_1^+(r_1, r_2) \sum_{\bullet \in \{\nwarrow, \searrow\}} \upsilon_{r_1, r_2}^{\bullet}
+ \chi_{r_1, r_2} \Bigr\}, \notag\\
&g^{-1} v^{\nearrow}(r_1, r_2) \rightharpoonup \lambda
=\frac{1}{\Phi^+(r_1, r_2)}\left\{
\upsilon_{r_1, r_2}^{\nearrow} - \Psi_1^+(r_1, r_2) \left(
\tau_{r_1, r_2}^{\uparrow} + \tau_{r_1, r_2}^{\rightarrow} \right)
\right\}, \notag\\
&g^{-1} v^{\swarrow}(r_1, r_2) \rightharpoonup \lambda
=\frac{1}{\Phi^+(r_1, r_2)}\left\{
\upsilon_{r_1, r_2}^{\swarrow} - \Psi_1^+(r_1, r_2) \left(
\tau_{r_1, r_2}^{\downarrow} + \tau_{r_1, r_2}^{\leftarrow} \right)
\right\}, \notag\\
&g^{-1} v^{\nwarrow}(r_1, r_2) \rightharpoonup \lambda
=\frac{1}{\Phi^+(r_1, r_2)}\left\{
\upsilon_{r_1, r_2}^{\nwarrow} - \Psi_2^+(r_1, r_2) \left(
\tau_{r_1, r_2}^{\uparrow} + \tau_{r_1, r_2}^{\leftarrow} \right)
\right\}, \notag\\
&g^{-1} v^{\searrow}(r_1, r_2) \rightharpoonup \lambda
=\frac{1}{\Phi^+(r_1, r_2)}\left\{
\upsilon_{r_1, r_2}^{\searrow} - \Psi_2^+(r_1, r_2) \left(
\tau_{r_1, r_2}^{\downarrow} + \tau_{r_1, r_2}^{\rightarrow} \right)
\right\}, \notag\\
&g^{-1} w^{\bullet}(r_1, r_2) \rightharpoonup \lambda
=\frac{1}{\Phi^+(r_1, r_2)} \tau_{r_1, r_2}^{\bullet}, \ \bullet \in \{\uparrow, \downarrow, \leftarrow, \rightarrow\}.
\notag
\end{align}
\end{prop}

\appendix
\section{$q$-characters and symmetric linear functions}
Let $A$ be a finite-dimensional Hopf algebra.
Then the space of $q$-characters $\ch (A)$ of $A$ is defined by
\begin{align}
\ch (A) & = \{\beta \in A^{\ast} \ |  \ \ad_x (\beta) = \varepsilon (x) \beta \ 
\text{for all} \ x \in A \} \notag\\
&= \{\beta \in A^{\ast} \ | \ \beta (xy) = \beta (S^2(y) x) \
\text{for all} \ x, y \in A\} \notag
\end{align}
where $\ad_a (\beta) =  \beta (\sum S(a_1) ? a_2)$.

Let $A = \gpq$.
Then $g = K^{p_1 - p_2}$
is a group-like element and for any $x \in \gpq$ 
we have $S^{2} (x) = g x g^{-1}$.
It is shown in \cite{FGST4} that the linear function
\begin{align}
\gamma^{\pm} (r_1, r_2) : \gpq^{\ast} \to \mathbb{C}, \
x \mapsto \trace_{X_{r_1, r_2}^{\pm}} (g^{-1} x). \notag
\end{align}
for $1 \le r_1 \le p_1$ and $1 \le r_2 \le p_2$
belongs to $\ch (\gpq)$.

Note that the map $\theta : \ch (\gpq) \to \SLF(\gpq)$
defined by $\beta \mapsto \beta \leftharpoonup g$
is an isomorphism.

Let $\mathbb{P}_{r_1, r_2}$ be the direct sum of four indecomposable
projective modules:
\begin{align}
\mathbb{P}_{r_1, r_2} = \begin{cases}
P_{r_1, p_2}^{+} \oplus P_{p_1-r_1, p_2}^-, & 1 \le r_1 \le p_1-1, \ r_2 = p_2, \\
P_{p_1, r_2}^+ \oplus P_{p_1, p_2- r_2}^-, & 1 \le r_2 \le p_2-1, \ r_1 = p_1, \\
P_{r_1, r_2}^+ \oplus P_{p_1-r_1, r_2}^- \oplus P_{r_1, p_2-r_2}^- \oplus
P_{p_1-r_1, p_2-r_2}^+, & (r_1, r_2) \in I.
\end{cases}
\end{align}

\subsection{$q$-characters and symmetric linear functions}
Define the linear maps $\sigma_{p_1, r_2} : \mathbb{P}_{p_1, r_2} \to \mathbb{P}_{p_1, r_2}$ and
$\sigma_{r_1, p_2} : \mathbb{P}_{r_1, p_2} \to \mathbb{P}_{r_1, p_2}$:
the maps $\sigma_{p_1, r_2}$ and $\sigma_{r_1, p_2}$ act as zero except 
\begin{align}
&\sigma_{p_1} : \mathsf{b}_{n_1, n_2}^{\downarrow, \bullet} \mapsto 
\alpha^{\uparrow, \bullet} \mathsf{b}_{n_1, n_2}^{\uparrow, \bullet} 
+ \alpha^{\downarrow, \bullet} \mathsf{b}_{n_1, n_2}^{\downarrow, \bullet}, \notag\\
&\sigma_{p_2} : \mathsf{b}_{n_1, n_2}^{\downarrow, \circ} \mapsto 
\beta^{\uparrow, \circ} \mathsf{b}_{n_1, n_2}^{\uparrow, \circ} 
+ \beta^{\downarrow, \circ} \mathsf{b}_{n_1, n_2}^{\downarrow, \circ}. \notag
\end{align}
where $\bullet \in \{\uparrow, \rightarrow \}$ and $\circ \in \{\downarrow, \leftarrow\}$.
In the above equations, $\mathsf{b}_{n_1, n_2}^{\downarrow, \uparrow}$
corresponds to the basis of $P_{p_1, r_1}^+$,
$\mathsf{b}_{n_1, n_2}^{\downarrow, \rightarrow}$ corresponds to the basis of $P_{p_1, p_2-r_2}^-$.

We also define the linear map $\sigma_{r_1, r_2}$ with $(r_1, r_2) \in I$:
the map $\sigma_{r_1, r_2}$ acts as zero except
\begin{align}
\mathsf{b}_{n_1, n_2}^{\downarrow, \bullet} \mapsto 
\alpha^{\uparrow, \bullet} \mathsf{b}_{n_1, n_2}^{\uparrow, \bullet} 
+ \alpha^{\downarrow, \bullet} \mathsf{b}_{n_1, n_2}^{\downarrow, \bullet}
+ \beta^{\uparrow, \bullet} \mathsf{t}_{n_1, n_2}^{\uparrow, \bullet}
+ \beta^{\downarrow, \bullet} \mathsf{t}_{n_1, n_2}^{\downarrow, \bullet} \notag
\end{align}
where $\bullet \in \{\uparrow, \downarrow, \leftarrow, \rightarrow\}$.
In the above equation,
$\mathsf{b}_{n_1, n_2}^{\downarrow, \bullet}$ for $\bullet = \uparrow$ corresponds 
to the basis of $P_{r_1, r_2}^+$,
for $\bullet = \leftarrow$ corresponds to the basis of $P_{r_1, p_2-r_2}^-$,
for $\bullet = \rightarrow$ corresponds to the basis of $P_{p_1-r_1, r_2}^-$
and for $\bullet  = \downarrow$ corresponds to the basis of $P_{p_1-r_1, p_2-r_2}^-$.

Then we define $\gamma(r_1, r_2) : x \mapsto \trace_{\mathbb{P}_{r_1, r_2}} 
(g^{-1} x \sigma_{r_1, r_2})$.
In \cite{FGST4}, the following is shown.
\begin{prop}[\cite{FGST4}, Proposition 2.3]
The map $\gamma(r_1, r_2)$ belongs to $\ch(\gpq)$
if and only if
\begin{align}
&\alpha^{\uparrow, \uparrow} = \alpha^{\uparrow, \rightarrow}, \ 
\alpha^{\uparrow, \downarrow} = \alpha^{\uparrow, \leftarrow}, \notag\\
&\beta^{\downarrow, \uparrow} = \beta^{\downarrow, \leftarrow}, \
\beta^{\downarrow, \downarrow} = \beta^{\downarrow, \rightarrow}, \notag\\
&\beta^{\uparrow, \uparrow} = \beta^{\uparrow, \leftarrow}
= \beta^{\uparrow, \rightarrow} = \beta^{\uparrow, \downarrow}. \notag
\end{align}
\end{prop}

\begin{thm}[\cite{FGST4}]
The following form a basis of $\ch(\gpq)$\,{\rm :}
\begin{align}
&\gamma^{\pm}(r_1, r_2),  \ 1 \le r_1 \le p_1, \ 1 \le r_2 \le p_2,  \notag\\
&\gamma^{\nearrow} (r_1, r_2), \ \alpha^{\uparrow, \downarrow} 
= \alpha^{\downarrow, \bullet} = \beta^{\downarrow, \bullet}
= \beta^{\uparrow, \bullet} = 0, \ (r_1, r_2) \in I \cup \{(r_1, p_2) \ | 1 \le 
 r_1 \le p_1 - 1 \}, \notag\\
&\gamma^{\swarrow} (r_1, r_2), \ 
\alpha^{\uparrow, \uparrow} = \alpha^{\downarrow, \bullet}
= \beta^{\downarrow, \bullet} = \beta^{\uparrow, \bullet} = 0, 
\ (r_1, r_2) \in I,\notag\\
&\gamma^{\nwarrow} (r_1, r_2), \ \alpha^{\downarrow, \bullet}
= \alpha^{\uparrow, \bullet} = \beta^{\downarrow, \downarrow}
= \beta^{\uparrow, \bullet} = 0, 
\ (r_1, r_2) \in I,\notag\\
&\gamma^{\searrow}(r_1, r_2),
\alpha^{\downarrow, \bullet}
= \alpha^{\uparrow, \bullet} = \beta^{\downarrow, \uparrow}
= \beta^{\uparrow, \bullet} = 0, 
\ (r_1, r_2) \in I \cup \{(p_1, r_2) \ | 1 \le 
 r_2 \le p_2 - 1 \}, \notag\\
&\gamma^{\downarrow, \downarrow} (r_1, r_2), \
\alpha^{\downarrow, \bullet} = \alpha^{\uparrow, \bullet}
= \beta^{\downarrow, \bullet} = 0, \ (r_1, r_2) \in I.  \notag
\end{align}
\end{thm}

Then by this isomorphism and Table 1, we have
\begin{thm}
We have
\begin{align}
&\theta (\gamma^+(p_1, p_2)) = \tau_{p_1, p_2}, \
\theta (\gamma^-(p_1, p_2)) = \tau_{0, r_2}, \notag\\
&\theta (\gamma^+(r_1, p_2)) = \tau^+_{r_1, p_2},  \ 
\theta (\gamma^-(p_1 - r_1, p_2) = \tau^-_{r_1, p_2}, \
\theta (\gamma^{\nearrow}(r_1, p_2)) =  \chi_{r_1, p_2}, 1 \le r_1 \le p_1-1, \notag\\
&\theta(\gamma^+(p_1, r_2)) = \tau^+(p_1, r_2), \ 
\theta (\gamma^-(p_1, p_2-r_2)) = \tau^-(p_1, r_2), \notag\\
&\theta (\gamma^{\searrow} (p_1, r_2)) = \chi_{p_1, r_2}, \ 1 \le r_2 \le p_2-1, \notag\\
&\theta (\gamma^+(r_1, r_2)) = \tau_{r_1, r_2}^{\uparrow}, \ 
\theta (\gamma^-(r_1, p_2-r_2)) = \tau_{r_1, p_2-r_2}^{\leftarrow}, \notag\\
&\theta (\gamma^-(p_1-r_1, r_2)) = \tau_{p_1 - r_1, r_2}^{\rightarrow}, \
\theta (\gamma^+(p_1-r_1, p_2-r_2)) = \tau_{p_1-r_1, p_2-r_2}^{\downarrow}, \notag\\
&\theta (\gamma^{\nearrow} (r_1, r_2)) = \upsilon_{r_1, r_2}^{\nearrow}, \
\theta (\gamma^{\swarrow} (r_1, r_2)) = \upsilon_{r_1, r_2}^{\swarrow}, \notag\\
&\theta (\gamma^{\nwarrow} (r_1, r_2)) =  \upsilon_{r_1, r_2}^{\nwarrow}, \
\theta (\gamma^{\searrow}(r_1, r_2)) = \upsilon_{r_1, r_2}^{\searrow}, \notag\\
 &\theta (\gamma^{\downarrow, \downarrow} (r_1, r_2)) = \chi_{r_1, r_2}, \ (r_1, 
 r_2) \in I. \notag
\end{align}
\end{thm}

\end{document}